\documentclass[12pt, a4paper]{amsart}
\usepackage{amscd,amsmath,amssymb,amsthm,amsfonts}
\usepackage[alphabetic]{amsrefs}
\usepackage{mathrsfs}
\usepackage[shortlabels]{enumitem}
\usepackage[all]{xy}
\usepackage{xcolor}
\usepackage[hypertexnames=true, citecolor = green, colorlinks = false, linkcolor = red, linktoc = none, pdffitwindow = false, urlbordercolor = white]{hyperref}%
\usepackage{thmtools,cleveref}

\def\Aut{\operatorname{Aut}}

\def\ker{\operatorname{ker}}

\def\id{\operatorname{id}}

\def\max{\operatorname{max}}

\def\id{\operatorname{id}}

\def\kms{\operatorname{KMS}}
\def\Irr{\operatorname{Irr}}

\def\C{\mathbb{C}}

\def\R{\mathbb{R}}
\def\N{\mathbb{N}}
\def\Z{\mathbb{Z}}

 \newcommand{\IF}[0]{\mathbb{F}}

 \newcommand{\IN}[0]{\mathbb{N}}

 \newcommand{\IT}[0]{\mathbb{T}}

 \newcommand{\CH}[0]{\mathcal{H}}
 
 \newcommand{\CL}[0]{\mathcal{L}}
 
 \newcommand{\CP}[0]{\mathcal{P}}
\newcommand{\CQ}[0]{\mathcal{Q}} 
 \newcommand{\CT}[0]{\mathcal{T}}

\def\gxp{G \rtimes_\theta P}
\def\gpt{G,P,\theta}

\newtheorem{thm}{Theorem}[section]
\newtheorem{corollary}[thm]{Corollary}
\newtheorem{lemma}[thm]{Lemma}

\newtheorem{proposition}[thm]{Proposition}

\theoremstyle{definition}
\newtheorem{definition}[thm]{Definition}

\theoremstyle{remark}
\newtheorem{remark}[thm]{Remark}

\newtheorem{question}[thm]{Question}
\newtheorem{example}[thm]{Example}

\newtheorem{questions}[thm]{Questions}

\numberwithin{equation}{section}

\begin{document}

\title{Equilibrium states on right LCM semigroup $C^*$-algebras}

\author[Z.~Afsar]{Zahra Afsar}
\address{School of Mathematics and Applied Statistics \\ University of Wollongong \\ Australia}
\email{za.afsar@gmail.com}

\author[N.~Brownlowe]{Nathan Brownlowe}
\address{School of Mathematics and Statistics \\ University of Sydney\\ Australia}
\email{nathan.brownlowe@gmail.com}

\author[N.S.~Larsen]{Nadia S.~Larsen}
\address{Department of Mathematics \\ University of Oslo \\ P.O. Box 1053 \\ Blindern \\ NO-0316 Oslo \\ Norway}
\email{nadiasl@math.uio.no, nicolsta@math.uio.no}

\author[N.~Stammeier]{Nicolai Stammeier}

\begin{abstract}
We determine the structure of equilibrium states for a natural dynamics on the boundary quotient diagram of $C^*$-algebras for a large class of right LCM semigroups. The approach is based on abstract properties of the semigroup and covers the previous case studies on $\N \rtimes \N^\times$, dilation matrices, self-similar actions, and Baumslag--Solitar monoids. At the same time, it provides new results for right LCM semigroups associated to algebraic dynamical systems.
\end{abstract}

\date{4 February 2017}
\maketitle

\section{Introduction}\label{sec:intro}

Equilibrium states have been studied in operator algebras starting with quantum systems modelling ensembles of particles, and have been the means of describing properties of physical models with $C^*$-algebraic tools, cf. \cite{BRII}.  Specifically, given a quantum  statistical mechanical system, which is a pair formed of a $C^*$-algebra $A$ that encodes the observables and a one-parameter group of automorphisms $\sigma$ of $A$  interpreted  as a time evolution,  one seeks to express equilibrium of the system via states with specific properties. A $\kms_\beta$-state for $(A,\sigma)$ is a state on $A$ that satisfies the $\kms_\beta$-condition (for Kubo-Martin-Schwinger) at a real parameter $\beta$ bearing the significance of an inverse temperature. Gradually, it has become apparent that the study of KMS-states for systems  that do not necessarily have physical origins brings valuable insight into the structure of the underlying $C^*$-algebra, and uncovers new directions of interplay between the theory of $C^*$-algebras and other fields of mathematics. A rich supply of examples is by now present in the literature, see \cites{BC,LR2,LRRW,CDL,Nes}, to mention only some.

This article is concerned with the study of KMS-states for systems whose underlying $C^*$-algebras are Toeplitz-type algebras modelling a large class of semigroups. While our initial motivation was to classify KMS-states for a specific class of examples, as had been done previously, we felt that by building upon the deep insight developed in the case-studies already present in the literature, the time was ripe for proposing a general framework that would encompass our motivating example and cover all these case-studies.  We believe it is a strength of our approach that we can identify simply phrased conditions at the level of the semigroup which govern the structure of $\kms$-states, including uniqueness of the $\kms_\beta$-state for $\beta$ in the critical interval. In all the examples we treat, these conditions admit natural interpretations. Moreover, they are often subject to feasible verification.

To be more explicit, the interest in studying KMS-states on Toeplitz-type $C^*$-algebras gained new momentum with the work of Laca and Raeburn \cite{LR2} on the Toeplitz algebra of the $ax+b$-semigroup over the natural numbers $\N\rtimes \N^\times$. Further results on the KMS-state structure of Toeplitz-type $C^*$-algebras include the work on Exel crossed products associated to dilation matrices \cite{LRR}, $C^*$-algebras associated to self-similar actions \cite{LRRW}, and $C^*$-algebras associated to Baumslag--Solitar monoids \cite{CaHR}. In all these cases the relevant Toeplitz-type $C^*$-algebra can be viewed as a semigroup $C^*$-algebra in the sense of \cite{Li1}. Further, the semigroups in question are right LCM semigroups, see \cite{BLS1} for dilation matrices, and \cite{BRRW} for all the other cases. In \cites{BLS1,BLS2}, the last three named authors have studied $C^*$-algebras of right LCM semigroups associated to algebraic dynamical systems. These $C^*$-algebras often admit a natural dynamics that is built from the underlying algebraic dynamical system. The original motivation for our  work was to classify KMS-states in this context.

Although there is a common thread in the methods and techniques used to prove the KMS-classification results in \cites{LR2, LRR, LRRW, CaHR}, one cannot speak of a single proof that runs with adaptations. Indeed, in each of these papers a careful and  rather involved analysis is carried out using concrete properties of the respective setup. One main achievement of this paper is a general theory of KMS-classification which both unifies the classification results from \cites{LR2, LRR, LRRW, CaHR}, and also considerably enlarges the class of semigroup $C^*$-algebras for which KMS-classification is new and interesting, see Theorem~\ref{thm:KMS results-gen} and Section~\ref{sec:examples}.

We expect that our work may build a bridge to the analysis of KMS-states via groupoid models and the result of Neshveyev \cite{Nes}*{Theorem~1.3}. At this point we would also like to mention  the recent analysis of equilibrium states on $C^*$-algebras associated to self-similar actions of groupoids on graphs that generalises \cite{LRRW}, see \cite{LRRW1}. It would be interesting to explore connections to this line of development.

One of the keys to establishing our general theory is the insight that it pays off to work with the boundary quotient diagram for right LCM semigroups proposed by the fourth-named author in \cite{Sta3}. The motivating example for this diagram was first considered in \cite{BaHLR}, where it was shown to give extra insight into the structure of $\kms$-states on $\CT(\N\rtimes\N^\times)$. Especially the \emph{core subsemigroup} $S_c$, first introduced in \cite{Star}, and motivated by the quasi-lattice ordered situation in \cite{CrispLaca}, and the \emph{core irreducible elements} $S_{ci}$ introduced in \cite{Sta3} turn out to be central for our work. Recall that $S_c$ consists of all the elements whose principal right ideal intersects all other principal right ideals and $S_{ci}$ of all $s\in S\setminus S_c$ such that every factorisation $s=ta$ with $a\in S_c$ forces $a$ to be invertible in $S$.
Via so-called (accurate) foundation sets, $S_c$ and $S_{ci}$ give rise to quotients $\CQ_c(S)$ and $\CQ_p(S)$ of $C^*(S)$ of $S$, respectively. Together with $C^*(S)$ and the boundary quotient $\CQ(S)$ from \cite{BRRW}, they form a commutative diagram:
\begin{equation}\label{eq:BQD for thm}
\begin{gathered}
\xymatrix@=17mm{
C^*(S) \ar^{\pi_p}@{->>}[r]\ar_{\pi_c}@{->>}[d] & \CQ_p(S) \ar^{\pi'_c}@{->>}[d] \\
\CQ_c(S) \ar^{\pi'_p}@{->>}[r] & \CQ(S)
}
\end{gathered}
\end{equation}
We refer to \cite{Sta3} for examples and further details, but remark that if $S=S_{ci}^1S_c$ holds (with $S_{ci}^1 = S_{ci} \cup \{1\}$), a property we refer to as $S$ being \emph{core factorable}, then the quotient maps $\pi'_p$ and $\pi'_c$ are induced by the corresponding relations for $\pi_p$ and $\pi_c$, respectively, see \cite{Sta3}*{Proposition~2.10}. Another relevant condition in this context is the property that any right LCM in $S$ of a pair in $S_{ci}$ belongs to $S_{ci}$ again. If this holds, then we say that $S_{ci}\subset S$ is \emph{$\cap$-closed}.

In this paper we work with what we call \emph{admissible} right LCM semigroups, which are core factorable right LCM semgroups with $S_{ci}\subset S$ being $\cap$-closed that admit a suitable homomorphism $N\colon S\to \N^\times$. We refer to such a map $N$ as a \emph{generalised scale} to stress the analogy to the \emph{scale} appearing in \cite{Laca98}. The dynamics $\sigma$ of $\R$ on $C^*(S)$ that we consider for a generalised scale $N$ is natural in the sense that $\sigma_x(v_s) = N_s^{ix} v_s$ on the canonical generators of $C^*(S)$. In each of the examples we consider in Section~\ref{sec:examples} there is a natural choice for a generalised scale. Every generalised scale on a right LCM semigroup $S$ gives rise to a $\zeta$-function on $\R$. The \emph{critical inverse temperature} $\beta_c$ for the dynamics $\sigma$ is defined to be the least bound in $\R \cup \{\infty\}$ so that $\zeta_S(\beta)$ is finite for all $\beta$ above $\beta_c$. The uniqueness of $\kms_\beta$-states for $\beta$ inside the \emph{critical interval} ($[1,\beta_c]$ for finite $\beta_c$, and $[1,\infty)$ otherwise) has become an intriguing topic in the field.

Our main result Theorem~\ref{thm:KMS results-gen} classifies the KMS-state structure on \eqref{eq:BQD for thm} for an admissible right LCM semigroup $S$ with regards to the dynamics $\sigma$ arising from a generalised scale $N$. There are no $\kms_\beta$-states for $\beta<1$, and for $\beta > \beta_c$ there is an affine homeomorphism between $\kms_\beta$-states on $C^*(S)$ and normalised traces on $C^*(S_c)$. For $\beta$ in the critical interval there exists at least one $\kms_\beta$-state, and two sufficient criteria for the $\kms_\beta$-state to be the unique one are established. Both require faithfulness of a natural action $\alpha$ of $S_c$ on $S/S_c$ arising from left multiplication. The first criterion demands the stronger hypothesis that $\alpha$ is almost free and works for arbitrary $\beta_c$. In contrast, the second criterion only works for $\beta_c=1$, but relies on the seemingly more modest, yet elusive condition that $S$ has what we call \emph{finite propagation}. The ground states on $C^*(S)$ are affinely homeomorphic to the states on $C^*(S_c)$. Moreover, if $\beta_c < \infty$, then a ground state is a $\kms_\infty$-state if and only if it corresponds to a trace on $C^*(S_c)$. With regards to the boundary quotient diagram, it is proven that every $\kms_\beta$-state and every $\kms_\infty$-state factors through $\pi_c\colon C^*(S)\to \CQ(S_c)$. In contrast, a $\kms_\beta$-state factors through $\pi_p\colon C^*(S)\to \CQ_p(S)$ if and only if $\beta=1$. Last but not least, there are no ground states on $\CQ_p(S)$, and hence none on the boundary quotient $\CQ(S)$.

As every result requiring abstract hypotheses ought to be tested for applicability to concrete examples, we make an extensive effort to provide a wide range of applications. The results on $\kms$-classification for self-similar actions, (subdynamics of) $\N\rtimes\N^\times$, and Baumslag-Solitar monoids from \cite{LRRW},\cite{LR2} and \cite{BaHLR}, and \cite{CaHR} are recovered in sections \ref{subsec:SSA}, \ref{subsec:NxNtimes}, and \ref{subsec:BS-monoid}, respectively; we remark that in order to make the similarities between these examples more apparent we make a different choice of time evolution leading to different critical temperatures. Our analysis is achieved by first establishing reduction results for general Zappa-Sz\'{e}p products $S=U\bowtie A$ of right LCM semigroups $U$ and $A$, see Corollary~\ref{cor:reduction from UxA to U} and Corollary~\ref{cor:reduction from UxA to U - alpha}. These in turn rely on an intriguing result concerning the behaviour of $S_c$ and $S_{ci}$ under the Zappa-Sz\'{e}p product construction $S=U\bowtie A$, which we expect to be of independent interest, see Theorem~\ref{thm:ZS-products with cap condition}. As a simple application of the reduction results, we also cover the case of right-angled Artin monoids that are  direct products of a free monoid by an abelian monoid. We then investigate admissibility for right LCM semigroups arising from algebraic dynamical systems, see \ref{subsec:ADS}. This yields plenty of examples for which no $\kms$-classification was known before, and we obtain the $\kms$-classification results from \cite{LRR} as a corollary, see Example~\ref{ex:dilation matrix}. In addition, we give an example of an admissible right LCM semigroup with finite propagation for which the action $\alpha$ is faithful but not almost free by means of standard restricted wreath products of finite groups by $\N$, see Example~\ref{ex:ADS min but not str min}.

Apart from this variety of examples, there is also a structural link to be mentioned here: Via a categorical equivalence from \cite{Law0} based on Clifford's work \cite{Cli}, right LCM semigroups correspond to $0$-bisimple inverse monoids, which constitute a well-studied and rich class of inverse semigroups, see for instance \cite{McA} and the references therein. The results of this work may therefore be seen as a contribution to the study of $0$-bisimple inverse monoids, and the relevance of the core structures exhibited in right LCM semigroups for the respective $0$-bisimple inverse monoids are yet to be investigated.

The paper is organised as follows: A brief description of basic notions and objects related to right LCM semigroups and their $C^*$-algebras, as well as to $\kms$-state of a $C^*$-algebra is provided in Section~\ref{sec:background}. In Section~\ref{sec:setting}, the class of admissible right LCM semigroups is introduced. Section~\ref{sec:mainthm} is a small section devoted to stating our main theorem and to providing an outline of the proof. Before proving our main theorem, we apply Theorem~\ref{thm:KMS results-gen} to a wide range of examples in Section~\ref{sec:examples}. Sections~\ref{sec:constraints}--\ref{sec:uniqueness} constitute the essential steps needed in the proof of Theorem~\ref{thm:KMS results-gen}: In Section~\ref{sec:constraints} algebraic characterisations for the existence of $\kms$-states are provided. These yield bounds on the kinds of $\kms$-states that can appear. In Section~\ref{sec:theReconstruction} a reconstruction formula in the spirit of \cite{LR2}*{Lemma~10.1} is provided; and in Section~\ref{sec:construction} $\kms$-states are constructed, where Theorem~\ref{thm:rep of C*(S)} provides a means to construct representations for $C^*(S)$ out of states on $C^*(S_c)$, which we believe to be of independent interest. Even more so, since it only requires $S$ to be core factorable and $S_{ci}\subset S$ to be $\cap$-closed. Section~\ref{sec:uniqueness} addresses the existence of a unique $\kms_\beta$-state for $\beta$ within the critical interval, and Section~\ref{sec:questions} is an attempt at identifying some challenges  for future research.\vspace*{2mm}\\

\emph{Acknowledgements:} This research was supported by the Australian Research Council. Parts of this work were carried out when N.L.~visited Institut Mittag-Leffler (Sweden) for the program ``Classification of operator algebras: complexity, rigidity, and dynamics", and the University of Victoria (Canada), and she thanks Marcelo Laca and the mathematics department at UVic for the hospitality extended to her. N.S.~was supported by ERC through AdG 267079, by DFG through SFB 878 and by RCN through FRIPRO 240362. This work was initiated during a stay of N.S.~with N.B.~at the University of Wollongong (Australia), and N.S.~thanks the work group for its great hospitality.

\section{Background}\label{sec:background}
\subsection{Right LCM semigroups}\label{sebsec:RLCMsemigroups}
A discrete, left cancellative semigroup $S$ is {right LCM} if the intersection of principal right ideals is empty or another principal right ideal, see \cites{Nor1, Law2, BRRW}. If $r,s,t\in S$ with $sS\cap tS= rS$, then we call $r$ a \emph{right LCM} of $s$ and $t$. In this work, all semigroups will admit an identity element (and so are monoids) and will be countable. We let $S^*$ denote the subgroup of invertible elements in $S$. Two subsets of a right LCM semigroup $S$ are used in \cite{Sta3}: the \emph{core subsemigroup}
\[
S_c:=\{ a\in S \mid aS\cap sS\neq\emptyset\text{ for all }s\in S\},
\]
which was first considered in \cite{Star} in the context of right LCM semigroups, and the set $S_{ci}$ of \emph{core irreducible elements}, where $s\in S\setminus S_c$ is core irreducible if every factorisation $s=ta$ with $a\in S_c$ forces $a \in S^*$. The latter is a subsemigroup of $S$, at least under  moderate conditions, see Proposition~\ref{prop:hom of right LCM for S_c and S_ci}. The core subsemigroup $S_c$ contains $S^*$, and induces an equivalence relation on $S$ defined by
\[
s \sim t :\Leftrightarrow sa = tb\text{ for some }a,b \in S_c.
\]
We call $\sim$ the \emph{core relation} and say that $s$ and $t$ are \emph{core equivalent} if $s \sim t$. The equivalence class of $s\in S$ under $\sim$ is denoted $[s]$. We write $s \perp t$ if $s,t \in S$ have disjoint right ideals. One can verify that if $s \sim s'$ for some $s,s' \in S$, then $s \perp t$ if and only if $s' \perp t$ for all $t \in S$. By a transversal $\CT$ for $A/_\sim$ with $A\subset S$, we mean a subset of $A$ that forms a minimal complete set of representatives for $A/_\sim$.

Recall from \cite{BRRW} that a finite subset $F$ of $S$ is called a \emph{foundation set} if for every $s \in S$ there is $f \in F$ with $sS \cap fS \neq \emptyset$. The second and fourth author coined the refined notion of \emph{accurate foundation set} in \cite{BS1}, where a foundation set $F$ is accurate if $fS\cap f'S=\emptyset$ for all $f,f'\in F$ with $f\neq f'$. A right LCM semigroup $S$ has the \emph{accurate refinement property}, or property (AR), if for every foundation set $F$ there is an accurate foundation set $F_a$ such that $F_a\subseteq FS$, in the sense that for every $f_a\in F_a$ there is $f\in F$ with $f_a\in fS$.

\subsection{\texorpdfstring{$C^*$-algebras associated to right LCM semigroups}{C*-algebras associated to right LCM semigroups}}\label{subsec:RLCMSCstars}
The \emph{full semigroup $C^*$-algebra} $C^*(S)$ of a right LCM semigroup $S$ is the $C^*$-algebra of $S$ as defined by Li \cite{Li1}*{Definition~2.2}. It can be shown that $C^*(S)$ is the universal $C^*$-algebra generated by an isometric representation $v$ of $S$ satisfying
\[
v_s^*v_t^{\phantom{*}} =
\begin{cases}
v_{s'}^{\phantom{*}}v_{t'}^* & \text{if $sS\cap tS=ss'S,\, ss'=tt'$}\\
0 & \text{if $s\perp t$,}
\end{cases}
\]
for all $s,t\in S$. For each $s\in S$, we write $e_{sS}$ for the range projection $v_s^{\phantom{*}}v_s^*$. The \emph{boundary quotient} $C^*$-algebra $\CQ(S)$ is defined in \cite{BRRW} to be the quotient of $C^*(S)$ by the relation $\prod_{f\in F} (1-e_{fS})=0$ for all foundation sets $F\subseteq S$. However, if $S$ has property~(AR), then in $\CQ(S)$ this relation reduces to
\begin{equation}\label{eq:BQrelation}
\begin{array}{c} \sum\limits_{f\in F} e_{fS}=1 \quad \text{for all accurate foundation sets $F\subseteq S$}.\end{array}
\end{equation}
In \cite{Sta3}, the fourth author introduced two intermediary quotients. The \emph{ core boundary quotient} $\CQ_c(S)$ is the quotient of $C^*(S)$ by the relation $v_s^{\phantom{*}}v_s^*=1$ for all $s\in S_c$. Note that this amounts to the quotient by the relation \eqref{eq:BQrelation} for accurate foundation sets contained in the core $S_c$. The second intermediary quotient depends on the notion of a \emph{proper foundation set}, which is a foundation set contained in the core irreducible elements $S_{ci}$. The \emph{proper boundary quotient} $\CQ_p(S)$ is the quotient of $C^*(S)$ by \eqref{eq:BQrelation} for all proper accurate foundation sets. As first defined in \cite{Sta3}*{Definition~2.9}, the \emph{boundary quotient diagram} is the commutative diagram from \eqref{eq:BQD for thm} which encompasses the $C^*$-algebras $C^*(S)$, $\CQ_c(S)$, $\CQ_p(S)$ and $\CQ(S)$.

\subsection{KMS-states}\label{subsec:KMSstates}
We finish the background section with the basics of the theory of KMS-states. A much more detailed presentation can be found in \cite{BRII}. Suppose $\sigma$ is an action of $\R$ by automorphisms of a $C^*$-algebra $A$. An element $a \in A$ is said to be \emph{analytic} for $\sigma$ if $t \mapsto \sigma_t(a)$ is the restriction of an analytic function $z \mapsto \sigma_z(a)$ from $\C$ into $A$. For $\beta>0$, a state $\phi$ of $A$ is called a \emph{KMS$_\beta$-state} of $A$ if it satisfies
\begin{equation}\label{eq:KMScond}
\phi(ab) = \phi(b\sigma_{i\beta}(a))\quad\text{ for all analytic $a,b \in A$.}
\end{equation}
We call \eqref{eq:KMScond} the {\em KMS$_\beta$-condition}. To prove that a given state of $A$ is a KMS$_\beta$-state, it suffices to check \eqref{eq:KMScond} for all $a, b$ in any $\sigma$-invariant set of analytic elements that spans a dense subspace of $A$.

A \emph{KMS$_\infty$-state} of $A$ is a weak$^*$ limit of KMS$_{\beta_n}$-states as $\beta_n\to \infty$, see \cite{con-mar}. A state $\phi$ of $A$ is a \emph{ground state} of $A$ if $z\mapsto \phi(a\sigma_z(b))$ is bounded on the upper-half plane for all analytic $a,b\in A$.

\section{The setting}\label{sec:setting}
Given a right LCM semigroup $S$ and its $C^*$-algebra $C^*(S)$, we wish to find a natural dynamics for which we can calculate the $\kms$-state structure. We then ask: what conditions on $S$ are needed in order to define such a dynamics? The answer that first comes to mind is the existence of a nontrivial homomorphism of monoids $N\colon S \to \N^\times$, which allows us to set $\sigma_x(v_s) := N_s^{ix} v_s$ for $s \in S$ and $x \in \R$. However, it soon becomes apparent that the mere existence of a homomorphism $N$ on $S$ seems far from sufficient to give rise to KMS-states on $C^*(S)$, yet alone describe them all. As announced in the introduction, we shall take into account the boundary quotient diagram for $C^*(S)$. The ingredients that are at work behind the scenes to get this diagram are the core $S_c$ and the core irreducible elements $S_{ci}$ of $S$. Looking for relationships between these and the homomorphism $N\colon S \to \N^\times$ will provide clues as to what conditions $S$ ought to satisfy.

Before we introduce our main concept of admissibility we fix the following notation: if $P$ is a subset of positive integers, we let $\langle P \rangle$ be the subsemigroup of $\N^\times$ generated by $P$. An element of $\langle P \rangle$ is \emph{irreducible} if $m=nk$ in $\langle P\rangle$ implies $n$ or $k$ is $1$. We denote by $\text{Irr}(\langle P\rangle)$ the set of irreducible elements of $\langle P\rangle$. We are now ready to present the outcome of our efforts of distilling a minimal sufficient set of conditions for the aim of defining a suitable dynamics $\sigma$:

\begin{definition}\label{def:admissible semigroups}
A right LCM semigroup $S$ is called \emph{admissible} if it satisfies the following conditions:
\begin{enumerate}[label=(A\arabic*)]
\item\label{cond:A1} $S=S_{ci}^1S_c$.
\item\label{cond:A2} Any right LCM in $S$ of a pair of elements in $S_{ci}$ belongs to $S_{ci}$.
\item\label{cond:A3} There is a nontrivial homomorphism of monoids $N\colon S \to \N^\times$ such that
\begin{enumerate}[label=(\alph*),ref=\ref{cond:A3}(\alph*)]
 \item\label{cond:A3a} $\lvert N^{-1}(n)/_\sim\rvert = n$ for all $n \in N(S)$ and
 \item\label{cond:A3b} for each $n \in N(S)$, every transversal of $N^{-1}(n)/_\sim$ is an accurate foundation set for $S$.
\end{enumerate}
\item\label{cond:A4} The monoid $N(S)$ is the free abelian monoid in $\text{Irr}(N(S))$.
\end{enumerate}
We refer to the property in \ref{cond:A1} as saying that $S$ is \emph{core factorable}, and when $S$ satisfies \ref{cond:A2} we say that $S_{ci}\subset S$ is \emph{$\cap$-closed}. We call a homomorphism $N$ satisfying \ref{cond:A3} a \emph{generalised scale}. This name is inspired by the notion of a scale on a semigroup originally defined in \cite{Laca98}*{Definition~9}.
\end{definition}

To state and prove our results on the KMS structure of the $C^*$-algebras of admissible right LCM semigroups, we also need to introduce a number of properties of these semigroups, and to prove a number of consequences of admissibility. The rest of this section is devoted to establishing these properties and thereby also explaining the meaning of the conditions \ref{cond:A1}--\ref{cond:A3}. Our first observation explains the meaning of \ref{cond:A2} in the presence of \ref{cond:A1}.

\begin{lemma}\label{lem:equiv forms of (VII)}
Suppose $S$ is a core factorable right LCM semigroup. Then for each $s \in S$ the set $[s] \cap S_{ci}^1$ is non-empty, and the following conditions are equivalent:
	\begin{enumerate}[(i)]
		\item The equivalence classes for $\sim$ have minimal representatives, that is, $[s] = tS_c$ holds for every $t \in [s] \cap S_{ci}^1$ and all $s \in S$.
		\item If $s,t \in S_{ci}^1$ satisfy $s \sim t$, then $s \in tS^*$.
	\end{enumerate}
Moreover, if $S_{ci}\subset S$ is $\cap$-closed, then (i) and (ii) hold.
\end{lemma}
\begin{proof}
Suppose $S$ has \ref{cond:A2}. The set $[s] \cap S_{ci}^1$ is non-empty for all $s \in S$ due to \ref{cond:A1}. Let $t \in [s]$, and write $s=s'a, t=t'b$ with $s',t' \in S_{ci}^1, a,b \in S_c$ using \ref{cond:A1}. Fix $c,d \in S_c$ with $sc=td$. Then $s' \not\perp t'$, so \ref{cond:A2} implies $s'S \cap t'S = rS$ for some $r \in S_{ci}^1$, say $r = s'p$ for a suitable $p \in S$. As $sc=td$, we have $s'ac = re = s'pe$ for some $e \in S$. Since $ac \in S_c$, we deduce that $p,e \in S_c$. But then $r \in S_{ci}^1$ forces $p \in S^*$ so that $r \in s'S^*$. The same argument applied to $t$ in place of $s$ shows $r \in t'S^*$. Thus we have shown $s' \in t'S^*$, and hence $[s] = s'S_c$. This establishes (i).
	
If (i) holds and $s \sim t$ for $s,t \in S_{ci}^1$, then $s \in [t] = tS_c$ forces $s \in tS^*$ as $s$ is core irreducible. Therefore, (i) implies (ii). Conversely, suppose (ii) holds. Now assume (iii) holds true. For $s \in S, s',t \in [s]$ with $t \in S_{ci}^1$, condition \ref{cond:A1} gives some $t'\in S_{ci}^1, a \in S_c$ with $s'=t'a$ so that $t\sim t'$. By (ii), there is $x \in S^*$ such that $t'=tx$, and hence $s'=txa \in tS_c$, proving (i).
\end{proof}

We proceed with a natural notion of homomorphisms between right LCM semigroups that appeared in \cite{BLS2}*{Theorem~3.3}, though without a name.

\begin{definition}\label{def:hom of right LCM sgps}
Let $\varphi\colon S \to T$ be a monoidal homomorphism between two right LCM semigroups $S$ and $T$. Then $\varphi$ is called a \emph{homomorphism of right LCM semigroups} if
\begin{equation}\label{eq:hom of right LCM}
\varphi(s_1)T \cap \varphi(s_2)T = \varphi(s_1S\cap s_2S)T \quad \text{for all } s_1,s_2 \in S.
\end{equation}
\end{definition}

In particular, \eqref{eq:hom of right LCM} implies that $\varphi(r)$ is a right LCM for $\varphi(s_1)$ and $\varphi(s_2)$ whenever $r \in S$ is a right LCM for $s_1$ and $s_2$ in $S$. This notion allows us to recast the condition of $S_{ci}\subset S$ being $\cap$-closed using the semigroup $S_{ci}':= S_{ci} \cup S^*$:

\begin{proposition}\label{prop:hom of right LCM for S_c and S_ci}
Let $S$ be a right LCM semigroup. Then the following hold:
\begin{enumerate}[(i)]
\item The natural embedding of $S_c$ into $S$ is a homomorphism of right LCM semigroups.
\item Suppose $S$ is core factorable. Then $S_{ci} \subset S$ is $\cap$-closed if and only if $S_{ci}'$ is a right LCM semigroup and the natural embedding of $S_{ci}'$ into $S$ is a homomorphism of right LCM semigroups.
\end{enumerate}
\end{proposition}
\begin{proof}
In both cases, the right hand side of \eqref{eq:hom of right LCM} is contained in the left hand side. By the definition of $S_c$, any $c \in S$ satisfying $aS\cap bS = cS$ for some $a,b \in S_c$ belongs to $S_c$, and we have $aS_c \cap bS_c \neq \emptyset$ for all $a,b \in S_c$. This proves (i).

Suppose \ref{cond:A1} holds. Assuming \ref{cond:A2}, we must prove that $S_{ci}'$ is a semigroup. The non-trivial case is to show that $st\in S_{ci}'$ when $s,t$ are core irreducible. So let $s,t \in S_{ci}$ and $st=ra$ for some $r \in S\setminus S_c, a \in S_c$. As $S$ is core factorable, we can assume $r \in S_{ci}$. Then $sS \cap rS = r'S$ for some $r'\in S$. Since $s,r$ are in $S_{ci}$, so is $r'$  by our assumption. Since $ra \in r'S \subset rS$, it follows that $r'\sim r$. By Lemma~\ref{lem:equiv forms of (VII)}, \ref{cond:A2} yields $r\in r'S^* \subset sS$. Therefore $r=ss'$ for some $s' \in S$, and left cancellation leads to $t=s'a$. By core irreducibility of $t$, we conclude that $a \in S^*$. Thus $st \in S_{ci}$.

For the other  claim in the forward direction of (ii), it suffices to look at  intersections $sS_{ci}' \cap tS_{ci}'$ versus $sS \cap tS$ with $s,t \in S_{ci}$ because when $s \in S^*$ we are left with $tS_{ci}'$ and $tS$, respectively. Since $S_{ci}'$ is a subsemigroup of $S$, we  have $sS_{ci}' \cap tS_{ci}' \subset sS \cap tS$. In particular, we can assume that $sS \cap tS=rS$ for some $r \in S$. By \ref{cond:A2}, $r \in S_{ci}$. Then any $s',t' \in S$ with $ss'=r=tt'$ belong to $S_{ci}'$. Therefore $sS_{ci}' \cap tS_{ci}' =rS_{ci}'$, showing  that the intersection of principal right ideals in $S_{ci}'$ may  be computed in $S$. Hence, $S_{ci}'$ is right LCM because $S$ is, and the embedding is a homomorphism of right LCM semigroups.

Conversely, let $s,t \in S_{ci}$  such that $sS \cap tS=rS$ for some $r \in S$. Then $r \in S\setminus S_c$. But the right LCM property for the embedding of $S_{ci}'$ into $S$ gives $r \in S_{ci}'$, so $r \in (S\setminus S_c) \cap S_{ci}' = S_{ci}$, which is \ref{cond:A2}.
\end{proof}

\begin{lemma}\label{lem:OldA2}
Suppose $S$ is a core factorable right LCM semigroup, and $a\in S_c$, $s\in S_{ci}$ satisfy $aS\cap sS=atS$ and $at=sb$ for some $b,t\in S$. Then $b \in S_c$. Further, if $S_{ci}\subset S$ is $\cap$-closed, then $t$ belongs to $S_{ci}$.
\end{lemma}
\begin{proof}
To see that $b\in S_c$, let $s'\in S$. Since $a\in S_c$, we have $aS\cap ss'S\not=\emptyset$. Hence left cancellation gives
\[
sbS\cap ss'S = (aS\cap sS)\cap ss'S \not=\emptyset \implies bS\cap s'S\not=\emptyset \implies b\in S_c.
\]
Since $s\in S_{ci}$, we must have $t\in S\setminus S_c$. We use \ref{cond:A1} to write $t=rc$ for $r\in S_{ci}$ and $c\in S_c$. Then $ar\sim s$, and due to \ref{cond:A2} we know from Lemma~\ref{lem:equiv forms of (VII)} that $ar\in sS_c\subseteq sS$. Hence $ar\in aS\cap sS=atS$, and so $arS\subseteq at S$. Since $atS=arcS\subseteq arS$, we have $arS=atS$. So $rS=tS$. This implies that $t=rx$ for some $x\in S^*$, and thus $t\in S_{ci}$ as $r\in S_{ci}$.
\end{proof}

We now prove some fundamental results about right LCM semigroups that admit a generalised scale.

\begin{proposition}\label{prop:NisaHMofRLCMs}
Let $N$ be a generalised scale on a right LCM semigroup $S$.
\begin{enumerate}[(i)]
\item We have $\ker N=S_c$, and $S_c$ is a proper subsemigroup of $S$.
\item If $s,t \in S$ satisfy $N_s=N_t$, then either $s \sim t$ or $s \perp t$.
\item Let $F$ be a foundation set for $S$ with $F \subset N^{-1}(n)$ for some $n \in N(S)$. If $\lvert F \rvert = n$ or $F$ is accurate, then $F$ is a transversal for $N^{-1}(n)/_\sim$.
\item The map $N$ is a homomorphism of right LCM semigroups.
\item Every foundation set for $S$ admits an accurate refinement by a transversal for $N^{-1}(n)/_\sim$ for some $n \in N(S)$. In particular, $S$ has the accurate refinement property.
\end{enumerate}
\end{proposition}
\begin{proof}
For (i), suppose $s\in \ker N$. Then \ref{cond:A3} implies that $\{s\}$ is a foundation set for $S$, which says exactly that $s \in S_c$. On the other hand, if $s\in S$ with $N_s >1$, then \ref{cond:A3} implies that $s$ is part of an accurate foundation set of cardinality $N_s$. Thus $\{s\}$ is not a foundation set, and hence $s \notin S_c$. The second claim holds because $N$ is nontrivial.

For (ii), we first claim that if $s,s'\in S$ satisfy $s\sim s'$, then $s \perp t$ if and only $s'\perp t$ for all $t\in S$. Indeed, for $sa=s'b$ with $a,b \in S_c$ we have
\[sS \cap tS \neq \emptyset \stackrel{a \in S_c}{\Longrightarrow} saS\cap tS \neq \emptyset \stackrel{sa=s'b}{\Longrightarrow} s'S \cap tS \neq \emptyset.\]
Now suppose we have $s,t \in S$ with $N_s=N_t$. Choose a transversal for $N^{-1}(N_s)/_\sim$ and let $s'$ and $t'$ be the unique elements of the transversal satisfying $s \sim s'$ and $t \sim t'$. If $s'=t'$, then $s\sim t$. If $s'\not= t'$, then since we know from \ref{cond:A3} that the transversal is an accurate foundation set for $S$, we have $s' \perp t'$. It now follows from the claim that $s \perp t$.

For (iii), let $\CT_n$ be a transversal for $N^{-1}(n)/_\sim$. We know that $\CT_n$ is accurate because of \ref{cond:A3b}. By (ii), we know that for every $f \in F$, there is precisely one $t_f \in \CT_n$ with $f \sim t_f$. So $f\mapsto t_f$ is a well-defined map from $F$ to $\CT_n$, which is surjective because $F$ is a foundation set. The map is injective if $F$ is accurate or has cardinality $\lvert \CT_n\rvert = n$.

For (iv), let $s,t,r \in S$ with $sS \cap tS = rS$. Then $N_r \in N_sN(S) \cap N_tN(S)$. We need to show that $N_r$ is the least common multiple of $N_s$ and $N_t$ inside $N(S)$. Suppose there are $m,n \in N(S)$ with $N_sm=N_tn$. For $k=N_s,N_t,m,n$ pick transversals $\CT_k$ for $N^{-1}(k)/_\sim$ such that $s \in \CT_{N_s}$ and $t \in \CT_{N_t}$. Since $s \not \perp t$, and $\CT_m$ is a foundation set, there exists $s' \in \CT_m$ such that $ss' \not\perp t$. The same argument for $n$ shows that there is $t' \in \CT_n$ with $ss' \not\perp tt'$. Now $\CT_{N_s}\CT_m$ and $\CT_{N_t}\CT_n$ are accurate foundation sets contained in $N^{-1}(N_sm)$. By (iii), they are transversals for $N^{-1}(N_sm)/_\sim$. In particular, $ss' \not\perp tt'$ is equivalent to $ss' \sim tt'$. Thus there are $a,b \in S_c$ with $ss'a = tt'b \in sS \cap tS=rS$, which results in $N_{s}m=N_{ss'a} \in N_rN(S)$ due to (i). Thus $N_r$ is the least common multiple of $N_s$ and $N_t$.

For (v), let $F$ be a foundation set for $S$. Define $n$ to be the least common multiple of $\{N_f \mid f \in F\}$ inside $N(S)$, and set $n_f := n/N_f \in N(S)$ for each $f \in F$. Choose a transversal $\CT_f$ for $N^{-1}(n_f)/_\sim$ for every $f \in F$. Then $F' := \bigcup_{f \in F} f\CT_f$ is a foundation set as $F$ is a foundation set and each $\CT_f$ is a foundation set. Moreover, we have $F' \subset N^{-1}(n)$. Thus by (ii), for $f,f' \in F'$, we either have $f \sim f'$ or $f \perp f'$. If we now choose a maximal subset $F_a$ of $F'$ with elements having mutually disjoint principal right ideals, $F_a$ is necessarily an accurate foundation set that refines $F$. In fact, $F_a$ is a transversal for $N^{-1}(n)/_\sim$. In particular, this shows that $S$ has the accurate refinement property.
\end{proof}

The next result is an immediate corollary of Proposition~\ref{prop:NisaHMofRLCMs}~(ii) and Lemma~\ref{lem:equiv forms of (VII)}.

\begin{corollary}\label{cor:dichotomy for S_ci under gen scale}
Let $S$ be a core factorable right LCM semigroup such that $S_{ci}\subset S$ is $\cap$-closed. If $N$ is a generalised scale on $S$, then $N_s=N_t$ for $s,t \in S_{ci}^1$ forces either $s \perp t$ or $s \in tS^*$.

In particular, if $\CT_n=\{f_1,\dots, f_n\}$ and $\CT_n'=\{f_1', \dots, f_n'\}$ are two transversals for $N^{-1}(n)/_\sim$, respectively, such that both are contained in $S_{ci}$, then there are $x_1,\ldots,x_n$ in $S^*$ and a permutation $\rho$ of $\{1,\ldots,n\}$ so that $f'_i = f_{\rho(i)}x_i$ for $i=1,\ldots,n$.
\end{corollary}

Lemma~\ref{lem:equiv forms of (VII)} allows us to strengthen the factorisation $S=S_{ci}^1S_c$ in the following sense.

\begin{lemma}\label{lem:maps i and c}
Let $S$ be a core factorable right LCM semigroup such that $S_{ci} \subset S$ is $\cap$-closed. Then there are a transversal $\CT$ for $S/_\sim$ with $\CT \subset S_{ci}^1$ and maps $i \colon S \to \CT$, $c\colon S \to S_c$ such that $s = i(s)c(s)$ for all $s \in S$. For every family $(x_t)_{t \in T} \subset S^*$ with $x_1=1$, the set $\CT' := \{ tx_t \mid t \in \CT\}$ defines a transversal for $S/_\sim$ with $\CT'\subset S_{ci}^1$, and every transversal for $S/_\sim$ contained in $S_{ci}^1$ is of this form.
\end{lemma}
\begin{proof}
By \ref{cond:A1}, $S/_\sim$ admits a transversal $\CT$ with $\CT \subset S_{ci}^1$. For each $s\in S$ we define $i(s)$ to be the single element in $\CT\cap[s]$, and $c(s)$ to be the unique element in $S_c$ such that $s=i(s)c(s)$, which exists by virtue of Lemma~\ref{lem:equiv forms of (VII)} due to the presence of \ref{cond:A2}. This gives us the maps $i$ and $c$.

As $S_{ci}S^* = S_{ci}$ and $S^* \cap S_{ci}^1 = \{1\}$, the remaining claims follow from Lemma~\ref{lem:equiv forms of (VII)}.
\end{proof}

In our investigations we have found that uniqueness of the $\kms_\beta$-state for $\beta$ in a critical interval, see Theorem~\ref{thm:KMS results-gen}, is closely related to properties of a natural action $\alpha$ of $S_c$ on $S/_\sim$ arising from left multiplication. By an action $\alpha$ of a semigroup $T$ on a set $X$ we mean a map $T \times X \to X, (t,x) \mapsto \alpha_t(x)$ such that $\alpha_{st}(x)=\alpha_s(\alpha_t(x))$ for all $s,t\in T$ and $x\in X$. For later use, we denote $\text{Fix}(\alpha_t)$ the set $\{x\in X\mid \alpha_t(x)=x\}$.

\begin{lemma}\label{lem:thebijectionsg}
Let $S$ be a right LCM semigroup. Then left multiplication defines an action $\alpha\colon S_c \curvearrowright S/_\sim$ by bijections $\alpha_a([s]) = [as]$. Every generalised scale $N$ on $S$ is invariant under this action. If $S$ is admissible and $\CT \subset S_{ci}^1$ is a transversal for $S/_\sim$, then there is a corresponding action of $S_c$ on $\CT$, still denoted $\alpha$, which satisfies
\[ aS \cap tS = a\alpha_{a}^{-1}(t)S \quad\text{for all } a \in S_c, t \in \CT.\]
\end{lemma}
\begin{proof}
For every $a \in S_c$, the map $\alpha_a$ is well-defined as left multiplication preserves the core equivalence relation. If $as \sim at$ for some $s,t \in S$, then $s \sim t$ by left cancellation. Hence $\alpha_a$ is injective. On the other hand, Lemma~\ref{lem:OldA2} states that for $s \in S$ we have $aS\cap sS = atS, at=sb$ for some $b \in S_c$ and $t \in S$. Thus $\alpha_a([t]) = [at] = [s]$, and we conclude that $\alpha_a$ is a bijection.

For a generalised scale $N$, invariance follows from $\ker N=S_c$, see Proposition~\ref{prop:NisaHMofRLCMs}~(i).

Finally, let $\CT$ be a transversal for $S/_\sim$ contained in $S_{ci}^1$. For $a \in S_c$ and $t \in \CT$, it is immediate that $a\alpha_{a}^{-1}(t) = tc(a\alpha_{a}^{-1}(t))$. Assume $t\neq 1$ (else the claim about the interesection is trivial), and use Lemma~\ref{lem:OldA2} to write $aS \cap tS = asS$ for some $s \in S_{ci}$. By the choice of $\CT$, there is a unique $s \in \CT$ with this property. Proposition~\ref{prop:NisaHMofRLCMs}~(i) and (iv) show that $N_{\alpha_{a}(s)}=N_{as}=N_t$, so that $t=\alpha_{a}(s)$ by Corollary~\ref{cor:dichotomy for S_ci under gen scale}.
\end{proof}

 Recall that an action $\gamma:T\curvearrowright X$ is \emph{faithful} if, for all $s,t \in T, s \neq t$, there is $x \in X$ such that $\gamma_s(x) \neq \gamma_t(x)$.

\begin{corollary}\label{cor:S_c action faithful implies right cancellation for S_c}
If $S$ is a right LCM semigroup for which $S_c \curvearrowright S/_\sim$ is faithful, then $S_c$ is right cancellative, and hence a right Ore semigroup. In particular, $S_c$ is group embeddable.
\end{corollary}
\begin{proof}
By Lemma~\ref{lem:thebijectionsg}, we have a monoidal homomorphism from $S_c$ to the bijections on $S/_\sim$, which is injective due to faithfulness of the action.
\end{proof}

\begin{remark}\label{rem:almost free actions}
Recall that an action $\gamma\colon T \curvearrowright X$ of a monoid $T$ on a set $X$ is called \emph{almost free} if the set $\{x \in X \mid \gamma_s(x)=\gamma_t(x)\}$ is finite for all $s,t \in T, s\neq t$. Clearly, if $X$ is infinite, then every almost free action $\gamma$ on $X$ is faithful. In particular, this is the case for $\alpha\colon S_c \curvearrowright S/_\sim$ for every admissible right LCM semigroup $S$.
\end{remark}

\begin{remark}\label{rem:strong minimality}
If $T$ is a group, then almost freeness states that every $\gamma_t$ fixes only finitely many points in $X$. The same conclusion holds if $T$ is totally ordered with respect to $s \geq t:\Leftrightarrow s\in tT$ and acts by injective maps: We then have $T^*=\{1\}$, and $s=tr$ ($s \geq t$) gives $\gamma_t^{-1}\gamma_s^{\phantom{1}} = \gamma_r$, while $t=sr$ ($t\geq s$) yields $\gamma_s^{-1}\gamma_t^{\phantom{1}} = \gamma_r$.
\end{remark}

Almost freeness is known to be too restrictive a condition for some of the semigroups we wish to study, for instance for $S=X^* \bowtie G$ arising from a self-similar action $(G,X)$ that admits a word $w \in X^*$ such that $g(w)=w$ for some $g \in G\setminus\{1\}$ (see \ref{subsec:SSA} for more on self-similar actions). To study the uniqueness of $\kms_\beta$-states with $\beta$ in the critical interval, we will thus also elaborate on the approach used in \cite{LRRW} for self-similar actions. But we phrase their ideas in a more abstract setting which gives us a potentially useful extra degree of freedom: The condition we work with is a localised version of the finite state condition used in \cite{LRRW}.

\begin{definition}\label{def:finite propagation}
 Let $S$ be an admissible right LCM semigroup, and $\CT$ a transversal for $S/_\sim$ with $\CT\subset S_{ci}^1$. For $a \in S_c$, let $C_a := \{ c(af) \mid f \in \CT \}$. The transversal $\CT$ is said to be a \emph{witness of finite propagation} of $S$ at $(a,b) \in S_c \times S_c$ if $C_a$ and $C_b$ are finite. The right LCM semigroup $S$ is said to have \emph{finite propagation} if there is a witness of finite propagation at every pair $(a,b) \in S_c \times S_c$.
\end{definition}

\begin{remark}\label{rem:fin prop independent of choice}
(a) If $S$ is such that $S^*$ is infinite, then for any prescribed $a \in S_c\setminus\{1\}$ there exists a transversal $\CT$ for which $C_a$ is infinite, just take $\{xc(a)\mid x\in S^*\}$. Thus, the choice of $\CT$ is crucial in Definition~\ref{def:finite propagation}.

(b) Note that $S$ has finite propagation if and only if for every transversal $\CT$ for $S/_\sim$ with $\CT\subset S_{ci}^1$ (with corresponding maps $i,c$) and for all $a,b \in S_c$, there is $(x_f)_{f \in \CT} \subset S^*$ such that $\{ x_fc(af) \mid f \in \CT \}$ and $\{ x_fc(bf) \mid f \in \CT \}$ are finite. Indeed, assuming that $S$ has finite propagation, let $\CT$ be a transversal contained in $S_{ci}^1$ and let $a,b\in S_c$. Then there is a transversal $\CT'$ which is a witness of finite propagation at $(a,b)$. However, by Corollary~\ref{cor:dichotomy for S_ci under gen scale}, there are $x_f\in S^*$ for all $f\in \CT$ such that $\CT'=\{fx_f\mid f\in \CT\}$. Hence $C_a=\{x_{i(af')}c'(af')\mid f'\in \CT'\}$ is finite, and similarly for $C_b$. Reversing the argument gives the converse assertion.
\end{remark}

\section{The main theorem}\label{sec:mainthm}
If $S$ is a right LCM semigroup admitting a generalised scale $N$, then standard arguments using the universal property of $C^*(S)$ show that there is a strongly continuous action $\sigma:\R\to \Aut C^*(S)$, where
\[
\sigma_x(v_s)=N_s^{ix}v_s\quad\text{for each $x\in\R$ and $s\in S$.}
\]
It is easy to see that $\{v_s^{\phantom{*}}v_t^*\mid s,t\in S\}$ is a dense family of analytic elements. Noting that $\sigma$ is the identity on $\ker~\pi_c$ and $\ker~\pi_p$, hence also on $\ker~\pi$, the action $\sigma$ drops to strongly continuous actions of $\R$ on $\CQ_c(S),\CQ_p(S)$ and $\CQ(S)$. Thus we may talk of $\kms$-states for the corresponding dynamics on $\CQ_c(S)$ and $\CQ_p(S)$. 

\begin{remark}\label{rem:C*(S_c) in C*(S)}
Since the canonical embedding of $S_c$ into $S$ is a homomorphism of right LCM semigroups, see Proposition~\ref{prop:hom of right LCM for S_c and S_ci}~(i), the universal property of semigroup $C^*$-algebras guarantees existence of a $*$-homomorphism $\varphi:C^*(S_c)\to C^*(S), w_a \mapsto v_a$, where $w_a$ denotes the generating isometry for $a\in S_c$ in $C^*(S_c)$. 
\end{remark}

In our analysis of $\kms$-states on $C^*(S)$ and its boundary quotients under the dynamics $\sigma$, a key role is played by a $\zeta$-function for (parts of) $S$:
\begin{definition}\label{def:partial zeta function}
Let $I \subset \Irr(N(S))$. For each $n\in I$, let $\CT_n$ be a transversal for $N^{-1}(n)/_\sim$, which by \ref{cond:A3b} is known to be an accurate foundation set. Then the formal series
\[\begin{array}{c}
\zeta_I(\beta) := \sum\limits_{n \in \langle I \rangle}\sum\limits_{f \in \CT_n} N_f^{-\beta} = \sum\limits_{n \in \langle I \rangle} n^{-(\beta-1)},
\end{array}\]
where $\beta \in \R$, is called the \emph{$I$-restricted $\zeta$-function} of $S$. For $I=\Irr(N(S))$, we write simply $\zeta_S$ for $\zeta_{\Irr(N(S))}$ and call it the \emph{$\zeta$-function} of $S$. The \emph{critical inverse temperature} $\beta_c \in \R \cup \{\infty\}$ is the smallest value so that $\zeta_S(\beta) < \infty$ for all $\beta \in \R$ with $\beta > \beta_c$. The \emph{critical interval} for $S$ is given by $[1,\beta_c]$ if $\beta_c$ is finite, and $[1,\infty)$ otherwise.
\end{definition}

Recall the action $\alpha$ introduced in Lemma~\ref{lem:thebijectionsg}. The statement of our main result, which is the following theorem, makes reference to a trace $\rho$ on $C^*(S_c)$ constructed in Proposition~\ref{prop:min gives unique kms in the crit int} under the assumption that $\beta_c=1$.

\begin{thm}\label{thm:KMS results-gen}
Let $S$ be an admissible right LCM semigroup and let $\sigma$ be the one-parameter group of automorphisms of $C^*(S)$ given by $\sigma_x(v_s)=N_s^{ix}v_s$ for $x\in \R$ and $s\in S$. The $\kms$-state structure with respect to $\sigma$ on the boundary quotient diagram \eqref{eq:BQD for thm} has the following properties:
\begin{enumerate}
\item[(1)] There are no $\kms_\beta$-states on $C^*(S)$ for $\beta < 1$.
\item[(2a)] If $\alpha$ is almost free, then for each $\beta$ in the critical interval there is a unique $\kms_\beta$-state $\psi_\beta$ given by $\psi_\beta( v_s^{\phantom{*}}v_t^*) = N_s^{-\beta} \delta_{s,t}$ for $s,t\in S$.
\item[(2b)] If $\beta_c=1$, $\alpha$ is faithful, and $S$ has finite propagation, then there is a unique $\kms_\beta$-state $\psi_\beta$ determined by the trace $\rho$  on $C^*(S_c)$.
\item[(3)] For $\beta > \beta_c$, there is an affine homeomorphism between $\kms_\beta$-states on $C^*(S)$ and normalised traces on $C^*(S_c)$.
\item[(4)] Every $\kms_\beta$-state factors through $\pi_c$.
\item[(5)] A $\kms_\beta$-state factors through $\pi_p$ if and only if $\beta=1$.
\item[(6)] There is an affine homeomorphism between ground states on $C^*(S)$ and states on $C^*(S_c)$. In case that $\beta_c < \infty$, a ground state is a $\kms_\infty$-state if and only if it corresponds to a normalised trace on $C^*(S_c)$ under this homeomorphism.
\item[(7)] Every $\kms_\infty$-state factors through $\pi_c$. If $\beta_c < \infty$, then all ground states on $\CQ_c(S)$ are $\kms_\infty$-states if and only if every ground state on $C^*(S)$ that factors through $\pi_c$ corresponds to a normalised trace on $C^*(S_c)$ via the map from {(6)}.
\item[(8)] No ground state on $C^*(S)$ factors through $\pi_p$, and hence none factor through $\pi$.
\end{enumerate}
\end{thm}

\begin{proof} Assertion (1) will follow directly from Proposition~\ref{prop:no KMS below 1-gen}. For the existence claim in (2a) and (2b), we apply Proposition~\ref{prop:construction of KMS-states} starting from the canonical trace $\tau$ on $C^*(S_c)$ given by $\tau(w_sw_t^*)=\delta_{s,t}$ for all $s,t\in S_c$. The uniqueness assertion will follow from Proposition~\ref{prop:strong min gives unique kms in the crit int} for (2a), and from Proposition~\ref{prop:min gives unique kms in the crit int} for (2b).

For assertion (3), Proposition~\ref{prop:construction of KMS-states} shows the existence of a continuous, affine parametrisation $\tau\mapsto \psi_{\beta,\tau}$ that produces a $\kms_\beta$-state $\psi_{\beta,\tau}$ on $C^*(S)$ out of a trace $\tau$ on $C^*(S_c)$. This map is surjective by Corollary~\ref{cor:KMS-states parametrization surjective} and injective by Proposition~\ref{prop:KMS-states parametrization injective}.

To prove (4), it suffices to show that a $\kms_\beta$-state $\phi$ on $C^*(S)$ for $\beta \in \R$ vanishes on $1-v_a^{\phantom{*}}v_a^*$ for all $a\in S_c$. But $S_c = \ker N$ implies $\sigma_{i\beta}(v_a) =v_a$ so $\phi(v_a^{\phantom{*}}v_a^*) = \phi(v_a^*v_a^{\phantom{*}}) = \phi(1)=1$ since $v_a$ is an isometry. Assertion (5) is the content of Proposition~\ref{prop:KMS-state factoring through Q_p must have beta=1}.

Towards proving (6), Proposition~\ref{prop:construction of ground states} provides an affine assignment $\rho\mapsto \psi_\rho$ of states $\rho$ on $C^*(S_c)$ to ground states $\psi_\rho$ of $C^*(S)$ such that $\psi_\rho \circ \varphi = \rho$, for the canonical map $\varphi\colon C^*(S_c) \to C^*(S)$ from Remark~\ref{rem:C*(S_c) in C*(S)}. Thus $\rho\mapsto \psi_\rho$ is injective, and surjectivity follows from Proposition~\ref{prop:ground states on C*(S)-gen}. If $\beta_c<\infty$, then (3) and Proposition~\ref{prop:ground states that factor through Q_c vs. KMS infty} show that $\kms_\infty$-states correspond to normalised traces on $C^*(S_c)$.

The first claim in statement (7) follows from (4). For the second part of (7), we observe that $\phi \circ \pi_c$ is a ground state on $C^*(S)$ for every ground state $\phi$ on $\CQ_c(S)$. Suppose first that all ground states on $\CQ_c(S)$ are $\kms_\infty$-states. If $\phi'$ is a ground state on $C^*(S)$ such that $\phi'= \phi \circ \pi_c$ for some ground state $\phi$ on $\CQ_c(S)$, then $\phi$ is a $\kms_\infty$-state by assumption. This readily implies that $\phi'$ is also a $\kms_\infty$-state, and hence corresponds to a trace as $\beta_c < \infty$, see (6). Conversely, suppose that all ground states on $C^*(S)$ that factor through $\pi_c$ correspond to traces under the map from (6). Then every ground state $\phi$ on $\CQ_c(S)$ corresponds to a trace via $\phi \circ \pi_c$. Hence $\phi \circ \pi_c$ is a $\kms_\infty$-state by (6) as $\beta_c < \infty$, and (4) implies that $\phi$ is a $\kms_\infty$-state on $\CQ_c(S)$ as well. 

Finally, statement (8) is proved in Proposition~\ref{prop:no GS on Q_p}.
\end{proof}

\begin{remark}\label{rem:crit interval uniqueness of KMS-states conditions}
The idea behind the trace $\rho$ on $C^*(S_c)$ in Theorem~\ref{thm:KMS results-gen}~(2b) stems from \cite{LRRW}. If $S$ is right cancellative, then $\rho$ is the canonical trace, i.e. $\rho(w_a^{\phantom{*}}w_b^*) = \delta_{a,b}$ for $a,b\in S_c$. If $\alpha$  is almost free then $\alpha$ is faithful by Remark~\ref{rem:almost free actions}, and  $S_c$ is right cancellative by Corollary~\ref{cor:S_c action faithful implies right cancellation for S_c}. Now if $S_c$ is right cancellative, then the value $\rho(w_a^{\phantom{*}}w_b^*)$ describes the asymptotic proportion of elements in $S/_\sim \cap N^{-1}(n)$ that are fixed under $\alpha_a^{\phantom{1}} \alpha_b^{-1}$ as $n \to \infty$ in $N(S)$. Thus, if $\beta_c=1$, and $S$ has finite propagation, then almost freeness forces $\rho$ to be the canonical trace so that (2a) and (2b) are consistent.
\end{remark}

\begin{remark}\label{rem:connection to CDL-paper}
In \cite{CDL}, the $\kms$-state structure for $S=R \rtimes R^\times$, where $R$ is the ring of integers in a number field was considered for a dynamics of the same type as we discuss here. $S$ is right LCM if and only if the class group of $R$ is trivial, and it is quite intriguing to see how the results of \cite{CDL} relate to Theorem~\ref{thm:KMS results-gen}: We have $\beta_c=2$ and the unique $\kms_\beta$-state for $\beta \in [1,2]$ in \cite{CDL}*{Theorem~6.7} is analogous to the one in Theorem~\ref{thm:KMS results-gen}~(2a). Since the action $\alpha\colon S_c \curvearrowright S/_\sim$ is almost free, we may expect this uniqueness result to hold outside the class of right LCM semigroups. 

On the other hand, the parametrisation of $\kms_\beta$-states for $\beta >2$ from \cite{CDL}*{Theorem~7.3} is by traces on $\bigoplus_{\gamma} C^*(J_\gamma \rtimes R^*)$, where $\gamma$ ranges over the class group along with a fixed reference ideal $J_\gamma \in \gamma$. If $S$ is right LCM, then this coincides with Theorem~\ref{thm:KMS results-gen}~(3) as $S_c=S^*=R\rtimes R^*$ and we can take $J_1 = R$. So \cite{CDL} predicts that we can expect to see a more complicated trace simplex beyond the right LCM case that reflects the finer structure of $\alpha\colon S_c \curvearrowright S/_\sim$. In connection with the ideas behind Theorem~\ref{thm:KMS results-gen}~(6). this might also yield a new perspective on the description of ground states provided in \cite{CDL}*{Theorem~8.8}.
\end{remark}

\begin{remark}\label{rem:KMS_infty vs Q_c(S)}
The phase transition result for ground states for the passage from $C^*(S)$ to $\CQ_c(S)$ in Theorem~\ref{thm:KMS results-gen}~(7) raises the question whether there is an example of an admissible right LCM semigroup $S$ for which $\pi_c$ is nontrivial and not all ground states on $\CQ_c(S)$ are $\kms_\infty$-states. In all the examples we know so far, see Section~\ref{sec:examples}, we either have $\pi_c = \id$, for instance for the case of self-similar actions or algebraic dynamical systems, or the core subsemigroup $S_c$ is abelian so that all ground states that factor through $\pi_c$ correspond to traces on $C^*(S_c)$.
\end{remark}

\section{Examples}\label{sec:examples}
Before proving our main theorem on the KMS structure of admissible right LCM semigroups, we use this section to prove that a large number of concrete examples from the literature are admissible. Indeed, for some of the requirements for admissibility (\ref{cond:A1} and \ref{cond:A2}), we do not know any right LCM semigroups that fail to possess them. We break this section into subsections, starting with reduction results for Zappa-Sz\'{e}p products of right LCM semigroups that will be applied to easy right-angled Artin monoids \ref{subsec:examples easy right-angled Artin monoids}, self-similar actions \ref{subsec:SSA}, subdynamics of $\N\rtimes\N^\times$ \ref{subsec:NxNtimes}, and Baumslag-Solitar monoids \ref{subsec:BS-monoid}. However, the case of algebraic dynamical systems requires different considerations, see \ref{subsec:ADS}.

\subsection{Zappa-Sz\'ep products}\label{subsec:ZSProducts}
Let $U,A$ be semigroups with identity $e_U,e_A$. Suppose there exist maps $(a,u)\mapsto a(u):A\times U\to U$ and
$(a,u)\mapsto a|_u:A\times U\rightarrow A$ such that
 \begin{tabbing}
 \,\,\,\,(ZS1) $e_A(u) = u$;\hspace{3cm} \= (ZS5) $e_A|_{u}=e_A$; \\
 \,\,\,\,(ZS2) $(a b)(u) = a(b (u))$; \>(ZS6) $a|_{uv} = (a|_u)|_v$;\\
 \,\,\,\,(ZS3) $a (e_U) = e_U$; \> (ZS7) $a (uv) = a (u)a|_u(v)$ ; and\\
 \,\,\,\,(ZS4) $a|_{e_U}=a$; \> (ZS8) $(a b)|_u = a|_{b (u)} b|_u$.
\end{tabbing}
Following \cite{Bri1}*{Lemma~3.13(xv)} and \cite{BRRW}, the external Zappa-Sz\'ep product $U \bowtie A$ is the cartesian product $U\times A$ endowed with the multiplication $(u,a)(v,b)=(ua(v),a|_v b)$ for all $a,b\in A$ and $u,v\in U$. Given $a\in A,u\in U$, we call $a(u)$ the action of $a$ on $u$, and $a|_u$ the restriction of $a$ to $u$. For convenience, we write $1$ for both $e_A$ and $e_U$.

In this subsection, we derive an efficient way of identifying admissible Zappa-Sz\'{e}p products $S= U \bowtie A$ among those which share an additional intersection property: For this class, the conditions \ref{cond:A1}--\ref{cond:A4} mostly boil down to the corresponding statements on $U$. Suppose that $S=U\bowtie A$ is a right LCM semigroup and that $A$ is also right LCM. Then $U$ is necessarily a right LCM semigroup.
Let us remark that $S^* = U^* \bowtie A^*$. Moreover, for $(u,a),(v,b) \in S$, it is straightforward to see that $(u,a) \not\perp (v,b)$ implies $u \not\perp v$. We are interested in a strong form of the converse implication:
\begin{equation}\label{eq:cap condition ZS-product}\tag{$\cap$}
\begin{aligned}
\text{If } uU \cap vU = wU, \text{ then } (u,a)S \cap (v,b)S = (w,c)S \text{ for some } c\in A.
\end{aligned}
\end{equation}
We say that \eqref{eq:cap condition ZS-product} holds (for $S$) if it is valid for all $(u,a),(v,b) \in S$.

\begin{remark}\label{rem:examples of cap condition}
The property \eqref{eq:cap condition ZS-product} is inspired by \cite{BRRW}*{Remark~3.4}. If a Zappa-Sz\'ep product $U \bowtie A$ satisfies the hypotheses of \cite{BRRW}*{Lemma~3.3}, that is,
\begin{enumerate}[(a)]
\item $U$ is a right LCM semigroup,
\item $A$ is a left cancellative monoid whose constructible right ideals are totally ordered by inclusion, and
\item $u\mapsto a(u)$ is a bijection of $U$ for each $a\in A$,
\end{enumerate}
then \cite{BRRW}*{Remark~3.4} shows that the right LCM semigroup $U \bowtie A$ satisfies \eqref{eq:cap condition ZS-product}.
\end{remark}

While surprisingly many known examples fit into the setup of \cite{BRRW}*{Lemma~3.3}, the following easy example shows that they are not necessary to ensure that $U\bowtie A$ is right LCM or that \eqref{eq:cap condition ZS-product} holds.

\begin{example}\label{ex:right LCM SZ-prod not covered by BRRW3.3}
Consider the standard restricted wreath product $S:= \N\wr \N = (\bigoplus_\N \N) \rtimes \N$, where the endomorphism appearing in the semidirect product is the shift. Then one can check that
$A:= S_c = \{(m,0) \mid m \in \bigoplus_\N \N\}$ and $U:= S_{ci}^1 = \{ (m,n) \mid n \in \N, m \in \bigoplus_\N \N: m_k=0 \text{ for all } k \geq n\}$
with action and restriction determined by $(m,0)(m',n) = (m+m',n) = (m'',n)((m_{\ell+n})_{\ell \in \N},0)$ with $m''_k = \chi_{\{0,\ldots,n-1\}}(k)~(m_k+m'_k)$ yields a Zappa-Sz\'{e}p product description $S=U\bowtie A$. It is also straightforward to check that $U\bowtie A$ is right LCM and satisfies \eqref{eq:cap condition ZS-product}. However, the constructible ideals of $A$ are directed but not totally ordered by inclusion, and the action of $A$ on $U$ is by non-surjective injections for all $a \in A\setminus\{0\}$. Note that the same treatment applies to $S' = T\wr \N$ for every nontrivial, left cancellative, left reversible monoid $T$.
\end{example}

\begin{thm}\label{thm:ZS-products with cap condition}
If $S=U \bowtie A$ is a right LCM semigroup with $A$ right LCM such that \eqref{eq:cap condition ZS-product} holds, then
\begin{enumerate}[(i)]
\item $A_c = A$, i.e. $A$ is left reversible,
\item $S_c$ equals $U_c \bowtie A$,
\item $(u,a) \sim (v,b)$ if and only if $u \sim v$,
\item $S_{ci}$ is given by $U_{ci} \bowtie A^*$,
\item $S$ is core factorable if and only $U$ is core factorable, and
\item $S_{c_i}\subset S$ is $\cap$-closed if and only if $U_{c_i}\subset U$ is $\cap$-closed.
\end{enumerate}
\end{thm}
\begin{proof}
Part (i) follows from \eqref{eq:cap condition ZS-product} as $(1,a)S \cap (1,b)S = (1,c)S$ for some $c \in A$, so that $c \in aA \cap bA$. Similarly, \eqref{eq:cap condition ZS-product} forces $S_c = U_c \times A$ as $A=A_c$. In addition, to get a Zappa-Sz\'{e}p product as claimed in (ii), we note that $a(u) \in U_c$ for all $a \in A, u \in U_c$ because $(a(u),a|_u) = (1,a)(u,1) \in S_c = U_c \times A$.

Part (iii) is an easy consequence of (ii): $(u,a) \sim (v,b)$ forces $u \sim v$ since $a(w) \in U_c$ for all $a \in A,w \in U_c$. Conversely, let $uw= vw'$ for some $w,w' \in U_c$ and let $a,b\in A$. Then $(1,a)(w,1),(1,b)(w',1) \in S_c$ and hence there are $s,s',t,t' \in S_c$ such that $(1,a)s'=ws$ and $(1,b)t'=w't$. Likewise, there are $s'',t'' \in S_c$ such that $ss''=tt''$. This leads us to 
\[(u,1)(1,a)s's'' = (u,1)(w,1)ss''= (uw,1)tt''=(vw',1)tt''=(v,b)t't'',\]
which proves $(u,a) \sim (v,b)$ as $s's'',tt'' \in S_c$.

For (iv), we first show that $S_{ci} = U_{ci} \times A^*$. Suppose we have $(u,a) \in S_{ci}$. If $a \in A\setminus A^*$, then $(u,a) = (u,1)(1,a)$ together with (ii) shows that $(u,a)$ is not core irreducible. Thus we must have $a \in A^*$. Similarly, if $u \notin U_{ci}$, then $u \in U^*$ or there exist $v \in U, w \in U_c\setminus U^*$ such that $u=vw$. In the first case, we get $(u,a) \in S^*$ and hence $(u,a) \notin S_{ci}$. The latter case yields a factorization $(u,a) = (v,1)(w,a)$ with $(w,a) \in S_c\setminus S^*$. Altogether, this shows $S_{ci} \subset U_{ci} \times A^*$.
To prove the reverse containment, suppose $(u,a) \in U_{ci} \times A^*$ and $(u,a)=(v,b)(w,c)=(v b(w),b|_wc)$ with $(v,b) \in S, (w,c) \in S_c=U_c \bowtie A$. Then $b|_w, c \in A^*$ as $a \in A^*$, and the Zappa-Sz\'{e}p product structure on $S_c$ implies $b(w) \in U_c$. Since $u \in U_{ci}$, we conclude that $b(w) \in U^*$. But then $(1,b)(w,1) = (b(w),b|_w) \in U^* \bowtie A^* = S^*$ forces $(1,b),(w,1) \in S^*$. In particular, we have $w \in U^*$ and hence $(w,c) \in S^*$. Thus $(u,a) \in S_{ci}$ and we have shown that $S_{ci} = U_{ci} \times A^*$.

Next we prove that $a(u) \in U_{ci}$ and $a|_u \in A^*$ whenever $a \in A^*, u \in U_{ci}$. The second part is easily checked with $a|_u^{-1} = a^{-1}|_{a(u)}$. So suppose $a(u) = vw$ with $w \in U_c$. Then we get $u=a^{-1}(v)a^{-1}|_{v}(w)$, and $S_c = U_c \bowtie A$ yields $a^{-1}|_{v}(w) \in U_c$. Since $u$ is core irreducible, we get $a^{-1}|_{v}(w) \in U^*$, and then $w \in U^*$ due to $S^*=U^* \bowtie A^*$. Therefore, $a(u) \in U_{ci}$ and we have proven (iv).

For (v), let $U=U_{ci}^1U_c$. Observe that (ii) and (iv) yield $S_{ci}^1S_c = (U_{ci} \bowtie A^*)^1(U_c \bowtie A) \supseteq U_{ci}^1U_c \times A \supseteq S$. Thus $S$ is core factorable.
Conversely, suppose that $S=S_{ci}^1S_c$. For $u \in U$, take $(u,1)=(v,a)(w,b)=(va(w),a|_wb)$ with $(v,a) \in S_{ci}^1= (U_{ci} \bowtie A^*)^{1}$ and $(w,b) \in S_c = U_c \bowtie A$. For $u \in U\setminus U_c$ we get $v \in U_{ci}$. Combining this with $a(w) \in U_c$ using (ii), we arrive at $u=va(w) \in U_{ci}U_c$. Thus $U = U_{ci}^1U_c$ holds.

For (vi), suppose that $S_{ci}\subset S$ is $\cap$-closed. Let $u,v \in U_{ci}$ such that $uU \cap vU = wU$ for some $w \in U$. Therefore $(u,1)S \cap (v,1)S = (w,1)S$, which by (iv) means that $w \in U_{ci}$. Conversely, suppose that $U_{ci}\subset U$ is $\cap$-closed. As $S_{ci}= U_{ci}\bowtie A^*$ by (iv) and $sxS = sS$ for all $s \in S, x \in S^*$, it suffices to consider $(u,1)S \cap (v,1)S = (w,1)S$ for $u,v \in U_{ci}$. If this holds, then $uU\cap vU = w'U$ for some $w'\in U$ with $w \in w'U$. As $U_{ci}\subset U$ is $\cap$-closed, $w' \in U_{ci}$. Moreover, \eqref{eq:cap condition ZS-product} implies that $(w,1)S= (u,1)S\cap (v,1)S = (w',a)S$ for some $a \in A$, which amounts to $a \in A^*$ and, more importantly, $w \in w'U^* \subset U_{ci}$. Thus $(w,1) \in S_{ci}$ by (iv).
\end{proof}

\begin{proposition}\label{prop:reduction for (I)-(VII) from UxA to U}
Let $S=U \bowtie A$ be a right LCM semigroup with $A$ right LCM such that \eqref{eq:cap condition ZS-product} holds. The restriction $N\mapsto N|_U$ defines a one-to-one correspondence between
\begin{enumerate}[(a)]
\item generalised scales $N\colon S \to \N^\times$; and
\item generalised scales $N'\colon U \to \N^\times$ satisfying $N'_{a(u)}=N'_u$ for all $a \in A, u \in U$.
\end{enumerate}
Moreover, $N(S)=N|_U(U)$, so that \ref{cond:A4} holds for $N$ if and only if it holds for $N|_U$.
\end{proposition}
\begin{proof}
For every generalised scale $N$ on $S$, the restriction $N|_U\colon U\to\N^\times$ defines a non-trivial homomorphism. Since $S_c = \ker~N$, Theorem~\ref{thm:ZS-products with cap condition}~(ii) implies that $N(S)=N|_U(U)$ as $N((u,a))=N((u,1))N((1,a))=N((u,1))$ for all $(u,a)\in S$, and $N|_U(a(u))=N((a(u),1))=N((a(u),a|_u))=N((1,a))N((u,1))=N|_U(u)$ for all $a \in A, u \in U$.

By virtue of Theorem~\ref{thm:ZS-products with cap condition}~(iii), \ref{cond:A3a} passes from $N$ to $N|_U$. Similarly, if $F'$ is a transversal for $N|_U^{-1}(n)/_\sim $ for $n\in N|_U(U)$, then Theorem~\ref{thm:ZS-products with cap condition}~(iii) implies that $F:= \{(u,1)|u\in F'\}$ is a transversal for $N^{-1}(n)/_\sim $. Due to \ref{cond:A3b} for $N$, we know that $F$ is an accurate foundation set for $S$. But then $F'$ is an accurate foundation set because of \eqref{eq:cap condition ZS-product}. Thus we get \ref{cond:A3b} for $N|_U$.

Conversely, suppose that $N'$ is one of the generalised scales in (b). Then $N\colon S \to \N^\times,(u,a) \mapsto N'(u)$ defines a non-trivial homomorphism on $S$ with $N|_U=N'$ because $N((u,a)(v,b)) = N((ua(v),a|_vb)) = N'(ua(v)) = N'(u)N'(v) = N((u,a))N((v,b))$ for $(u,a),(v,b) \in S$ by the invariance of $N'$ under the action of $A$. Similar to the first part, Theorem~\ref{thm:ZS-products with cap condition}~(iii) shows that \ref{cond:A3a} passes from $N'$ to $N$. Now if $F$ is a transversal of $N^{-1}(n)/_\sim$ for some $n\in N(S)$, then $F_U := \{ u \in U \mid (u,a) \in F \text{ for some } a \in A\}$ is a transversal for $N'^{-1}(n)/_\sim$, by Theorem~\ref{thm:ZS-products with cap condition}~(iii). Since $F_U$ is an accurate foundation set for $U$ due to \ref{cond:A3b} for $N'$, \eqref{eq:cap condition ZS-product} implies that $F$ is an accurate foundation set for $S$, i.e.~\ref{cond:A3b} holds for $N$.
\end{proof}

\begin{corollary}\label{cor:reduction from UxA to U}
Let $S=U \bowtie A$ be a right LCM semigroup with $A$ right LCM such that \eqref{eq:cap condition ZS-product} holds. Then $S$ is admissible if and only if $U$ is admissible with a generalised scale that is invariant under the action of $A$ on $U$
\end{corollary}
\begin{proof}
This follows from Theorem~\ref{thm:ZS-products with cap condition}~(v),(vi) and Proposition~\ref{prop:reduction for (I)-(VII) from UxA to U}.
\end{proof}

Observe that for $S=U\bowtie A$ with $S_c= U_c\bowtie A$, the action of $A$ on $U$ induces a well-defined action $\gamma\colon A\times U/_\sim \to U/_\sim, (a,[u])\mapsto [a(u)]$.

\begin{corollary}\label{cor:reduction from UxA to U - alpha}
Suppose that $S=U \bowtie A$ is a right LCM semigroup with $U$ right LCM satisfying Theorem~\ref{thm:ZS-products with cap condition}~(ii),(iii). If $U_c$ is trivial, then the following statements hold:
\begin{enumerate}[(i)]
\item The action $\alpha\colon S_c \curvearrowright S/_\sim$ is faithful if and only if $\gamma\colon A\curvearrowright U/_\sim$ is faithful.
\item The action $\alpha$ is almost free if and only if $\gamma$ is almost free.
\end{enumerate}
\end{corollary}
\begin{proof}
We start by noting that $\alpha_{(1,a)}[(u,c)] = [(a(u),a|_uc)]= [b(u),b|_uc)] = \alpha_{(1,b)}[(u,c)]$ if and only if $\gamma_a[u] = [a(u)]=[b(u)]=\gamma_b[u]$ for all $a,b,c \in A, u \in U$ by Theorem~\ref{thm:ZS-products with cap condition}~(ii) and (iii). If $U_c$ is trivial the canonical embedding $A \to S_c = U_c\bowtie A$ is an isomorphism which is equivariant for $\alpha$ and $\gamma$.
\end{proof}

\subsection{Easy right-angled Artin monoids}\label{subsec:examples easy right-angled Artin monoids}
For $m,n \in \N \cup \infty$, consider the direct product $S := \IF_m^+ \times \N^n$, where $\IF_m^+$ denotes the free monoid in $m$ generators and $\N^n$ denotes the free abelian monoid in $n$ generators. The monoid $S$ is a very simple case of a right-angled Artin monoid arising from the graph $\Gamma=(V,E)$ with vertices $V = V_m \sqcup V_n$, where $V_m := \{v_k \mid 1 \leq k \leq m\}$ and $V_n:= \{ v'_k \mid 1 \leq k \leq n\}$, and edge set $E = \{ (v,w) \mid v \in V_n \text{ or } w \in V_n\}$. In particular, $(S,\IF_m\times \Z^n)$ is a quasi lattice-ordered group, see \cite{CrispLaca}, so $S$ is right LCM.

We note that $\IF_m^+$ admits a generalised scale if and only if $m$ is finite, and thus $\IF_m^+$ is admissible if and only if $2\leq m < \infty$.

\begin{proposition}\label{prop:F_m x N^n example}
For $2\leq m < \infty, 0 \leq n \leq \infty$, the right LCM semigroup $S := \IF_m^+ \times \N^n$ has the following properties:
\begin{enumerate}[(i)]
\item $S$ is admissible with $\beta_c=1$.
\item The action $\alpha$ is faithful if and only if $\alpha$ is almost free if and only if $n=0$.
\item $S$ has finite propagation.
\end{enumerate}
\end{proposition}
\begin{proof}
By appealing to Corollary~\ref{cor:reduction from UxA to U}, $S = \IF_m^+ \bowtie \N^n$ with trivial action and restriction is admissible for the generalised scale $N$ induced by the invariant generalised scale $N'\colon \IF_m^+ \to \N^\times, w \mapsto m^{\ell(w)}$, where $\ell(w)$ denotes the word length of $w$ with respect to the canonical generators of $\IF_m^+$. In particular, $\text{Irr}(N(S)) = \{m\}$, so $\beta_c=1$.

Due to Corollary~\ref{cor:reduction from UxA to U - alpha}, we further know that $\alpha$ is faithful (almost free) if and only if $\gamma\colon \N^n\curvearrowright \IF_m^+/_\sim$ is faithful (almost free). But $\gamma$ is trivial as the action of $\N^n$ on $\IF_m^+$ is trivial. Thus we get faithfulness if only if $S_c\cong \N^n$ is trivial. In this case, the action is of course also almost free.

The only transversal $\CT \subset S_{ci}^1 = \IF_m^+$ for $S/_\sim$ is $\IF_m^+$. For every $a \in S_c = \IN^n$, we get $c(aw) = c(wa) = a$ so that $C_a = \{a\}$ is finite. Hence $S$ has finite propagation.
\end{proof}

In the case of $S := \IF_m^+ \times \N^n$ with $2 \leq m < \infty$, it is easy to see that $\kms_1$-states correspond to traces on $C^*(\IN^n)$, or, equivalently, Borel probability measures on $\IT^n$ via $C^*(\Z^n) \cong C(\IT^n)$. Thus we cannot expect uniqueness unless $n=0$. In fact, this example describes a scenario where $\alpha$ is degenerate. However, the situation is much less clear for general right-angled Artin monoids.

\subsection{Self-similar actions}\label{subsec:SSA}
Given a finite alphabet $X$, let $X^n$ be the set of words with length $n$. The set $X^*:=\cup_{n=0}^{\infty}X^n$ is a monoid with concatenation of words as the operation and empty word $\varnothing$ as the identity. A self-similar action $(G,X)$ is an action of a group $G$ on $X^*$ such that for every $g \in G$ and $x \in X$ there is a uniquely determined element $g|_x \in G$ satisfying
\[ g(xw) = g(x)g|_x(w) \quad \text{for all } w \in X^*.\]
We refer to \cite{Nek1} for a thorough treatment of the subject. It is observed in \cite{Law1} and \cite{BRRW}*{Theorem~3.8} that the Zappa-Sz\'{e}p product $X^* \bowtie G$ with $(v,g)(w,h) := (vg(w),g|_wh)$ defines a right LCM semigroup that satisfies \eqref{eq:cap condition ZS-product}.

\begin{proposition}\label{prop:self-similar action}
Suppose that $(G,X)$ is a self-similar action. Then the Zappa-Sz\'{e}p product $S:=X^* \bowtie G$ has the following properties:
\begin{enumerate}[(i)]
\item $S$ is admissible with $\beta_c=1$.
\item The action $\alpha$ is faithful (almost free) if and only if $G \curvearrowright X^*$ is faithful (almost free).
\item $S$ has finite propagation if and only if for all $g^{(1)},g^{(2)} \in G$ there exists a family $(h_w)_{w \in X^*} \subset G$ such that $\{h_{g^{(i)}(w)}^{-1}g^{(i)}|_w h_w \mid w \in X^*\}$ is finite for $i=1,2$. In particular, $S$ has finite propagation if $(G,X)$ is finite state.
\end{enumerate}
\end{proposition}
\begin{proof}
For (i), note that $S$ is a right LCM semigroup that satisfies \eqref{eq:cap condition ZS-product}. Then by Corollary~ \ref{cor:reduction from UxA to U}, it suffices to show that $X^*$ is admissible with an invariant generalised scale $N'\colon X^* \to \N^\times$. Since the core subsemigroup of $X^*$ consists of the empty word $\varnothing$ only, the equivalence relation $\sim$ is trivial on $X^*$. It follows that $X^*$ satisfies the conditions \ref{cond:A1} and \ref{cond:A2}. It is apparent that the map $N'\colon X^* \to \N^\times, w \mapsto \lvert X\rvert^{\ell(w)}$ is a generalised scale, where $\ell(w)$ denotes the \emph{word length} of $w$ in the alphabet $X$. Since for each $g\in G$, $w\mapsto g(w)$ is a bijection on $X^n$ (see for example \cite{BRRW}*{Lemma~3.7}), $N'$ is invariant under $G\curvearrowright X^*$. Finally, \ref{cond:A4} follows from $\text{Irr}(N'(X^*)) = \{\lvert X \rvert\}$. Thus $N'$ induces a generalised scale $N$ on $S$ that makes $S$ admissible. Moreover, we have $\beta_c=1$ as $\text{Irr}(N'(X^*)) = \{\lvert X \rvert\}$ is finite.

Part (ii) is an immediate consequence of Corollary~\ref{cor:reduction from UxA to U - alpha}.

For (iii), we observe that an arbitrary transversal $\CT\subset S_{ci}^1 = (X^*\setminus\{\varnothing\}\bowtie G)^1$ for $S/_\sim$ is given by $((w,h_w))_{w \in X^*}$, where $(h_w)_{w \in X^*} \subset G$ satisfies $h_\varnothing=1$. For a given $g \in G \cong S_c$, we get $C_g := C_{(\varnothing,g)} = \{h_{g(w)}^{-1}g|_w h_w \mid w \in X^*\}$ as $(\varnothing,g)(w,h_w) = (g(w),h_{g(w)})(\varnothing,h_{g(w)}^{-1}g|_w h_w) \in \CT S_c$. This establishes the characterisation of finite propagation. So if $(G,X)$ is finite state, i.e.~$\{ g|_w \mid w \in X^*\}$ is finite for all $g \in G$, then $h_w =1$ for all $w \in X^*$ yields finite propagation for $S$.
\end{proof}

Faithfulness of $G\curvearrowright X^*$ is assumed right away in \cite{LRRW}*{Section~2}. Our approach shows that this is unnecessary unless we want a unique $\kms_1$-state.

\subsection{\texorpdfstring{Subdynamics of $\N \rtimes \N^\times$}{Subdynamics of the ax+b-semigroup over the natural numbers}}\label{subsec:NxNtimes}
Consider the semidirect product $\N \rtimes \N^\times$ with
\[(m,p)(n,q) = (m+np,pq) \text{ for all }m,n\in \N \text{ and } p,q\in \N^\times.\]
 Let $\CP\subset \N^\times$ be a family of relatively prime numbers, and $P$ the free abelian submonoid of $\N^\times$ generated by $\CP$. Then we get a right LCM subsemigroup $S:=\N \rtimes P$ of $\N \rtimes \N^\times$. As observed in \cite{Sta3}*{Example~3.1}, $S$ can be displayed as the internal Zappa-Sz\'{e}p product
$S_{ci}^1 \bowtie S_c$, where $S_c = \N \times \{1\}$ and $S_{ci}^1 = \{ (n,p) \in S \mid 0 \leq n \leq p-1\}$. The action and restriction maps
satisfy
\begin{align*}
(m,1)\big((n,p)\big)=\big((m+n)\,(\text{mod } p),p\big), \text{ and }
\end{align*}
\[(m,1)|_{(n,p)}=\Big(\frac{m+n-\big((m+n)\,(\text{mod } p)\big)}{p},1\Big).\]
 Both $S_{ci}^1$ and $S_c$ are right LCM. Also the argument of the paragraph after \cite{BRRW}*{Lemma~3.5} shows that $S_{ci}^1 \bowtie S_c$ satisfies the hypothesis of \cite{BRRW}*{Lemma~3.3}. Therefore, $S_{ci}^1 \bowtie S_c$ is a right LCM semigroup satisfying \eqref{eq:cap condition ZS-product}, see Remark~\ref{rem:examples of cap condition}.

\begin{proposition}\label{prop:subdynamics of NxNtimes}
Suppose that $\CP\subset \N^\times$ is a nonempty family of relatively prime numbers. Then the right LCM semigroup $S= \N \rtimes P$ has the following properties:
\begin{enumerate}[(i)]
\item $S$ is admissible, and if $\CP$ is finite, then $\beta_c=1$.
\item $\alpha$ is faithful and almost free.
\item $S$ has finite propagation.
\end{enumerate}
\end{proposition}
\begin{proof}
Since $S$ is isomorphic to $S_{ci}^1 \bowtie S_c$, (i) follows if $S_{ci}^1$ is admissible with an invariant generalised scale, see Corollary~ \ref{cor:reduction from UxA to U}. To see this, first note that the core subsemigroup for $S_{ci}^1$ is trivial, hence so is the relation $\sim$ for $S_{ci}^1$. Thus \ref{cond:A1} and \ref{cond:A2} hold for $S_{ci}^1$. We claim that $N'\colon S_{ci}^1 \to \N^\times,(n,p) \mapsto p $ defines a generalised scale on $S_{ci}^1$. The only transversal for $(N')^{-1}(p)/_\sim$ for $p\in N'(S_{ci}^1)$ is $(N')^{-1}(p) = \{ (m,p) \mid 0 \leq m \leq p-1\}$. This gives condition \ref{cond:A3a}. For \ref{cond:A3b}, we observe that
\begin{equation*}
(m,p) \perp (n,q) \text{ if and only if } m-n \notin \text{gcd}(p,q)\Z \text{ for all } (m,p),(n,q) \in S_{ci}^1.
\end{equation*}
It follows $(N')^{-1}(p)$ is an accurate foundation set. Thus $N'$ is a generalised scale on $S_{ci}^1$. Since $S_c = \N \times \{1\}$, the action of $S_c$ on $S_{ci}^1$ fixes the second coordinate. Therefore, $N'$ is invariant under the action of $S_c$ on $S_{ci}^1$. Finally, note that the set $\text{Irr}(N'(S_{ci}^1))$ is given by $\CP$, which is assumed to consist of relatively prime elements. Hence \ref{cond:A4} holds. Moreover, if $\CP$ is finite, then $\beta_c=1$ by the product formula for $\zeta_S$ from Remark~\ref{rem:product formula for zeta function-gen}.

For (ii), note that Corollary~\ref{cor:reduction from UxA to U - alpha} applies. As $S_c \cong \N \times \{1\}$ is totally ordered, it suffices to examine the fixed points in $S/_\sim$ for $\gamma_m$ with $m \geq 1$, see Remark~\ref{rem:strong minimality}. But $\gamma_m [(n,p)] = [(r(m+n),p)]$, where $0 \leq r(m+n) \leq p-1, m+n \in r(m+n) + p\N$, so that $[(n,p)]$ is fixed by $\gamma_m$ if and only if $m \in p\N$. As $\{p \in P \mid m \in p\N\}$ is finite, $\gamma$ is almost free. In particular, $\gamma$ is faithful because $S_{ci}^1/_\sim$ is infinite, see Remark~\ref{rem:almost free actions}. Hence $\alpha$ is faithful and almost free.

Part (iii) follows from the observation that the transversal $\CT = S_{ci}^1$ for $S/_\sim$ satisfies $C_{(m,1)} = \{(m,1)\} \cup \{ (n,1) \mid n \in \{k,k+1\} \text{ with } kp \leq m < (k+1)p \text{ for some } p \in P\setminus\{1\}\}$ for $m \geq 1$. Thus the cardinality of $C_{(m,1)}$ does not exceed $m+1$, and $S$ has finite propagation.
\end{proof}

By Proposition~\ref{prop:subdynamics of NxNtimes}, Theorem~\ref{thm:KMS results-gen} applies to $\N\rtimes \N^\times$. This recovers the essential results on $\kms$-states from \cites{LR2,BaHLR}.

\subsection{Baumslag-Solitar monoids}\label{subsec:BS-monoid}
Let $c,d \in \N$ with $cd > 1$. The \emph{Baumslag-Solitar monoid} $BS(c,d)^+$ is the monoid in two generators $a,b$ subject to the relation $ab^c=b^da$. It is known that $S:=BS(c,d)^+$ is quasi-lattice ordered, hence right LCM, see \cite{Spi1}*{Theorem~2.11}. It is observed in \cite{Sta3}*{Example~3.2} and \cite{BRRW}*{Subsection~3.1} that the core subsemigroup $S_c$ is the free monoid in $b$, whereas $S_{ci}^1$ is the free monoid in $\{b^ja \mid 0 \leq j \leq d-1\}$. In addition, we have $S_{ci}^1 \cap S_c = \{1\}$ and $S=S_{ci}^1S_c$, so that the two subsemigroups of $S$ give rise to an internal Zappa-Sz\'{e}p product description $S_{ci}^1 \bowtie S_c$. The action and restriction map (on the generators) are given by
\[\begin{array}{rcl}
b(b^ja) = \begin{cases}
b^{j+1}a &\text{if } j<d-1, \\
a&\text{if } j=d-1, \end{cases} &\text{ and }&
b|_{b^ja} = \begin{cases}
1 &\text{if } j<d-1, \\
b^c&\text{if } j=d-1.\end{cases}
\end{array}\]
Moreover, $S_{ci}^1 \bowtie S_c$ satisfies the hypothesis of \cite{BRRW}*{Lemma~3.3} and therefore $S_{ci}^1 \bowtie S_c$ is a right LCM satisfying \eqref{eq:cap condition ZS-product}, see Remark~\ref{rem:examples of cap condition}.

\begin{proposition}\label{prop:BS-monoid}
Let $c,d \in \N$ such that $cd > 1$. Then the Baumslag-Solitar monoid $S=BS(c,d)^+$ is admissible with $\beta_c=1$, and the following statements hold:
\begin{enumerate}[(i)]
\item The action $\alpha$ is almost free if and only if $\alpha$ is faithful if and only if $c \not\in d\N$.
\item $S$ has finite propagation if and only if $c \leq d$.
\end{enumerate}
\end{proposition}
\begin{proof}
Since $S\cong S_{ci}^1 \bowtie S_c$, we can invoke Corollary~\ref{cor:reduction from UxA to U}. As $S_{ci}^1$ is the free monoid in $d$ generators $X := \{ w_j:= b^ja \mid 0 \leq j \leq d-1\}$, it is admissible for the generalised scale $N'\colon S_{ci}^1 \to \N^\times, w \mapsto d^{\ell(w)}$, where $\ell(w)$ denotes the word length of $w$ with respect to $X$, compare Proposition~\ref{prop:self-similar action}. The action of $S_c$ on $S_{ci}^1$ preserves $\ell$, and hence $N'$ is invariant. Thus $N'$ induces a generalised scale $N$ on $S$ that makes $S$ admissible. Due to $\text{Irr}(N(S)) = \{d\}$, we get $\beta_c=1$.

The quotient $S/_\sim$ is infinite, almost freeness implies faithfulness by Remark~\ref{rem:almost free actions}. As $S\cong S_{ci}^1 \bowtie S_c$ satisfies \eqref{eq:cap condition ZS-product}, we can appeal to Corollary~\ref{cor:reduction from UxA to U - alpha}, and thus consider $\gamma\colon\N\curvearrowright S_{ci}^1/_\sim$ in place of $\alpha$, where we identify the quotient by $\sim$ with $S_{ci}^1$ because $\sim$ is trivial on $S_{ci}^1$. As $S_c\cong\N$ is totally ordered, it suffices to consider fixed points for $\gamma_n$ with $n \geq 1$, see Remark~\ref{rem:strong minimality}. Next, we observe that
\[\gamma_n(w) = w \text{ for } w \in S_{ci}^1 \text{ if and only if } n(c/d)^{k-1} \in d\N \text{ for all } 1 \leq k \leq \ell(w).\]
So if $c \in d\N$, then $\gamma_d = \text{id}$, so that $\gamma$ is not faithful. Hence $\alpha$ is not faithful as well. On the other hand, if $c \not\in d\N$, then $e:= \gcd(c,d)^{-1}d \geq 2$ and $\gamma_n$ fixes $w \in S_{ci}^1$ if and only if $n \in e^\ell(w)\N$. So as soon as $\ell(w)$ is large enough, $w$ cannot be fixed by $\gamma_n$. Since the alphabet $X$ is finite, we conclude that $\gamma$ (and hence $\alpha$) is almost free if $c \not\in d\N$.

For (ii), we remark that $S_{ci}^1$ is the only transversal for $S/_\sim$ that is contained in $S_{ci}^1$. For $m \in \N$, we let $m=kd+r$ with $k,r \in \N$ and $0\leq r\leq d-1$. Then we get
\begin{equation}\label{eq:BS finite propagation}
b^mw_i = w_jb^n \text{ for } 0 \leq i \leq d-1 \text{ with } n=\begin{cases} kc,&\text{ if } 0 \leq i < d-r,\\ (k+1)c,&\text{ if } d-r \leq i \leq d-1,\end{cases}
\end{equation}
with a suitable $0\leq j \leq d-1$. If $c=d$, then we get $C_{b^m} = \{ b^m\}$ for $m \in d\N$ and $C_{b^m}=\{b^m,b^{m+d-r},b^{m-r}\}$ for $m \not\in d\N$. Thus $S$ has finite propagation in this case. For $c<d$, it follows from \eqref{eq:BS finite propagation} that every expression $b^mv = wb^n$ for given $v \in S_{ci}^1$ and $m \in \N$ satisfies $n \leq (k+1)d$. Hence $C_{b^m}$ is finite and $S$ has finite propagation. Now suppose $c>d$, and let $m \geq d$, say $m=kd + r$ with $0\leq r\leq d-1, k \geq 1$. If $r=0$, \eqref{eq:BS finite propagation} gives $b^mw_i = w_ib^n$ with $n=kc>m$ for all $i$. If $r>0$, then there is $d-r \leq i \leq d-1$, and we have $b^mw_i = w_jb^n$ with $n=(k+1)c>m$ for all such $i$. Thus $C_{b^m}$ contains arbitrarily large powers of $b$, and hence must be infinite. Therefore $S$ does not have finite propagation for $c \geq d$.
\end{proof}

We note that for this example, finite propagation can be regarded as a consequence of faithfulness of $\alpha$. The condition that $d$ is not a divisor of $c$ also appeared in \cite{CaHR}. More precisely, it was shown in \cite{CaHR}*{Corollary~5.3, Proposition~7.1, and Example~7.3} that this condition is necessary and sufficient for the $\kms_1$-state to be unique. Let us remark here that we do not need to request that $d$ is not a divisor of $c$ in order to obtain the parametrisation of the $\kms_\beta$-states for $\beta>1$, compare Theorem~\ref{thm:KMS results-gen} and Proposition~\ref{prop:BS-monoid} with \cite{CaHR}*{Theorem~6.1}.

\subsection{Algebraic dynamical systems}\label{subsec:ADS}
Suppose $(\gpt)$ is an algebraic dynamical system, that is, an action $\theta$ of a right LCM semigroup $P$ on a countable, discrete group $G$ by injective endomorphisms such that $pP \cap qP = rP$ implies $\theta_p(G) \cap \theta_q(G) = \theta_r(G)$. The semidirect product $\gxp$ is the semigroup $G \times P$ equipped with the operation $(g, p)(h, q) = (g\theta_p(h), pq)$. It is known that injectivity and the intersection condition, known as \emph{preservation of order}, are jointly equivalent to $\gxp$ being a right LCM semigroup, see \cite{BLS1}*{Proposition~8.2}.

In order to have an admissible semidirect product $\gxp$, we need to consider more constraints on $(\gpt)$:
\begin{enumerate}[(a)]
\item $(\gpt)$ is \emph{finite type}, that is, the index $N_p:= [G:\theta_p(G)]$ is finite for all $p \in P$.
\item There exists $p \in P$ for which $\theta_p$ is not an automorphism of $G$.
\item If $\theta_p$ is an automorphism of $G$, then $p \in P^*$.
\item If $p,q \in P$ have the same index, i.e.~$N_p=N_q$, then $p \in qP^*$. That is to say, there is an automorphism $\vartheta$ of $G$ such that $\theta_p$ and $\theta_q$ are conjugate via $\vartheta$, and $\vartheta = \theta_x$ for some $x \in P^*$.
\end{enumerate}
Conditions (a) and (b) are mild assumptions. Also it is natural to assume that $p$ is invertible in $P$ if $\theta_p$ is an automorphism of $G$. For otherwise, we can replace $P$ by the acting semigroup of endomorphisms of $G$. But (d) is significantly more complex, and we will see that it incorporates deep structural consequences for $\gxp$.

\begin{proposition}\label{prop:algebraic dynamical system}
Let $(\gpt)$ be an algebraic dynamical system satisfying the above conditions (a)--(d). Then $S:=\gxp$ has the following properties:
\begin{enumerate}[(i)]
\item $S$ is admissible, and $\beta_c=1$ if $\text{Irr}(P/_\sim)$ is finite.
\item The action $\alpha$ is faithful if and only if $\bigcap_{(h,q) \in S} h\theta_q(G)h^{-1} = \{1\} $, and
\[\begin{array}{c}\bigcap_{(h,q) \in S} h\theta_q(G)\theta_p(h)^{-1} = \emptyset \text{ for all }p \in P^*\setminus\{1\}.\end{array}\]
\item The action $\alpha$ is almost free if and only if $P^*$ is trivial and the set $\{ (g,p) \in S \mid h \in g\theta_p(G)g^{-1}\}$ is finite for all $h \in G\setminus\{1\}$.
\end{enumerate}
\end{proposition}
\begin{proof}
For (i), note that $S_c = S^* = G \rtimes P^*$, see \cite{Sta3}*{Example~3.2}. So that $S_{ci} = S\setminus S^*$ and therefore $S$ satisfies \ref{cond:A1} and \ref{cond:A2}. Also note that
\begin{align}\label{equ1:algebraic dynamical system}
(g,p) \sim (h,q) \Leftrightarrow (g,p)S=(h,q)S \Leftrightarrow p \in qP^* \text{ and } g^{-1}h \in \theta_p(G) \ (=\theta_q(G)).
\end{align}
Then the natural candidate for a homomorphism $N\colon S \to \N^\times$ is given by $N(g,p):= N_p$. Properties (a) and (b) imply that $N$ is a nontrivial homomorphism from $S$ to $\N^\times$. The condition \ref{cond:A3a} follows from (d) together with \eqref{equ1:algebraic dynamical system}. Since $(\gpt)$ is finite type, \ref{cond:A3b} holds if and only if $P$ is directed with regards to $p \geq q :\Leftrightarrow p \in qP$, see \cite{Sta3}*{Example~3.2}. This follows immediately from (d): Given $p,q \in P$, we have $N_{pq} = N_pN_q= N_{qp}$ so that $pqx=qp \in pP \cap qP$ for a suitable $x \in P^*$. Thus $N$ is a generalised scale on $S$.

Now we are looking at \ref{cond:A4}. First note that (d) implies $P^*p \subset pP^*$ for all $p \in P$. Thus the restriction of $\sim$ to $P$ given by $p \sim' q :\Leftrightarrow p \in qP^*$ defines a congruence on $P$. If $\overline{P} := P/_{\sim'}$ denotes the resulting monoid, then $\overline{P}$ is necessarily abelian as $N_{pq} = N_{qp}$. The monoid $\overline{P}$ is generated by $\{ [p] \in \overline{P} \mid N_p \in \text{Irr}(N(S)) \}$. This set is minimal as a generating set for $\overline{P}$ and coincides with the irreducible elements of $\overline{P}$. So to prove \ref{cond:A4}, it suffices to show that $\overline{P}$ is in fact freely generated by $\text{Irr}(\overline{P})$.

To see this, let $\prod_{1 \leq i \leq k} [p_i]^{m_i} = \prod_{1 \leq j \leq \ell} [q_j]^{n_j}$ with $[p_i],[q_j] \in \text{Irr}(\overline{P}), m_i,n_j \in \N$, where we assume that the $[p_i]$ are mutually distinct, as are the $[q_j]$. Then we get a corresponding equation $\prod_{1 \leq i \leq k} N_{p_i}^{m_i} = \prod_{1 \leq j \leq \ell} N_{q_j}^{n_j}$ in $N(S)$. Since $N(S)$ has the unique factorization property (inherited from $\N^\times$), and $N_{p_i},N_{q_j} \in \text{Irr}(N(S))$ for all $i,j$, we get $k=\ell$ and a permutation $\sigma$ of $\{1,\ldots,k\}$ such that $N_{p_i} = N_{q_{\sigma(i)}}$ and $m_i=n_{\sigma(i)}$. Now $N_{p_i} = N_{q_{\sigma(i)}}$ forces $q_{\sigma(i)} \in p_iP^*$, so $[q_{\sigma(i)}] = [p_i]$. Thus $\overline{P}$ is the free abelian monoid in $\text{Irr}(\overline{P})$.

We clearly have $\beta_c = 1$ whenever $\text{Irr}(\overline{P})$ is finite, as these are in one-to-one correspondence with $\text{Irr}(N(\gxp))$ and $\zeta(\beta) = \prod_{n \in \text{Irr}(N(\gxp))} (1-n^{-(\beta-1)})^{-1}$, see Remark~\ref{rem:product formula for zeta function-gen}.

For (ii), since $S_c=S^*$ is a group, faithfulness of $\alpha$ amounts to the following: For every $(g,p) \in S^*\setminus\{1\}$ there exists $(h,q) \in S$ such that $h^{-1}g\theta_p(h) \notin \theta_q(G)$. This is equivalent to $\bigcap_{(h,q) \in S} h\theta_q(G)h^{-1} = \{1\}$ and $\bigcap_{(h,q) \in S} h\theta_q(G)\theta_p(h)^{-1} = \emptyset$ for all $p \in P^*\setminus\{1\}$.

For (iii), since $S_c = S^*=G\rtimes_\theta P^*$ is a group, by Remark~\ref{rem:strong minimality}, it suffices to look at the fixed points of the map $\alpha_{(h,q)}$ for each element $(h,q) \in S_c\setminus\{1\}$. Since $P^*p \subset pP^*$ for all $p \in P$, the equation \eqref{equ1:algebraic dynamical system} implies that an equivalence class $[(g,p)]\in S/_\sim$ is fixed by $\alpha_{(h,q)}$ if and only if $g^{-1}h\theta_q(g) \in \theta_p(G)$. Now suppose that $P^* \neq \{1\}$. Fix $q \in P^*\setminus\{1\}$ and take $h=1$. As $qp=pq'$ for some $q' \in P^*$, we get $\theta_q(\theta_p(G))=\theta_p(G)$ for all $p \in P$. Thus the element $[(g,p)]\in S/_\sim$ with $g \in \theta_p(G)$ is fixed by $\alpha_{(1,q)}$ for every $p \in P$. By assumption, $\overline{P}$ is infinite as there is at least one $p \in P$ for which $\theta_p$ is not an automorphism of $G$. This tells us that $(1,q)$ violates the almost freeness for the action $\alpha$. It follows that if $\alpha$ is almost free, then $P^*$ has to be trivial. Now suppose that $P^*=\{1\}$. Then \eqref{equ1:algebraic dynamical system} allows us to conclude that $\alpha$ is almost free if and only if the set $\{ (g,p) \in S \mid h \in g\theta_p(G)g^{-1}\}$ is finite for all $h \in G\setminus\{1\}$.
\end{proof}

\begin{remark}\label{rem:min conditions for ADS for simplicity of Q(S) vs. min cond. for uniqueness of KMS-states}
Let $(\gpt)$ be an algebraic dynamical system and $S=\gxp$. Then the conditions for $\alpha$ to be faithful are closely related to the conditions (2) and (3) in \cite{BS1}*{Corollary~4.10} which are needed for the simplicity of $\CQ(S)$. In addition to that, the requirement $P^*p \subset pP^*$ for all $p \in P$ coming from property (d) in the Proposition~\ref{prop:algebraic dynamical system} is needed there as well. This intimate relation between conditions for simplicity of $\CQ(S)$ and uniqueness of the $\kms_\beta$-state for $\beta \in [1,\beta_c]$ seems quite intriguing, and may very well deserve a thorough study. To the best of our knowledge, no explanation for this phenomenon is known so far.

We also note that the characterisation of almost freeness of $\alpha$ for the case of trivial $P^*$ we get here is a natural strengthening of the condition for faithfulness of $\alpha$.
\end{remark}

The next example describes a class of algebraic dynamical systems for which $\beta_c=1$, $S$ has finite propagation, and $\alpha$ is faithful but not almost free.

\begin{example}\label{ex:ADS min but not str min}
Let $H$ be a finite group with $\lvert H \rvert >2$. Then there exists a non-trivial subgroup $P_0 \subset \Aut(H)$. Consider the standard restricted wreath product $G := H \wr \N = \bigoplus_{n \in \N} H$. Each $p \in P_0$ corresponds to an automorphism of $G$ that commutes with the right shift $\Sigma$ on $G$, and we denote by $P := \langle \Sigma \rangle \times P_0$ the corresponding semigroup of endomorphisms of $G$. Then $P$ is lattice ordered, hence right LCM, and the natural action $\theta\colon P \curvearrowright G$ respects the order so that $(\gpt)$ forms an algebraic dynamical system for which we can consider the right LCM semigroup $S = \gxp$. We note the following features:
\begin{enumerate}[(a)]
\item $(\gpt)$ is of finite type because $H$ is finite.
\item $P^*p$ equals $pP^*$ for all $p \in P$ since the chosen automorphisms of $G$ commute with $\Sigma$.
\item The action $\alpha$ is faithful.
\item As $P_0$ is nontrivial, $\alpha$ is not almost free, see Proposition~\ref{prop:algebraic dynamical system}.
\item Take $\CT = \bigcup_{n \in \N} \CT_n$ with $\CT_n := \{ ((h_k)_{k \geq 0},\Sigma^n) \in S \mid h_k = 1_H \text{ for all } k > n\}$. This is a natural transversal for $S/_\sim$ that witnesses finite propagation (for all pairs $(s,t) \in S\times S$). Indeed, for $s=(g,p) \in S_c$ we get $\lvert C_s\rvert \leq \ell(g)+1 < \infty$, where $\ell(g)$ denotes the smallest $n \in \N$ with $h_m = 1_H$ for all $m>n$.
\item $\beta_c=1$ as $\text{Irr}(P/_\sim) = \{[\Sigma]\}$.
\end{enumerate}
Thus we conclude that $S$ is an admissible right LCM semigroup that has a unique $\kms_1$-state by virtue of Theorem~\ref{thm:KMS results-gen}~(2b), while the condition in Theorem~\ref{thm:KMS results-gen}~(2a) does not apply.
\end{example}

As a byproduct of our treatment of algebraic dynamical systems, we recover the results for a dilation matrix from \cite{LRR}:

\begin{example}[Dilation matrix]\label{ex:dilation matrix}
Let $d \in \N$ and $A \in M_d(\Z)$ an integer matrix with $|\det A| > 1$, and consider $(\gpt)=(\Z^d,\N,\theta)$ with $\theta_n(m) := A^n m$ for $n \in \N, m \in \Z^d$. Then $(\gpt)$ constitutes an algebraic dynamical system that satisfies (a)--(d) from Proposition~\ref{prop:algebraic dynamical system} as
\[N_1 = \text{coker} \theta_1 = \lvert \Z^d/\text{im}A \rvert = |\det A| \in \N\setminus\{1\}.\]
Also, we have $\beta_c=1$ because $\text{Irr}(N(\Z^d \rtimes_\theta \N)) = \{|\det A|\}$. Since $\Z^d$ is abelian, the action $\alpha$ is faithful if and only if $\bigcap_{n \in \N} A^n\Z^d = \{0\}$, see Proposition~\ref{prop:algebraic dynamical system}~(ii). As $P \cong \N$, we deduce that this is also equivalent to almost freeness of $\alpha$, see Proposition~\ref{prop:algebraic dynamical system}~(iii).
\end{example}

\begin{example}\label{ex:solenoid}
Let $A,B \subset \N^\times$ be non-empty disjoint sets such that $A\cup B$ consists of relatively prime numbers. Then the free abelian submonoid $P$ of $\N^\times$ generated by $A$ acts on the ring extension $G:= \Z[ 1/q \mid q \in B]$ via multiplication by injective group endomorphisms $\theta_p$ in an order preserving way, that is, $(G,P,\theta)$ is an algebraic dynamical system. We observe that $N_p = [G: pG] = p$ (as for $B=\emptyset$) so that the system is of finite type, and it is easy to check that the requirements (a)--(d) for admissibility of $S=G\rtimes_\theta P$ hold. It thus follows from Proposition~\ref{prop:algebraic dynamical system}~(i) that $S$ is admissible and $\beta_c=1$ if $A$ is finite. Since $G$ is abelian and $P^*$ is trivial, parts (ii) and (iii) imply that faithfulness and almost freeness of $\alpha$ are both equivalent to $A$ being non-empty. The behaviour of the $\kms$-state structure on this semigroup resembles very much the one of $\N\rtimes \N^\times$, except for the fact that the parametrising space corresponds to Borel probability measures on the solenoid $\IT[1/q \mid q \in B]$ instead of the torus $\IT$.
\end{example}

\begin{example}\label{ex:polynomial ring}
For a finite field $F$, consider the polynomial ring $F[t]$ as an additive group $G$. Via multiplication, every $f \in F[t]\setminus F$ gives rise to an injective, non-surjective group endomorphism $\theta_f$ of $F[t]$, whose index is $N_f = \lvert F\rvert^{\text{deg}~f}$. For every such $f$, the semigroup $S:= F[t] \rtimes_f \N$ is an admissible right LCM semigroup. By Proposition~\ref{prop:algebraic dynamical system}, we further have $\beta_c=1$, and $\alpha$ is almost free, hence faithful. In the simplest case of $f=t$, the semigroup $F[t]\rtimes_f \N$ is canonically isomorphic to $(F \wr \N) \rtimes \N$ from Example~\ref{ex:ADS min but not str min}.
\end{example}

\section{Algebraic constraints}\label{sec:constraints}
This section presents some considerations on existence of $\kms_\beta$ and ground states, and in particular identifies some necessary conditions for existence. Throughout this section, we shall assume that $S$ is a right LCM semigroup that admits a generalised scale $N\colon S \to \N^\times$, and denote by $\sigma$ the time evolution on $C^*(S)$ given by $\sigma_x(v_s)=N_s^{ix}v_s^{\phantom{x}}$ for all $x\in \R$ and $s\in S$.

\begin{proposition}\label{prop:no KMS below 1-gen}
There are no $\kms_\beta$-states on $C^*(S)$ for $\beta < 1$.
\end{proposition}
\begin{proof}
Let $\beta\in\R$ and suppose that $\phi$ is a $\kms_\beta$-state for $C^*(S)$. By Proposition~\ref{prop:NisaHMofRLCMs}~(iii), there exists an accurate foundation set $F$ for $S$ with $\lvert F \rvert = n = N_f$ for all $f \in F$ and some $n > 1$. The $\kms_\beta$-condition gives that $\phi(v_f^{\phantom{*}}v_f^*)=N_f^{-\beta}\phi(v_f^*v_f^{\phantom{*}})= n^{-\beta}$. Hence, $ 1 = \phi(1) \geq \sum_{f \in F} \phi(e_{fS}) = n^{1-\beta}$, so necessarily $\beta \geq 1$.
\end{proof}

\begin{proposition}\label{prop:ground states on C*(S)-gen}
A state $\phi$ on $C^*(S)$ is a ground state if and only if $\phi(v_s^{\phantom{*}}v_t^*) \neq 0$ for $s,t\in S$ implies $s,t \in S_c$. In particular, a ground state on $C^*(S)$ vanishes outside $\varphi(C^*(S_c))$.
\end{proposition}
\begin{proof}
Let $\phi$ be a state on $C^*(S)$ and $y_i := v_{s_i}^{\phantom{*}}v_{t_i}^*$, with $s_i,t_i\in S$ for $i=1,2$, be generic elements of the spanning family for $C^*(S)$. The expression
\begin{equation}\label{eq:ground states proof-gen}
\begin{array}{c} \phi(y_2\sigma_{x + iy}(y_1)) = \left(\frac{N_{s_1}}{N_{t_1}}\right)^{-y +ix} \phi(y_2y_1) \end{array}
\end{equation}
is bounded on $\{x+iy \mid x \in \R, y > 0\}$ if and only if either $N_{s_1} \geq N_{t_1}$ or $N_{s_1} < N_{t_1}$ and $\phi(y_2y_1)=0$. Now assume $\phi$ is a ground state and suppose $\phi(v_s^{\phantom{*}}v_t^*) \neq 0$ for some $s,t \in S$. Choosing $s_2=s, t_1=t, s_1=t_2=1$ in \eqref{eq:ground states proof-gen}, we get $N_{t_1} \leq N_{s_1}=1$. Thus $N_t=1$, and hence $t \in S_c$ by Proposition~\ref{prop:NisaHMofRLCMs}~(i). Taking adjoints shows that $s \in S_c$ as well.

Conversely, let $\phi$ be a state on $C^*(S)$ for which $\phi(v_s^{\phantom{*}}v_t^*) \neq 0 $ implies $s,t \in S_c$. By the Cauchy-Schwartz inequality for $y_1$ and $y_2$ we have
\[\begin{array}{l}
|\phi(y_2\sigma_{x+iy}(y_1))|^2 = \left(\frac{N_{s_1}}{N_{t_1}}\right)^{-2y} |\phi(y_2y_1)|^2 \leq \left(\frac{N_{s_1}}{N_{t_1}}\right)^{-2y} \phi(y_2^{\phantom{*}}y_2^*)\phi(y_1^*y_1^{\phantom{*}}).
\end{array}\]
As $y_1^*y_1^{\phantom{*}} = e_{t_1S}$ and $y_2^{\phantom{*}}y_2^* = e_{s_2S}$, the expression on the right-hand side of the inequality vanishes unless $t_1,s_2 \in S_c$. So if it does not vanish we must have $N_{t_1}=1$. Therefore using that $N_{s_1}\geq 1$, it follows that $|\phi(y_2\sigma_{x+iy}(y_1))|^2 \leq 1$ for $y>0$. Thus $\phi$ is a ground state on $C^*(S)$.
\end{proof}

\begin{remark}\label{rem:range proj behave like 1 for traces}
For the purposes of the next result we need the following observation: if $\tau$ is a trace on $C^*(S_c)$, then $\tau(yw_a^{\phantom{*}}w_a^*z) = \tau(yz)$ for all $y,z \in C^*(S_c)$ and $a \in S_c$. The reason is that the image of the isometry $w_a$ in the GNS-representation for $\tau$ is a unitary.
\end{remark}

\begin{proposition}\label{prop:algebraic characterization of KMS states-gen}
For $\beta \in [1,\infty)$, consider the following conditions:
\begin{enumerate}[(i)]
\item $\phi$ is a $\kms_\beta$-state.
\item $\phi$ is a state on $C^*(S)$ such that $\phi \circ \varphi$ is a normalised trace on $C^*(S_c)$ and for all $s,t \in S$
\begin{equation}\label{eq:alg char of KMS states-gen}
\phi(v_s^{\phantom{*}}v_t^*) = \begin{cases}
0 &\text{if } s \not\sim t, \\
N_s^{-\beta}\phi \circ \varphi(w_{r'}^{\phantom{*}}w_{r\phantom{'}}^*) &\text{if } sr=tr' \text{ with } r,r' \in S_c.\end{cases}
\end{equation}
\end{enumerate}
Then (i) implies (ii). Assume moreover that $S$ is core factorable and $S_{ci}\subset S$ is $\cap$-closed, and fix maps $c$ and $i$ as in Lemma~\ref{lem:maps i and c}. If $\phi$ is a state on $C^*(S)$ such that $\phi \circ \varphi$ is a trace, then \eqref{eq:alg char of KMS states-gen} is equivalent to
\begin{equation}\label{eq:alg char of KMS states-gen with i and c}
\phi(v_s^{\phantom{*}}v_t^*) = \begin{cases}
0 &\text{if } i(s) \neq i(t), \\
N_s^{-\beta}\phi \circ \varphi(w_{c(s)}^{\phantom{*}}w_{c(t)}^*) &\text{if } i(s)=i(t), \end{cases}
\end{equation}
and (ii) implies (i).
\end{proposition}
\begin{proof}
First, let $\phi$ be a $\kms_\beta$-state on $C^*(S)$. Since $\sigma = \id$ on $\varphi(C^*(S_c))$, the state $\phi \circ \varphi$ is tracial. The $\kms$-condition gives
\[\begin{array}{l} \phi(v_s^{\phantom{*}}v_t^*) = N_s^{-\beta} \ \phi(v_t^*e_{tS}^{\phantom{*}}e_{sS}^{\phantom{*}}v_s^{\phantom{*}}) = N_s^{-\beta} N_t^\beta \ \phi(v_s^{\phantom{*}}v_t^*). \end{array}\]
Thus $\phi(v_s^{\phantom{*}}v_t^*) \neq 0$ implies both $N_s=N_t$ and $s \not \perp t$. Proposition~\ref{prop:NisaHMofRLCMs}~(ii) implies that $s \sim t$. Suppose first that $sS\cap tS=saS=tbS$ with $sa=tb$ for some $a,b\in S_c$. Since $e_{saS}\leq e_{sS}$, we have $v_t^*v_s^{\phantom{*}}=v_b^{\phantom{*}}v_a^*$ . Using the KMS$_\beta$ property of $\phi$ we obtain \eqref{eq:alg char of KMS states-gen} in the case that we may choose a right LCM of $s$ and $t$ to be of the form $sa=tb$. The general case $sr=tr'$ displayed in \eqref{eq:alg char of KMS states-gen} is then of the form $r=ac,r'=bc$ with arbitrary $c \in S_c$. Thus we get $\phi \circ \varphi(w_{r'}^{\phantom{*}}w_{r}^*) = \phi \circ \varphi(w_{b}^{\phantom{*}}w_{c}^{\phantom{*}}w_{c}^*w_{a}^*) = \phi \circ \varphi(w_{b}^{\phantom{*}}w_{a}^*)$ by Remark~\ref{rem:range proj behave like 1 for traces}.

Now suppose that \ref{cond:A1} and \ref{cond:A2} are satisfied as well, and that $\phi$ is a state on $C^*(S)$ for which $\phi \circ \varphi$ is tracial. Note that $s\sim t$ happens precisely when $i(s)=i(t)$ for all $s,t\in S$. The equivalence of \eqref{eq:alg char of KMS states-gen} and \eqref{eq:alg char of KMS states-gen with i and c} will follow if show that for all $r,r' \in S_c$ with $sr=tr'$ we have
 $\phi \circ \varphi(w_{r'}^{\phantom{*}}w_{r}^*) = \phi \circ \varphi(w_{c(s)}^{\phantom{*}}w_{c(t)}^*)$ if $i(s)=i(t)$. Fix therefore a pair $r,r'$ in $S$ such that $sr=tr'$ and $i(s)=i(t)$. We saw in the first part of the proof that $\phi \circ \varphi(w_{r'}^{\phantom{*}}w_{r}^*) = \phi \circ \varphi(w_{b}^{\phantom{*}}w_{a}^*)$ whenever $a,b$ in $S_c$ satisfy $sS \cap tS = saS, sa=tb$. By left cancellation and $i(s)=i(t)$, we reduce this to $c(s)S_c \cap c(t)S_c = c(s)aS_c, c(s)a = c(t)b$. But then the trace property of $\phi\circ \varphi$ yields
\[\phi \circ \varphi(w_{c(s)}^{\phantom{*}}w_{c(t)}^*) = \phi \circ \varphi(w_{c(t)}^*w_{c(s)}^{\phantom{*}}) = \phi \circ \varphi(w_{b}^{\phantom{*}}w_{a}^*) = \phi \circ \varphi(w_{r'}^{\phantom{*}}w_{r}^*).\]

 It remains to prove that if $\phi$ is a state on $C^*(S)$ for which $\phi \circ \varphi$ is tracial and \eqref{eq:alg char of KMS states-gen with i and c} holds, then $\phi$ is a $\kms_\beta$-state. As the family of analytic elements $\{v_s^{\phantom{*}}v_t^*\mid s,t \in S\}$ spans a dense $*$-subalgebra of $C^*(S)$, $\phi$ is a $\kms_\beta$-state if and only if
\begin{equation}\label{eq:alg char:spanning family test-gen}
\phi(v_{s_1}^{\phantom{*}}v_{t_1}^*v_{s_2}^{\phantom{*}}v_{t_2}^*) = N_{s_1}^{-\beta}N_{t_1}^\beta \phi(v_{s_2}^{\phantom{*}}v_{t_2}^*v_{s_1}^{\phantom{*}}v_{t_1}^*) \quad \text{for all $s_i,t_i \in S, i=1,2$}.
\end{equation}
We will first show that \eqref{eq:alg char:spanning family test-gen} is valid for $t_1=e$, i.e. we show that
\begin{equation}\label{eq:alg char:intermediate step}
\phi(v_s^{\phantom{*}}v_t^{\phantom{*}}v_r^*) = N_s^{-\beta} \phi(v_t^{\phantom{*}}v_r^*v_s^{\phantom{*}}) \text{ for all } r,s,t \in S.
\end{equation}
Before we prove this identity, we note that \eqref{eq:alg char of KMS states-gen with i and c} implies
\begin{equation}\label{eq:alg char of KMS states aux I}
\phi(v_s^{\phantom{*}}\varphi(a)v_s^*) = N_s^{-\beta}\phi \circ \varphi(a) \text{ for every } a \in C^*(S_c),
\end{equation}
since $(w_t^{\phantom{*}}w_r^*)_{t,r \in S_c}$ form a spanning family for $C^*(S_c)$. Now let $r,s,t \in S$ and compute, using the trace property of $\phi\circ \varphi$ in the second equality, that
\[\begin{array}{lcl}
\phi(v_s^{\phantom{*}}v_t^{\phantom{*}}v_r^*) &\stackrel{\eqref{eq:alg char of KMS states-gen with i and c}}{=}& N_{st}^{-\beta} \ \delta_{i(st), i(r)} \phi \circ \varphi (w_{c(si(t))}^{\phantom{*}}w_{c(t)}^{\phantom{*}}w_{c(r)}^*) \\
&=& N_{st}^{-\beta} \ \delta_{i(st), i(r)} \phi \circ \varphi (w_{c(t)}^{\phantom{*}}w_{c(r)}^*w_{c(si(t))}^{\phantom{*}}) \\
&\stackrel{\eqref{eq:alg char of KMS states aux I} \text{ for } i(t)}{=}& N_{s}^{-\beta} \ \delta_{i(st), i(r)} \phi (v_t^{\phantom{*}}v_{c(r)}^*v_{c(si(t))}^{\phantom{*}}v_{i(t)}^*).
\end{array}\]
Suppose that $i(st)=i(r)$. We claim that $v_{c(si(t))}^{\phantom{*}}v_{i(t)}^* = v_{i(r)}^*v_s^{\phantom{*}}$. To see this, we first show that $sS \cap i(r)S = si(t)S$. Since $si(t)=i(si(t))c(si(t))=i(st)c(si(t))$, we have $si(t)S\subseteq sS\cap i(r)S$. For the reverse containment, let $ss'$ be a right LCM for $s$ and $i(r)$. Then by Proposition~\ref{prop:NisaHMofRLCMs} (iv) we have $N_sN_s'=\text{lcm}(N_s,N_{i(r)})=\text{lcm}(N_s,N_sN_t)=N_sN_t$. Hence $N_s'=N_t$, and by Proposition~\ref{prop:NisaHMofRLCMs} (ii) we know that $i(s')=i(t)$ or $s'\perp t$. We cannot have $s'\perp t$ because this would contradict that $ss'\in i(st)S\subseteq stS$. Hence $i(s')=i(t)$, and so $ss'\in si(t)S$. This gives the reverse containment, and thus we have $sS \cap i(r)S = si(t)S$. Hence $si(t) = i(si(t))c(si(t))=i(st)c(si(t))=i(r)c(si(t))$ is a right LCM for $s$ and $i(r)$, and so gives the claim with the standard argument. Therefore,
\[\phi(v_s^{\phantom{*}}v_t^{\phantom{*}}v_r^*) = N_{s}^{-\beta} \ \delta_{i(st), i(r)} \phi(v_t^{\phantom{*}}v_{c(r)}^*v_{i(r)}^*v_s^{\phantom{*}}) = N_{s}^{-\beta} \ \delta_{i(st), i(r)} \phi (v_t^{\phantom{*}}v_r^*v_s^{\phantom{*}}),\]
and to prove \eqref{eq:alg char:intermediate step} it remains to show that $\phi (v_t^{\phantom{*}}v_r^*v_s^{\phantom{*}}) = 0$ if $i(st) \neq i(r)$.
Suppose $\phi (v_t^{\phantom{*}}v_r^*v_s^{\phantom{*}})\neq 0$. Then $sS \cap rS \neq\emptyset$, so we can assume $v_r^*v_s^{\phantom{*}} = v_{s'}^{\phantom{*}}v_{r'}^*$ for some $s',r' \in S$ with $rs'=sr'$. According to \eqref{eq:alg char of KMS states-gen with i and c}, $\phi (v_t^{\phantom{*}}v_r^*v_s^{\phantom{*}}) = \phi (v_{ts'}^{\phantom{*}}v_{r'}^*)$, so our assumption that this is nonzero implies $i(ts')=i(r')$. But then $i(sts') = i(sr') = i(rs')$, and in particular $N_{st}=N_r$ and $st \not\perp r$. Thus, by the definition of the map $i$, we have $i(st)=i(r)$, showing the desired converse that establishes \eqref{eq:alg char:intermediate step}. To get $\phi(v_s^*v_t^{\phantom{*}}v_r^*) = N_s^{\beta} \phi(v_t^{\phantom{*}}v_r^*v_s^*)$, we use the case just established to get
\[\phi(v_t^{\phantom{*}}v_r^*v_s^*)=\overline{\phi(v_s^{\phantom{*}}v_r^{\phantom{*}}v_t^*)}=N_s^{-\beta} \overline{\phi(v_r^{\phantom{*}}v_t^*v_s^{\phantom{*}})}=N_s^{-\beta} \phi(v_s^*v_t^{\phantom{*}}v_r^*).\]
Now suppose we have $s_1,s_2,t_1,t_2 \in S$. Then $\phi(v_{s_1}^{\phantom{*}}v_{t_1}^*v_{s_2}^{\phantom{*}}v_{t_2}^*)$ vanishes unless $t_1S \cap s_2S = t_1t_3S, t_1t_3=s_2s_3$ for some $s_3,t_3 \in S$. In this case, we get
\[\phi(v_{s_1}^{\phantom{*}}v_{t_1}^*v_{s_2}^{\phantom{*}}v_{t_2}^*) = \phi(v_{s_1}^{\phantom{*}}v_{t_3}^{\phantom{*}}v_{t_2s_3}^*) = N_{s_1}^{-\beta}\phi(v_{t_3}^{\phantom{*}}v_{t_2s_3}^*v_{s_1}^{\phantom{*}}) = N_{s_1}^{-\beta} \phi(v_{t_1}^*v_{s_2}^{\phantom{*}}v_{t_2}^*v_{s_1}^{\phantom{*}})\]
by \eqref{eq:alg char:intermediate step}. Proceeding in the same manner with $v_{t_1}^*$ against $v_{s_2}^{\phantom{*}}v_{t_2}^*v_{s_1}^{\phantom{*}}$, the term vanishes unless $t_2S \cap s_1S \neq \emptyset$, in which case the adjoint version of \eqref{eq:alg char:intermediate step} gives
\[\phi(v_{s_1}^{\phantom{*}}v_{t_1}^*v_{s_2}^{\phantom{*}}v_{t_2}^*) = N_{s_1}^{-\beta} \phi(v_{t_1}^*v_{s_2}^{\phantom{*}}v_{t_2}^*v_{s_1}^{\phantom{*}}) = N_{s_1}^{-\beta} N_{t_1}^\beta \phi(v_{s_2}^{\phantom{*}}v_{t_2}^*v_{s_1}^{\phantom{*}}v_{t_1}^*).\]
As a final step we remark that an analogous argument with $s_2$ and $t_2$ in place of $s_1$ and $t_1$ shows that $\phi(v_{s_2}^{\phantom{*}}v_{t_2}^*v_{s_1}^{\phantom{*}}v_{t_1}^*)$ also vanishes unless $t_1S \cap s_2S \neq \emptyset$ and $t_2S \cap s_1S \neq \emptyset$. This concludes the proof that every state $\phi$ on $C^*(S)$ with $\phi \circ \varphi$ tracial and such that \eqref{eq:alg char of KMS states-gen with i and c} holds is a $\kms_\beta$-state.
\end{proof}

We continue this section with some results describing which $\kms$-states factor through the maps appearing in the \emph{boundary quotient diagram} \eqref{eq:BQD for thm}. Recall that the kernel of the homomorphism $\pi_c$ is generated by $1-v_a^{\phantom{*}}v_a^*$ with $a\in S_c$.


\begin{proposition}\label{prop:ground states that factor through Q_c vs. KMS infty}
Suppose that for each normalised trace $\tau$ on $C^*(S_c)$ there exists a $\kms_\beta$-state $\phi$ on $C^*(S)$ with $\phi \circ \varphi = \tau$ for arbitrarily large $\beta \in \R$. Then a ground state $\psi$ on $C^*(S)$ is a $\kms_\infty$-state if and only if $\psi\circ \varphi$ is tracial.
\end{proposition}
\begin{proof}
If $\psi$ is a $\kms_\infty$-state, then $\psi\circ\varphi$ is the weak* limit of normalised traces $\phi_n \circ \varphi$ for some sequence of $\kms_{\beta_n}$-states $\phi_n$ with $\beta_n \to \infty$, see Proposition~\ref{prop:algebraic characterization of KMS states-gen}. Hence $\psi\circ\varphi$ is also a normalised trace.

Now let $\psi$ be a ground state on $C^*(S)$ such that $\psi\circ \varphi$ is tracial. By assumption, there is a sequence of $\kms_{\beta_n}$-states $\phi_n$ on $C^*(S)$ with $\beta_n \to \infty$ and $\phi_n \circ \varphi = \psi\circ \varphi$. Weak* compactness of the state space of $C^*(S)$ yields a subsequence of $(\phi_n)_n$ converging to a $\kms_\infty$-state $\phi'$ which necessarily satisfies $\phi' \circ \varphi =\psi\circ \varphi$. Therefore $\psi$ and $\phi'$ agree on $\varphi(C^*(S_c))$, hence $\psi=\phi'$ according to Remark~\ref{rem:C*(S_c) in C*(S)}, and $\psi$ is thus a $\kms_\infty$-state.


\end{proof}


\begin{proposition}\label{prop:KMS-state factoring through Q_p must have beta=1}
A $\kms_\beta$-state $\phi$ on $C^*(S)$ factors through $\pi\colon C^*(S) \to \CQ(S)$ if and only if $\beta=1$. If $S$ is also core factorable, then $\pi$ can be replaced by $\pi_p\colon C^*(S) \to \CQ_p(S)$.
\end{proposition}
\begin{proof}
Assume that $\phi= \phi' \circ \pi$ for a $\kms_\beta$-state $\phi'$ on $\CQ(S)$. As in the proof of Proposition~\ref{prop:no KMS below 1-gen}, there is an accurate foundation set $F$ with $\lvert F \rvert = N_f = n$ for all $f \in F$ and some $n > 1$. Thus
\begin{equation}\label{eq:computations at beta 1} 1 = \phi'(1) = \phi'(\sum\limits_{f \in F}\pi(e_{fS})) = \sum\limits_{f \in F}\phi(e_{fS}) = n^{1-\beta},
\end{equation}
 so necessarily $\beta = 1$.

Conversely, suppose that $\phi$ is a $\kms_1$-state. We must show that $\ker \pi \subset \ker\phi$. Recall from Proposition~\ref{prop:NisaHMofRLCMs}~(v) that $S$ has the accurate refinement property from \cite{BS1}, thus the kernel of $\pi$ is generated by differences $1-\sum_{f \in F} e_{fS}$ for $F$ running over accurate foundation sets. By Proposition~\ref{prop:NisaHMofRLCMs}, it suffices to show that $1-\sum_{f \in F} e_{fS}$ is in $\ker\phi$ for $F$ running over accurate foundation sets that are transversals for $N^{-1}(n)/_\sim$ as $n\in N(S)$. However, for such $F$, the computations in \eqref{eq:computations at beta 1} done backwards prove that $\phi(1-\sum_{f \in F} e_{fS})=0$, as needed.

If $S$ is core factorable, then $F$ can be chosen as a subset of $S_{ci}^{1}$ so that we need to have $\beta=1$. As $\ker \pi_p \subset \ker\pi$, the second part is clear.
\end{proof}

For the proof of the next result we refer to \cite{Sta3}*{Theorem~4.1}.

\begin{proposition}\label{prop:no GS on Q_p}
There are no ground states on $\CQ(S)$. If $S$ is core factorable, then $\CQ_p(S)$ does not admit any ground states.
\end{proposition}

\section{The reconstruction formula}\label{sec:theReconstruction}
As in the previous section, we shall assume that $S$ is a right LCM semigroup that admits a generalised scale $N\colon S \to \N^\times$, and denote by $\sigma$ the time evolution on $C^*(S)$ given by $\sigma_x(v_s)=N_s^{ix}v_s$ for all $x\in \R$ and $s\in S$. Under the additional assumptions that make $S$ an admissible semigroup, we obtain a formula to reconstruct $\kms_\beta$-states on $C^*(S)$ from conditioning on subsets $I$ of $\text{Irr}(N(S))$ for which $\zeta_I(\beta)$ is finite, see Lemma~\ref{lem:rec formula-gen}~(iv).

\begin{lemma}\label{lem:basic defect projection-gen}
Suppose that $S$ is core factorable and admits a generalised scale $N$. For $n \in N(S)$ and a transversal $\CT_n$ for $N^{-1}(n)/_\sim$ with $\CT_n\subset S_{ci}^1$, let $d_n := 1 - \sum_{f \in \CT_n} e_{fS}$. Then $d_n$ is a projection in $C^*(S)$, and if $S_{ci}\subset S$ is $\cap$-closed, then $d_n$ is independent of the choice of $\CT_n$.
\end{lemma}
\begin{proof}
For $n=1$, we have $\CT_1 = \{1\}$ and hence $d_1 = 0$. So let $n \in N(S)\setminus\{1\}$. As $S$ is core factorable, there exists a transversal $\CT_n$ for $N^{-1}(n)/_\sim$ with $\CT_n \subset S_{ci}$. By \ref{cond:A3b}, $\CT_n$ is an accurate foundation set. Therefore $\{e_{fS}\mid f\in \CT_n\}$ are mutually orthogonal projections in $C^*(S)$, so $d_n$ is a projection. If $S_{ci}\subset S$ is $\cap$-closed and $\CT'_n$ is another transversal for $N^{-1}(n)/_\sim$ contained in $S_{ci}^1$, then Corollary~\ref{cor:dichotomy for S_ci under gen scale} implies that $\sum_{f\in \CT_n}e_{fS}=\sum_{f'\in \CT'_n}e_{f'S}$.
\end{proof}

\begin{lemma}\label{lem:d_p commutes with v_s-gen}
Suppose that $S$ is core factorable, $S_{ci}\subset S$ is $\cap$-closed, and $S$ admits a generalised scale $N$. Then for each $n \in N(S)$ and $a \in S_c$, the isometry $v_a$ commutes with the projection $d_n$.
\end{lemma}
\begin{proof}
Fix $n\in N(S)$ and $t\in S_c$. We can assume $n >1$. Since $v_a$ is an isometry, the claim will follow once we establish that
\[\begin{array}{c}
v_a^{\phantom{*}}d_n^{\phantom{*}}v_a^* = e_{aS}^{\phantom{*}} - \sum\limits_{f \in \CT_n} e_{afS}^{\phantom{*}} \stackrel{!}{=} e_{aS}^{\phantom{*}} - \sum\limits_{f \in \CT_n} e_{aS \cap fS}^{\phantom{*}} = d_n^{\phantom{*}}v_a^{\phantom{*}}v_a^*,
\end{array}\]
where the second equality is the one requiring a proof. To do so, observe that since $a\in S_c$ and $\CT_n = \{f_1,\ldots,f_n\} \subset S_{ci}$, we have $aS \cap f_iS = af'_iS$ for some $f'_i \in S$ with $N_{f'_i} = N_{f_i} = n$. By Lemma~\ref{lem:OldA2}, we have $\CT'_n := \{f'_1,\ldots,f'_n\} \subset S_{ci}$. Moreover, $f'_i \perp f'_j$ for all $i \neq j$ because the same is true of $f_i$ and $f_j$. Thus $\CT_n$ and $\CT_n'$ are two transversals for $N^{-1}(n)/_\sim$ that are contained in $S_{ci}^1$, and so Corollary~\ref{cor:dichotomy for S_ci under gen scale} implies that $\sum_{f' \in F'_n} e_{f'S} = \sum_{f \in F_n} e_{fS}$, showing the desired equality.
\end{proof}

The next lemma is crucial to what follows. Roughly, it says the following: for an admissible semigroup, condition \ref{cond:A4} allows that a transversal corresponding to a finite intersection of principal right ideals $f_nS$ with $N_{f_n}=n$ and $n\in \text{Irr}(N(S))$ can be taken in the form of a transversal for the inverse image, under the scale $N$, of the product of all $n\in N(S)$ that appear in the intersection.

\begin{lemma}\label{lem:prod of d_n for irreducibles}
 Let $S$ be an admissible right LCM semigroup and $m,n \in \text{Irr}(N(S)), m \neq n$. If $\CT_m=\{f_1,\ldots,f_m\}$ and $\CT_n=\{g_1,\ldots,g_n\}$ are transversals for $N^{-1}(m)/_\sim$ and $N^{-1}(n)/_\sim$, respectively, then $f_iS \cap g_jS=h_{i+m(j-1)}S$ for some $h_{i+m(j-1)} \in S$ for all $i,j$, and $\CT_{mn} := \{h_1,\ldots,h_{mn}\}$ is a transversal for $N^{-1}(mn)/_\sim$.
\end{lemma}
\begin{proof}
By \ref{cond:A3b}, $\CT_m$ and $\CT_n$ are accurate foundation sets. Let $\CT_{mn}$ consist of the $h_{i+m(j-1)}$ for which $f_i \not\perp g_j$, where $i=1,\dots, m$ and $j=1,\dots, n$. We will show that $f_i \not\perp g_j$ for all $i$ and $j$. Clearly $\CT_{mn}$ is an accurate foundation set.
Since the generalised scale $N$ is a homomorphism of right LCM semigroups by Proposition~\ref{prop:NisaHMofRLCMs}~(iv) and $m$ and $n$ are distinct irreducibles in $N(S)$, we have $\CT_{mn} \subset N^{-1}(mn)$. Then Proposition~\ref{prop:NisaHMofRLCMs}~(iii) implies that $\CT_{mn}$ must be a transversal for $N^{-1}(mn)/_\sim$. Thus, $|\CT_{mn}| = mn$ by \ref{cond:A3a} which gives the claim.
\end{proof}

\begin{remark}\label{rem:product formula for zeta function-gen}
We record another consequence of condition \ref{cond:A4}. Suppose $I\subset\Irr(N(S))$. We claim that the $I$-restricted $\zeta$-function $\zeta_I$ from Definition~\ref{def:partial zeta function} admits a product description $\zeta_I(\beta) = \prod_{n \in I} (1-n^{-(\beta-1)})^{-1}$. Indeed, all summands in the series defining $\zeta_I(\beta)$ are positive, so we are free to rearrange their order. Condition \ref{cond:A4} implies that the submonoid $\langle I \rangle$ of $\N^\times$ is the free abelian monoid in $I$, and therefore
\[\begin{array}{l}
\zeta_I(\beta) = \sum\limits_{n \in \langle I \rangle} n^{-(\beta-1)}
= \prod\limits_{n \in I} (1-n^{-(\beta-1)})^{-1}.
\end{array}\]
\end{remark}

The next result is an abstract version of \cite{LR2}*{Lemma 10.1}, proved for the semigroup $C^*$-algebra of $\N\rtimes \N^\times$. Laca and Raeburn had anticipated back then that the result is ``likely to be useful elsewhere". Here we confirm this point by showing that the techniques of proof, when properly updated, are valid in the abstract setting of our admissible semigroups.

We recall that the $*$-homomorphism $\varphi:C^*(S_c) \to C^*(S)$ was defined in Remark~\ref{rem:C*(S_c) in C*(S)}. In the following, we let $\CT$ be a transversal for $S/_\sim$ with $\CT \subset S_{ci}^1$ and $\CT_I := \bigcup_{n \in \langle I \rangle} \CT_n = N^{-1}(\langle I \rangle) \cap \CT$ for $I \subset \text{Irr}(N(S))$.

\begin{lemma}\label{lem:rec formula-gen}
 Let $S$ be an admissible right LCM semigroup. If $\phi$ is a $\kms_\beta$-state on $C^*(S)$ for some $\beta\in \R$, denote by $(\pi_\phi,\CH_\phi,\xi_\phi)$ its GNS-representation and by $\tilde{\phi} = \langle \cdot \ \xi_\phi, \xi_\phi \rangle$ the vector state extension of $\phi$ to $\CL(\CH_\phi)$. For every subset $I$ of $\Irr(N(S))$, $Q_I := \prod\limits_{n \in I} \pi_\phi(d_n)$ defines a projection in $\pi_\phi(C^*(S))''$. If $\zeta_I(\beta) < \infty$, then $Q_I$ has the following properties:
\begin{enumerate}[(i)]
\item $\tilde{\phi}(Q_I) = \zeta_I(\beta)^{-1}$.
\item The map $y \mapsto \zeta_I(\beta) \ \tilde{\phi}(Q_I\pi_\phi(y)Q_I)$ defines a state $\phi_I$ on $C^*(S)$ for which $\phi_I \circ \varphi$ is a trace on $C^*(S_c)$.
\item The family $(\pi_\phi(v_s^{\phantom{*}})Q_I^{\phantom{*}}\pi_\phi(v_s^*))_{s \in \CT_I}$ consists of mutually orthogonal projections, and
$Q^I := \sum_{s \in \CT_I} \pi_\phi(v_s^{\phantom{*}})Q_I^{\phantom{*}}\pi_\phi(v_s^*)$
defines a projection such that $\tilde{\phi}(Q^I) = 1$.
\item There is a reconstruction formula for $\phi$ given by
\[\begin{array}{c}
\phi(y) = \zeta_I(\beta)^{-1} \sum\limits_{s \in \CT_I} N_s^{-\beta} \ \phi_I(v_s^*yv_s^{\phantom{*}}) \quad \text{ for all } y \in C^*(S).
\end{array}\]
\end{enumerate}
\end{lemma}
\begin{proof}
If $I$ is finite, then $Q_I$ belongs to $\pi_\phi(C^*(S))$. If $I$ is infinite, then we obtain $Q_I$ as a weak limit of projections $Q_{I_n}$ with $|I_n| < \infty$ and $I_n \nearrow I$ because the $d_n$ commute. Hence $Q_I$ belongs to the double commutant of $\pi_\phi(C^*(S))$ inside $\CL(\CH_\phi)$.

Let us first assume that $I$ is finite. An iterative use of Lemma~\ref{lem:prod of d_n for irreducibles} then implies
\begin{equation}\label{eq:finite prod of defect projections-gen}
\begin{array}{c}
\prod\limits_{n \in I} d_n = \sum\limits_{A \subset I} (-1)^{|A|} \sum\limits_{(f_1,\ldots,f_{\lvert A \rvert }) \in \prod\limits_{n \in A} \CT_n } e_{\cap_{k=1}^{\lvert A \rvert} f_kS}
= \sum\limits_{A \subset I} (-1)^{|A|} \sum\limits_{f \in \CT_{n(A)}} e_{fS}.
\end{array}
\end{equation}
 Applying $\tilde{\phi}\circ \pi_\phi = \phi$ to \eqref{eq:finite prod of defect projections-gen}, using that $\phi$ is a $\kms_\beta$-state on $C^*(S)$ as in the last equality of \eqref{eq:computations at beta 1} and invoking the product formula from Remark~\ref{rem:product formula for zeta function-gen} yields
\[\begin{array}{l}
\tilde{\phi}(Q_I) = \sum\limits_{A \subset I} (-1)^{|A|} {n_A}^{-(\beta-1)} = \prod\limits_{n \in I} (1-n^{-(\beta-1)}) = \zeta_I(\beta)^{-1}.
\end{array}\]
 This proves part (i) in the case that $I$ is finite. The general case follows from the finite case by taking limits $I_n \nearrow I$ with $I_n$ finite, using continuity as $\tilde{\phi}$ is a vector state and the fact that $Q_I$ is the weak limit of $(Q_{I_n})_{n \geq 1}$.

For (ii), we note that $\phi_I$ is a state on $C^*(S)$ by i), and it remains to prove that it is tracial on $\varphi(C^*(S_c))$. Let $I$ be finite and $r_i \in S_c, i = 1,\ldots,4$. According to Lemma~\ref{lem:d_p commutes with v_s-gen} and the fact that $\phi$ is a $\kms_\beta$-state (keeping in mind that $N_{r_1}=N_{r_2}=1$), we have
\[\begin{array}{lcl}
\phi_I(v_{r_1}^{\phantom{*}}v_{r_2}^*v_{r_3}^{\phantom{*}}v_{r_4}^*) &=& \zeta_I(\beta) \ \phi(\prod\limits_{m \in I} d_m \ v_{r_1}^{\phantom{*}}v_{r_2}^*v_{r_3}^{\phantom{*}}v_{r_4}^* \prod\limits_{n \in I} d_n) \vspace*{2mm}\\
&=& \zeta_I(\beta) \ \phi(v_{r_1}^{\phantom{*}}v_{r_2}^* \prod\limits_{m \in I} d_m \ v_{r_3}^{\phantom{*}}v_{r_4}^* \prod\limits_{n \in I} d_n) \vspace*{2mm}\\
&=& \zeta_I(\beta) \ \phi(\prod\limits_{m \in I} d_m \ v_{r_3}^{\phantom{*}}v_{r_4}^* \prod\limits_{n \in I} d_n \ v_{r_1}^{\phantom{*}}v_{r_2}^*) \vspace*{2mm}\\
&=& \zeta_I(\beta) \ \phi(\prod\limits_{m \in I} d_m \ v_{r_3}^{\phantom{*}}v_{r_4}^*\ v_{r_1}^{\phantom{*}}v_{r_2}^* \prod\limits_{n \in I} d_n) = \phi_I(v_{r_3}^{\phantom{*}}v_{r_4}^*v_{r_1}^{\phantom{*}}v_{r_2}^*).
\end{array}\]
For general $I$, the state $\phi_I$ is the weak*-limit of $\phi_{I_n}$ where $I_n$ is finite and $I_n \nearrow I$, so that the trace property is preserved. This completes (ii).

Towards proving (iii), note first that by Lemma~\ref{lem:basic defect projection-gen} the family of projections appearing in the assertion is independent of the choice of the transversals as $n$ varies in $\langle I \rangle$. Let $s \neq t$ with $s$ in a transversal $\CT_m$ for $N^{-1}(m)/_\sim$ and $t$ in a transversal $\CT_n$ for $N^{-1}(n)/_\sim$ for some $m,n \in \langle I \rangle$. If $m=n$, then $s$ and $t$ are distinct elements belonging to a transversal that is also a (proper) accurate foundation set, so $s\perp t$, which implies $v_s^*v_t^{\phantom{*}} = 0$ and therefore shows that the projections $\pi_\phi(v_s^{\phantom{*}})Q_I\pi_\phi(v_s^*)$ and $\pi_\phi(v_t^{\phantom{*}})Q_I\pi_\phi(v_t^*)$ are orthogonal. Suppose $m \neq n$. Now either $s\perp t$, which implies the claimed orthogonality as above, or $s$ and $t$ have a right LCM in $S$. In the latter case, using that $N$ is a homomorphism of right LCM semigroups, it follows that there are $s',t'\in S$ with $ss'=tt'$ and $N_{s'} = m^{-1}\text{lcm}(m,n), N_{t'} = n^{-1}\text{lcm}(m,n)$. Since $s,t\in S_{ci}$, $sS\cap tS\neq\emptyset$ and $S_{ci}\subset S$ is $\cap$-closed, also $ss', tt'\in S_{ci}$. This forces $s', t'\in S_{ci}$. The assumption $m \neq n$ implies that $N_{s'} > 1$ or $N_{t'} > 1$. Suppose $N_{s'} > 1$. Note $N_{s'} \in \langle I \rangle$ because this monoid is free abelian by \ref{cond:A4}. Take $n'\in I$ such that $N_{s'}$ is a multiple of $n'$, and pick a corresponding proper accurate foundation set $F_{n'}$. Then $fS \cap s'S \neq \emptyset$ for some $f \in F_{n'}$. Using that $S_{ci}\subset S$ is $\cap$-closed, we can write $fS \cap s'S = s''S$ for some $s'' \in S_{ci}$ for which $N_{s''}=\text{lcm}(N_{s'},n')=N_{s'}$. Then $s'$ and $s''$ differ by multiplication on the right with an element of $S^*$ according to Corollary~\ref{cor:dichotomy for S_ci under gen scale}, so $fS \cap s'S = s'S$. Moreover, there is a unique $f \in F_{n'}$ with this property because $F_{n'}$ is accurate. Now $s'S\subset fS$ implies that $d_{n'}e_{s'S} = 0$, and this in turn yields
\[\begin{array}{c}
Q_I\pi_\phi(v_s^*v_t^{\phantom{*}})Q_I = Q_{I\setminus\{n'\}}\pi_\phi(d_{n'}e_{s'S}v_{s'}^{\phantom{*}}v_{t'}^*)Q_I = 0,
\end{array}\]
from which the desired orthogonality follows. The case $N_{t'}>1$ is analogous. Therefore, we get mutually orthogonal projections whenever $s \in \CT_m$ and $t \in \CT_n$ are distinct. In particular, $Q^I$ is a projection in $\pi_\phi(C^*(S))''$. For a finite set $I$, an application of the $\kms_\beta$-condition at the second equality shows that
\[\begin{array}{l}
\tilde{\phi}(Q^I) = \sum\limits_{s \in \CT_I} \phi(v_s^{\phantom{*}}\prod\limits_{m \in I} d_m^{\phantom{*}} \ v_s^*) 
= \sum\limits_{s \in \CT_I} N_s^{-\beta} \tilde{\phi}(Q_I) 
= \tilde{\phi}(Q_I) \zeta_I(\beta).
\end{array}\]
The last term is $1$ by part (i), and (iii) follows. For arbitrary $I$, we invoke continuity as was done in the proof of part i).

Finally, to prove (iv), note that $\tilde{\phi}(Q^I) = 1$ implies $\tilde{\phi}(\pi_\phi(y)) = \tilde{\phi}(Q^I\pi_\phi(y)Q^I)$ for all $y \in C^*(S)$. For finite $I$, appealing to (iii) and using the $\kms_\beta$-condition three times, we obtain
\[\begin{array}{lclcl}
\phi(y) &=& \tilde{\phi}(\pi_\phi(y)) = \tilde{\phi}(Q^I\pi_\phi(y)Q^I) 
&=& \sum\limits_{s,t \in \CT_I} \phi(v_s^{\phantom{*}}\prod\limits_{k \in I}d_k^{\phantom{*}} \ v_s^*yv_t^{\phantom{*}}\prod\limits_{\ell \in I}d_\ell^{\phantom{*}} \ v_t^*) \vspace*{2mm}\\
&&&=& \sum\limits_{s,t \in \CT_I} N_s^{-\beta} \ \phi(\prod\limits_{k \in I}d_k^{\phantom{*}} \ v_s^*yv_t^{\phantom{*}}\prod\limits_{\ell \in I}d_\ell^{\phantom{*}} \ v_t^*v_s^{\phantom{*}}\prod\limits_{k' \in I}d_{k'}^{\phantom{*}}) \vspace*{2mm}\\
&&&=& \sum\limits_{s,t \in \CT_I} \tilde{\phi}(Q_I \pi_\phi(v_s^*yv_t^{\phantom{*}})Q_I\pi_\phi(v_t^*v_s^{\phantom{*}})Q_I) \vspace*{2mm}\\
&&&=& \zeta_I(\beta)^{-1}\sum\limits_{s \in \CT_I} N_s^{-\beta}\phi_I(v_s^*yv_s^{\phantom{*}}).
\end{array}\]
This proves the reconstruction formula in the case of finite $I$. For general $I \subset \Irr(N(S))$ the claim follows by continuity arguments.
\end{proof}

\section{Construction of KMS-states}\label{sec:construction}
In this section we show that the constraints obtained in Section~\ref{sec:constraints} are optimal. To begin with, we assume that $S$ is a core factorable right LCM semigroup for which $S_{ci} \subset S$ is $\cap$-closed. We then fix a transversal $\CT$ for $S/_\sim$ with $\CT \subset S_{ci}^1$ according to Lemma~\ref{lem:maps i and c}; thus $\CT$ gives rise to maps $i\colon S \to \CT$ and $c\colon S \to S_c$ determined by $s=i(s)c(s)$ for all $s \in S$. Recall from Lemma~\ref{lem:thebijectionsg} that we denote by $\alpha$ both the action $S_c \curvearrowright S/_\sim$ and the corresponding action $S_c \curvearrowright \CT$. Recall further from Remark~\ref{rem:C*(S_c) in C*(S)} that the standard generating isometries for $C^*(S_c)$ are $w_a$ for $a \in S_c$.

\begin{remark}\label{rem:independence of T for inner product value}
Consider $C^*(S_c)$ as a right Hilbert module over itself with inner product $\langle y,z\rangle = y^*z$ for $y,z\in C^*(S_c)$. For every $a \in S_c$, the map $y \mapsto w_ay$ defines an isometric endomorphism of $C^*(S_c)$ as a right Hilbert $C^*(S_c)$ module. If $a \in S^*$, then the map is an isometric isomorphism. Indeed, the map is linear and compatible with the right module structures. It also preserves the inner product as $w_a$ is an isometry, and hence $\langle w_ay,w_az \rangle = y^*w_a^*w_a^{\phantom{*}}z = y^*z$. Thus, the map is isometric. If $a\in S^*$, then the map is surjective because $w_a^{\phantom{*}}(w_a^*y) = y$ for all $y \in C^*(S_c)$.
\end{remark}

We start with a brief lemma that allows us to recover $t$ and $c(st)$ from $s$ and $i(st)$ for $s \in S$ and $t \in \CT$, followed by a technical lemma that makes the proof of Theorem~\ref{thm:rep of C*(S)} more accessible. Both results are further elaborations on the factorisation arising from Lemma~\ref{lem:maps i and c}.

\begin{lemma}\label{lem:recovering t and c(st) from s and i(st)}
For $s \in S$, the map $\CT \to \CT, t \mapsto i(st)$ is injective on $\CT$, and there is a bijection $i(s\CT) \to \{ (t,c(st)) \mid t \in \CT\}$.
\end{lemma}
\begin{proof}
Suppose $s \in S, t_1,t_2 \in \CT$ satisfy $i(st_1)=i(st_2)$. Then $st_1 \sim st_2$, and left cancellation implies $t_1 \sim t_2$, so that $t_1=t_2$ as $t_1,t_2 \in \CT$. Thus, we can recover $t \in \CT$ from $i(st)$, and as $st=i(st)c(st)$, we then recover $c(st)$.
\end{proof}

\begin{lemma}\label{lem:technicality for representation}
Suppose there are $s_1,\ldots,s_4 \in S, t_1,t_2,r \in \CT$ such that $s_1S \cap s_2S = s_1s_3S, \ s_1s_3=s_2s_4, \ t_1 = i(s_4r), \ \text{and} \ t_2 = i(s_3r)$. Then $c(s_1t_2)c(s_3r) = c(s_2t_1)c(s_4r)$ and $c(s_1t_2)S_c \cap c(s_2t_1)S_c = c(s_1t_2)c(s_3r)S_c$.
\end{lemma}
\begin{proof}
The equation
\[s_1s_3r = s_1t_2c(s_3r) = i(s_1t_2)c(s_1t_2)c(s_3r) = i(s_2t_1)c(s_2t_1)c(s_4r) = s_2t_1c(s_4r) = s_2s_4r\]
yields $i(s_1t_2)\sim i(s_2t_1)$, and hence $i(s_1t_2)=i(s_2t_1)$. Thus left cancellation in $S$ gives $c(s_1t_2)c(s_3r) = c(s_2t_1)c(s_4r)$. For the second part of the claim, the first part implies $c(s_1t_2)S_c \cap c(s_2t_1)S_c \supset c(s_1t_2)c(s_3r)S_c$. Moreover, it suffices to show $c(s_1t_2)S \cap c(s_2t_1)S \subset c(s_1t_2)c(s_3r)S$ since the embedding of $S_c$ into $S$ is a homomorphism of right LCM semigroups, see Proposition~\ref{prop:hom of right LCM for S_c and S_ci}. So suppose $c(s_1t_2)S \cap c(s_2t_1)S = c(s_1t_2)aS, c(s_1t_2)a=c(s_2t_1)b$ for some $a,b \in S_c$. Then $c(s_3r) = ad$ for some $d \in S_c$, and we need to show that $d \in S^*$. As
\[s_1t_2a = i(s_1t_2)c(s_1t_2)a = i(s_2t_1)c(s_2t_1)b = s_2t_1b \in s_1S \cap s_2S = s_1s_3S,\]
left cancellation forces $t_2a \in s_3S$, i.e.~$t_2a = s_3r'$ for some $r' \in S$. Therefore, we get
\[s_3r = t_2c(s_3r) = t_2ad = s_3r'd.\]
By left cancellation, this yields $r = i(r')c(r')d$. But now $r \in \CT$, so $i(r')=r$, and hence $1 = i(r')d$, i.e.~$d \in S^*$.
\end{proof}

The next result is inspired by \cite{LRRW}*{Propositions 3.1 and 3.2}, where the authors construct a representation of the Toeplitz algebra for a Hilbert bimodule associated to a self-similar action.  We recall that when $X$ is a right-Hilbert module over a $C^*$-algebra $A$ and $\pi$ is a representation of $A$ on a Hilbert space $H$, then $X\otimes_A H$ is the Hilbert space obtained as the completion of the algebraic tensor product with respect to the inner-product characterised by $\langle x\otimes h, y\otimes k\rangle=\langle \pi(\langle y,x\rangle)h, k\rangle$, and where the tensor product is balanced over the left action on $H$ given by $\pi$, see for instance \cite{RW}. If $X$ is a direct sum Hilbert module $\oplus_{j\in J}A$ of the standard Hilbert module $A$, then $X\otimes_A H$ can be identified with the infinite direct sum of Hilbert spaces $A\otimes_A H$. In case $A$ has an identity $1$, then with $\xi_j$ denoting the image of $1\in A$ in the $j$-th copy of $A$, we may identify $\xi_ja\otimes h$ for $a\in A$ as the element in $\oplus_{j\in J}(A\otimes_A H)$ that has $a\otimes h$ in the $j$-th place and zero elsewhere. From now on, we will freely use this identification.

\begin{thm}\label{thm:rep of C*(S)}
Let $S$ be a core factorable right LCM semigroup such that $S_{ci}\subset S$ is $\cap$-closed. Let $M := \bigoplus_{t \in \CT} C^*(S_c)$ be the direct sum of the standard right Hilbert $C^*(S_c)$-module.
Suppose $\rho$ is a state on $C^*(S_c)$ with GNS-representation $(\pi'_\rho,\CH'_\rho,\xi_\rho)$. Form the Hilbert space $\CH_\rho:= M \otimes_{C^*(S_c)} \CH'_\rho$.
 The linear operator $V_s\colon \CH_\rho\to \CH_\rho$ given by
\[
V_s(\xi_t y \otimes \xi_\rho) := \xi_{i(st)}w_{c(st)}y \otimes \xi_\rho
 \]
 for $t \in \CT, y \in C^*(S_c)$ is adjointable for all $s \in S$. The family $(V_s)_{s \in S}$ gives rise to a representation $\pi_\rho\colon C^*(S) \to \CL(\CH_\rho)$.
\end{thm}
\begin{proof}
It is clear that $\CH_\rho$ is a Hilbert space, so we start by showing that the operator $V_s$ is adjointable. For $s \in S$ and $t_1,t_2 \in \CT, y,z \in C^*(S_c)$, we have
\[\begin{array}{lcl}
\langle V_s(\xi_{t_1}y \otimes \xi_\rho), \xi_{t_2}z \otimes \xi_\rho \rangle
&=& \langle \xi_\rho,\pi'_\rho(\langle \xi_{i(st_1)}w_{c(st_1)}y, \xi_{t_2}z \rangle) \xi_\rho \rangle \\
&=& \langle \xi_\rho,\pi'_\rho(\delta_{i(st_1), t_2}y^*w_{c(st_1)}^*z) \xi_\rho \rangle \\
&=& \delta_{i(st_1), t_2} \ \rho(z^*w_{c(st_1)}y).
\end{array}\]
Define the linear operator $V'_s$ on $\CH_\rho$ by
\[V'_s(\xi_t y \otimes \xi_\rho) := \chi_{i(s\CT)}(t) \xi_{t'}w_{c(st')}^*y \otimes \xi_\rho,\]
where $t' \in \CT$ is uniquely determined by $i(st')=t$ in the case where $t \in i(s\CT)$, see Lemma~\ref{lem:recovering t and c(st) from s and i(st)}. It is then straightforward to show that
\[\begin{array}{lcl}
\langle \xi_{t_1}y \otimes \xi_\rho, V'_s(\xi_{t_2}z \otimes \xi_\rho) \rangle
&=& \delta_{i(st_1), t_2} \ \rho(z^*w_{c(st_1)}y),
\end{array}\]
so that $V_s$ is adjointable with $V_s^* = V'_s$. It is also clear that each $V_s$ is an isometry. Now let $s_1,s_2 \in S$. To show that $V_{s_1}V_{s_2} = V_{s_1s_2}$, note that when computing $V_{s_1}V_{s_2}(\xi_{t}y \otimes \xi_\rho)$ for some $t\in \CT$ we encounter the term $
\xi_{i(s_1i(s_2t))}w_{c(s_1i(s_2t))}w_{c(s_2t)}y$, which is the same as $\xi_{i(s_1s_2t)}w_{c(s_1s_2t)}y$ entering $V_{s_1s_2}(\xi_{t}y \otimes \xi_\rho)$ because
\[\begin{array}{lcl}
i(s_1i(s_2t))c(s_1i(s_2t))c(s_2t) = s_1i(s_2t)c(s_2t) = s_1s_2t = i(s_1s_2t)c(s_1s_2t)
\end{array}\]
 by Lemma~\ref{lem:recovering t and c(st) from s and i(st)}. So it remains to prove that we have
\[V_{s_1}^*V_{s_2}^{\phantom{*}} =
\begin{cases}
V_{s_3}^{\phantom{*}}V_{s_4}^* &\text{if } s_1S \cap s_2S = s_1s_3S, s_1s_3=s_2s_4, \\
0 &\text{if } s_1 \perp s_2.
\end{cases}\]
We start by observing that since
\begin{align}
\langle V_{s_1}^*V_{s_2}^{\phantom{*}}(\xi_{t_1}y \otimes \xi_\rho), \xi_{t_2}z \otimes \xi_\rho \rangle
&= \langle V_{s_2}^{\phantom{*}}(\xi_{t_1}y \otimes \xi_\rho), V_{s_1}^{\phantom{*}}\xi_{t_2}z \otimes \xi_\rho \rangle \notag\\
&= \langle \xi_\rho,\pi'_\rho(\langle \xi_{i(s_2t_1)}w_{c(s_2t_1)}y, \xi_{i(s_1t_2)}w_{c(s_1t_2)}z \rangle) \xi_\rho \rangle \notag \\
&= \delta_{i(s_2t_1), i(s_1t_2)} \ \rho(z^*w_{c(s_1t_2)}^*w_{c(s_2t_1)}^{\phantom{*}}y),\label{eq:Vs1starVs2nonzero}
\end{align}
the inner product on the left can only be nonzero if $s_1 \not \perp s_2$. Indeed, if $s_1 \perp s_2$, then $i(s_2t_1) \neq i(s_1t_2)$ for all $t_1,t_2 \in \CT$ as $c(s_2t_1),c(s_1t_2) \in S_c$. So suppose there are $s_3,s_4 \in S$ such that $s_1S \cap s_2S = s_1s_3S, s_1s_3=s_2s_4$. If $t_1,t_2 \in \CT$ and $y,z \in C^*(S_c)$ are such that
\[\begin{array}{lcl}
\langle V_{s_3}^{\phantom{*}}V_{s_4}^*(\xi_{t_1}y \otimes \xi_\rho), \xi_{t_2}z \otimes \xi_\rho \rangle \neq 0,
\end{array}\]
then there must be $t_3 \in \CT$ with the property that $t_1 = i(s_4t_3)$ and $t_2=i(s_3t_3)$. In this case, we have
\[\begin{array}{lcl}
\langle V_{s_3}^{\phantom{*}}V_{s_4}^*(\xi_{t_1}y \otimes \xi_\rho), \xi_{t_2}z \otimes \xi_\rho \rangle = \rho(z^*w_{c(s_3t_3)}^{\phantom{*}}w_{c(s_4t_3)}^*y).
\end{array}\]
In addition, we get $s_2t_1c(s_4t_3) = s_2s_4t_3 = s_1s_3t_3 = s_1t_2c(s_3t_3)$ so that $i(s_2t_1) = i(s_1t_2)$, which by \eqref{eq:Vs1starVs2nonzero} also determines when $V_{s_1}^*V_{s_2}^{\phantom{*}}$ is nonzero. Now $w_{c(s_1t_2)}^*w_{c(s_2t_1)}^{\phantom{*}} = w_{r_1}^{\phantom{*}}w_{r_2}^*$ for all $r_1,r_2 \in S_c$ satisfying $c(s_1t_2)S_c \cap c(s_2t_1)S_c = c(s_1t_2)r_1S_c$ and $c(s_1t_2)r_1 = c(s_2t_1)r_2$. Due to Lemma~\ref{lem:technicality for representation}, $c(s_3t_3)$ and $c(s_4t_3)$ have this property. Thus we conclude that $V_{s_1}^*V_{s_2}^{\phantom{*}}=V_{s_3}^{\phantom{*}}V_{s_4}^*$ for $s_1 \not\perp s_2$. This shows that $(V_s)_{s \in S}$ defines a representation $\pi_\rho$ of $C^*(S)$ by the universal property of $C^*(S)$.
\end{proof}

From now on we assume that $S$ is a core factorable right LCM semigroup, $S_{ci}\subset S$ is $\cap$-closed, and $S$ admits a generalised scale.

\begin{lemma}\label{lem:eval gen}
For every state $\rho$ on $C^*(S_c)$ the following assertions hold:
\begin{enumerate}[(i)]
\item If $s,r \in S$ are such that $\langle V_s^{\phantom{*}}V_r^*(\xi_t \otimes \xi_\rho), \xi_t \otimes \xi_\rho \rangle \neq 0$ for some $t\in \CT$, then $s \sim r$.
\item Given $a,b \in S_c$ and $t\in \CT$, then for all $y,z\in C^*(S)$ we have
\[
\langle V_a^{\phantom{*}}V_b^*(\xi_t y\otimes \xi_\rho), \xi_t z\otimes \xi_\rho \rangle
= \chi_{\text{Fix}(\alpha_{a}^{\phantom{1}}\alpha_{b}^{-1})}(t) \ \rho(z^*w_{c(a\alpha_{b}^{-1}(t))}^{\phantom{*}}w_{c(b\alpha_{b}^{-1}(t))}^*y).
\]
\end{enumerate}
\end{lemma}
\begin{proof}
For (i), suppose $r,s$ and $t$ are given as in the first assertion of the lemma. Then there is $t' \in \CT$ with $i(st')=t=i(rt')$. Thus we have $st' \sim rt'$. Hence, $N_s = N_r$ and $sS \cap rS \neq \emptyset$. By Proposition~\ref{prop:NisaHMofRLCMs} (ii), we have $s \sim r$.

To prove (ii), note that straightforward calculations using Lemma~\ref{lem:thebijectionsg} show that
\[
V_a(\xi_t y\otimes \xi_\rho)=\xi_{\alpha_a(t)}w_{c(a\alpha_a(t))}y\otimes \xi_\rho\text{ and }
V_b^*(\xi_t y\otimes \xi_\rho)=\xi_{\alpha_{b}^{-1}(t)}(w^*_{c(b\alpha_b^{-1}(t))}y\otimes \xi_\rho).
\]
This immediately gives the claim on $V_a^{\phantom{*}}V_b^*$.
\end{proof}

\begin{proposition}\label{prop:construction of ground states}
For every state $\rho$ on $C^*(S_c)$, the map
\[\psi_\rho(y) := \langle \pi_\rho(y)(\xi_1 \otimes \xi_\rho),\xi_1 \otimes \xi_\rho \rangle\]
defines a ground state on $C^*(S)$ with $\psi_\rho \circ \varphi = \rho$.
\end{proposition}
\begin{proof}
The map $\psi_\rho$ is a state on $C^*(S)$ with $\psi_\rho \circ \varphi = \rho$, see Theorem~\ref{thm:rep of C*(S)}. By Lemma~\ref{lem:eval gen}, we see that
\[\langle V_t^*(\xi_1 \otimes \xi_\rho), V_s^*(\xi_1 \otimes \xi_\rho) \rangle \neq 0\]
forces $1 \in i(s\CT) \cap i(t\CT)$ for $s,t \in S$. This implies $s \sim 1 \sim t$, and hence $s,t \in S_c$, see Lemma~\ref{lem:equiv forms of (VII)}. Thus $\psi_\rho$ is a ground state by virtue of Proposition~\ref{prop:ground states on C*(S)-gen}.
\end{proof}

Before we construct $\kms_\beta$-states on $C^*(S)$ from traces $\tau$ on $C^*(S_c)$ with the help of the representation $\pi_\tau$, we make note of an intermediate step which produces states on $C^*(S)$ from traces on $C^*(S_c)$. In terms of notation, note that the finite subsets of $\text{Irr}(N(S))$ form a (countable) directed set when ordered by inclusion.

\begin{lemma}\label{lem:construction of KMS-states}
Let $\beta \geq 1$, $\tau$ a trace on $C^*(S_c)$, and $I \subset \text{Irr}(N(S))$ with $\zeta_I(\beta) < \infty$. Then
\[\begin{array}{lrl}
\psi_{\beta,\tau,I}(y) &:=& \zeta_I(\beta)^{-1} \sum\limits_{t \in \CT_I} N_t^{-\beta} \langle \pi_\tau(y)(\xi_t \otimes \xi_\tau),\xi_t \otimes \xi_\tau \rangle 
\end{array}\]
defines a state on $C^*(S)$ such that $\psi_{\beta,\tau,I} \circ \varphi$ is tracial on $C^*(S_c)$.
\end{lemma}
\begin{proof}
From the construction in Theorem~\ref{thm:rep of C*(S)}, it is clear that $\psi_{\beta,\tau,I}$ is a state on $C^*(S)$. We claim that
\begin{equation}\label{eq:trace on level n}
\begin{array}{c}
\sum\limits_{t \in \CT_n} \langle V_a^{\phantom{*}}V_b^*V_d^{\phantom{*}}V_e^*(\xi_t \otimes \xi_\tau),\xi_t \otimes \xi_\tau \rangle = \sum\limits_{t \in \CT_n} \langle V_d^{\phantom{*}}V_e^*V_a^{\phantom{*}}V_b^*(\xi_t \otimes \xi_\tau),\xi_t \otimes \xi_\tau \rangle
\end{array}
\end{equation}
holds for all $a,b,d,e \in S_c$ and $n \in \langle I \rangle$. Since $bS\cap dS\neq\emptyset$, Lemma~\ref{lem:eval gen} (ii) implies that
\[\langle V_a^{\phantom{*}}V_b^*V_d^{\phantom{*}}V_e^*(\xi_t \otimes \xi_\tau),\xi_t \otimes \xi_\tau \rangle\]
vanishes for $t \in \CT_n$, unless $t$ belongs to
\[\CT_n^{(a,b),(d,e)} := \CT_n \cap \text{Fix}(\alpha_{a}^{\phantom{1}}\alpha_{b}^{-1}\alpha_{d}^{\phantom{1}}\alpha_{e}^{-1});
\]
just write $bb'=dd'$ and invoke the inner product formula for $V^{\phantom{*}}_{ab'}V^*_{ed'}$ to decide when it vanishes. The bijection $f:=\alpha_{d}^{\phantom{1}}\alpha_{e}^{-1}$ of $\CT$ restricts to a bijection from $\CT_n^{(a,b),(d,e)}$ to $\CT_n^{(d,e),(a,b)}$ as $t \in \CT_n^{(a,b),(d,e)}$ implies
\[f(t) = \alpha_{d}^{\phantom{1}}\alpha_{e}^{-1}(\alpha_{a}^{\phantom{1}}\alpha_{b}^{-1}\alpha_{d}^{\phantom{1}}\alpha_{e}^{-1}(t)) = \alpha_{d}^{\phantom{1}}\alpha_{e}^{-1}\alpha_{a}^{\phantom{1}}\alpha_{b}^{-1}(f(t)) \in \CT_n^{(d,e),(a,b)}.\]
For $t \in \CT_n^{(a,b),(d,e)}$, Lemma~\ref{lem:eval gen} (ii) yields
\[\langle V_a^{\phantom{*}}V_b^*V_d^{\phantom{*}}V_e^*(\xi_t \otimes \xi_\tau),\xi_t \otimes \xi_\tau \rangle
= \tau(w_{c(ar)}^{\phantom{*}}w_{c(br)}^*w_{c(ds)}^{\phantom{*}}w_{c(es)}^*),\]
where $s,r \in \CT_n$ are given by $s = \alpha_{e}^{-1}(t)$ and $r = \alpha_{b}^{-1}\alpha_{d}^{\phantom{1}}\alpha_{e}^{-1}(t) = \alpha_{a}^{-1}(t)$. Likewise, we have
\[\langle V_d^{\phantom{*}}V_e^*V_a^{\phantom{*}}V_b^*(\xi_{f(t)} \otimes \xi_\tau),\xi_{f(t)} \otimes \xi_\tau \rangle
= \tau(w_{c(ds')}^{\phantom{*}}w_{c(es')}^*w_{c(ar')}^{\phantom{*}}w_{c(br')}^*),\]
where $s',r' \in \CT_n$ are given by
\[r' = \alpha_{b}^{-1}f(t) = \alpha_{b}^{-1}\alpha_{d}^{\phantom{1}}\alpha_{e}^{-1}(t) = \alpha_{a}^{-1}(t) = r,\]
and
\[s' = \alpha_{e}^{-1}\alpha_{a}^{\phantom{1}}\alpha_{b}^{-1}f(t) = \alpha_{e}^{-1}\alpha_{a}^{\phantom{1}}\alpha_{b}^{-1}\alpha_{d}^{\phantom{1}}\alpha_{e}^{-1}(t) = \alpha_{e}^{-1}(t) = s.\]
Since $\tau$ is a trace, we thus have
\[\begin{array}{lcl}
\langle V_d^{\phantom{*}}V_e^*V_a^{\phantom{*}}V_b^*(\xi_{f(t)} \otimes \xi_\tau),\xi_{f(t)} \otimes \xi_\tau \rangle
&=& \tau(w_{c(ds)}^{\phantom{*}}w_{c(es)}^*w_{c(ar)}^{\phantom{*}}w_{c(br)}^*) \vspace*{2mm}\\
&=& \tau(w_{c(ar)}^{\phantom{*}}w_{c(br)}^*w_{c(ds)}^{\phantom{*}}w_{c(es)}^*) \vspace*{2mm}\\
&=& \langle V_a^{\phantom{*}}V_b^*V_d^{\phantom{*}}V_e^*(\xi_t \otimes \xi_\tau),\xi_t \otimes \xi_\tau \rangle
\end{array}\]
as claimed. This establishes \eqref{eq:trace on level n}, and hence that $\psi_{\beta,\tau,I} \circ \varphi$ is a trace on $C^*(S_c)$.
\end{proof}

\begin{proposition}\label{prop:construction of KMS-states}
For $\beta \geq 1$ and every trace $\tau$ on $C^*(S_c)$, there is a sequence $(I_k)_{k \geq 1}$ of finite subsets of $\text{Irr}(N(S))$ such that $\psi_{\beta,\tau,I_k}$ weak*-converges to a $\kms_\beta$-state $\psi_{\beta,\tau}$ on $C^*(S)$ as $I_k \nearrow \text{Irr}(N(S))$. For $\beta \in (\beta_c,\infty)$, the state $\psi_{\beta,\tau}$ is given by $\psi_{\beta,\tau,\text{Irr}N(S)}$.
\end{proposition}
\begin{proof}
Due to weak*-compactness of the state space on $C^*(S)$, there is a sequence $(I_k)_{k \geq 1}$ of finite subsets of $\text{Irr}(N(S))$ with $I_k \nearrow \text{Irr}(N(S))$ such that $(\psi_{\beta,\tau,I_k})_{k \geq 1}$ obtained from Lemma~\ref{lem:construction of KMS-states} converges to some state $\psi_{\beta,\tau}$ in the weak* topology. By Proposition~\ref{prop:algebraic characterization of KMS states-gen}, $\psi_{\beta,\tau}$ is a $\kms_\beta$-state if and only if $\psi_{\beta,\tau}\circ \varphi$ defines a trace on $C^*(S_c)$ and \eqref{eq:alg char of KMS states-gen with i and c} holds. Note that $\psi_{\beta,\tau} \circ \varphi$ is tracial on $C^*(S_c)$ because each of the $\psi_{\beta,\tau,I_k}$ has this property by Lemma~\ref{lem:construction of KMS-states}. It remains to prove \eqref{eq:alg char of KMS states-gen with i and c}. For fixed $s,r \in S$, Lemma~\ref{lem:eval gen} (i) shows that $\langle V_s^{\phantom{*}}V_r^*(\xi_t \otimes \xi_\tau),\xi_t \otimes \xi_\tau \rangle$ vanishes for all $t \in \CT$, unless $s \sim r$. Therefore $\psi_{\beta,\tau}(v_s^{\phantom{*}}v_r^*) = 0$ if $s \not \sim r$. Now fix $s,r\in S$ such that $s \sim r$. Write $s=t'a$ and $r=t'b$ for some $t' \in \CT, a,b \in S_c$. Then $\langle V_s^{\phantom{*}}V_r^*(\xi_t \otimes \xi_\tau),\xi_t \otimes \xi_\tau \rangle$ can only be nonzero if $t \in i(t'\CT) \cap \CT$. Note that for $t''\in \CT$ such that $t=i(t't'')$ we have $tc(t't'')=t't''$, so $N_t=N_{t'}N_{t''}$. For large enough $k$, we have $t' \in \CT_{I_k}$. Hence, as we sum over $t\in \CT_{I_k}$, for $k$ large enough also $t''\in \CT_{I_k}$. Thus
\[\begin{array}{lcl}
\psi_{\beta,\tau}(v_s^{\phantom{*}}v_r^*)
&=&\lim\limits_{k\to\infty} \zeta_{I_k}(\beta)^{-1} \sum\limits_{t \in \CT_{I_k}} N_t^{-\beta} \langle V_s^{\phantom{*}}V_r^*(\xi_t \otimes \xi_\tau),\xi_t \otimes \xi_\tau \rangle \vspace*{2mm}\\
&=& \lim\limits_{k\to\infty}\zeta_{I_k}(\beta)^{-1} \sum\limits_{t \in \CT_{I_k}} N_t^{-\beta} \langle V_a^{\phantom{*}}V_b^*V_{t'}^*(\xi_t \otimes \xi_\tau),V_{t'}^*(\xi_t \otimes \xi_\tau) \rangle \vspace*{2mm}\\
&=& \lim\limits_{k\to\infty}\zeta_{I_k}(\beta)^{-1} \sum\limits_{t'' \in \CT_{I_k}} N_{i(t't'')}^{-\beta} \langle V_a^{\phantom{*}}V_b^*(\xi_{t''}w_{c(t't'')}^* \otimes \xi_\tau),\xi_{t''}w_{c(t't'')}^* \otimes \xi_\tau \rangle \\
&=&\lim\limits_{k\to\infty} N_{t'}^{-\beta} \zeta_{I_k}(\beta)^{-1} \sum\limits_{t'' \in \CT_{I_k}} N_{t''}^{-\beta} \langle V_a^{\phantom{*}}V_b^*(\xi_{t''} \otimes \xi_\tau),\xi_{t''} \otimes \xi_\tau \rangle \vspace*{2mm}\\
&=& \lim\limits_{k\to\infty} N_s^{-\beta} \psi_{\beta,\tau,I_k}(v_a^{\phantom{*}}v_b^*),
\end{array}\]
where we used Lemma~\ref{lem:eval gen} (ii) and the trace property of $\tau$ in the penultimate equality. This shows that the limit $\psi_{\beta,\tau}$ satisfies \eqref{eq:alg char of KMS states-gen with i and c}, and hence is a $\kms_\beta$-state. The claim for $\beta \in (\beta_c,\infty)$ follows immediately from $I_k \nearrow \text{Irr}(N(S))$ because the formula from Lemma~\ref{lem:construction of KMS-states} makes sense for $I = \text{Irr}(N(S))$.
\end{proof}

\begin{corollary}\label{cor:KMS-states parametrization surjective}
Let $S$ be an admissible right LCM semigroup. If $\phi$ is a $\kms_\beta$-state on $C^*(S)$ for $\beta \in (\beta_c,\infty)$, then $\psi_{\beta,\phi_{\Irr(N(S))} \circ \varphi} = \phi$.
\end{corollary}
\begin{proof}
As $\phi_{\Irr(N(S))} \circ \varphi$ is a trace on $C^*(S_c)$ by Proposition~\ref{prop:algebraic characterization of KMS states-gen}, $\psi_{\beta,\phi_{\Irr(N(S))} \circ \varphi}$ is a $\kms_\beta$-state, see Proposition~\ref{prop:construction of KMS-states}. Due to Proposition~\ref{prop:algebraic characterization of KMS states-gen}, it suffices to show that $\psi_{\beta,\phi_{\Irr(N(S))} \circ \varphi} \circ \varphi = \phi \circ \varphi$. For $a,b \in S_c$, Lemma~\ref{lem:eval gen}~(ii) and $\varphi(w_aw_b^*)=v_av_b^*$ give
\[\begin{array}{ll}
\psi_{\beta,\phi_{\Irr(N(S))} \circ \varphi}(v_a^{\phantom{*}}v_b^*) = \zeta_S(\beta)^{-1} \sum\limits_{t \in \CT} N_t^{-\beta} \langle V_a^{\phantom{*}}V_b^*(\xi_t \otimes \xi_{\phi_{\Irr(N(S))} \circ \varphi}),\xi_t \otimes \xi_{\phi_{\Irr(N(S))} \circ \varphi} \rangle \vspace*{2mm}\\
&\hspace*{-120mm}= \zeta_S(\beta)^{-1} \sum\limits_{n \in N(S)} \sum\limits_{t \in \CT_n} n^{-\beta} \chi_{\text{Fix}(\alpha_a^{\phantom{1}}\alpha_b^{-1})}(t) \ \phi_{\Irr(N(S))} \circ \varphi(w_{c(a\alpha_b^{-1}(t))}^{\phantom{*}}w_{c(b\alpha_b^{-1}(t))}^*).
\end{array}\]
We claim that $\chi_{\text{Fix}(\alpha_a^{\phantom{1}}\alpha_b^{-1})}(t) \ \varphi(w_{c(a\alpha_b^{-1}(t))}^{\phantom{*}}w_{c(b\alpha_b^{-1}(t))}^*) = v_t^*v_a^{\phantom{*}}v_b^*v_t^{\phantom{*}}$. Indeed, the expression $v_t^*v_a^{\phantom{*}}v_b^*v_t^{\phantom{*}}$ vanishes unless $t \in \text{Fix}(\alpha_a^{\phantom{1}}\alpha_b^{-1})$. But if $t \in \text{Fix}(\alpha_a^{\phantom{1}}\alpha_b^{-1})$, then $a\alpha_b^{-1}(t) = tc(a\alpha_b^{-1}(t))$ and $b\alpha_b^{-1}(t) = tc(b\alpha_b^{-1}(t))$, see the proof of Lemma~\ref{lem:thebijectionsg}, hence
\[\begin{array}{lcl}
v_t^*v_a^{\phantom{*}}v_b^*v_t^{\phantom{*}}
&=& v_t^*e_{a\alpha_b^{-1}(t)S}v_a^{\phantom{*}}v_b^*e_{b\alpha_b^{-1}(t)S}v_t^{\phantom{*}} \\
&=& v_t^*v_{ tc(a\alpha_b^{-1}(t))}^{\phantom{*}}v_{a\alpha_b^{-1}(t)}^*v_a^{\phantom{*}}v_b^*v_{b\alpha_b^{-1}(t)}^{\phantom{*}}v_{ tc(b\alpha_b^{-1}(t))}^*v_t^{\phantom{*}} \\
&=& v_{c(a\alpha_b^{-1}(t))}^{\phantom{*}}v_{\alpha_b^{-1}(t)}^*v_{\alpha_b^{-1}(t)}^{\phantom{*}}v_{c(b\alpha_b^{-1}(t))}^*\\
&=& v_{c(a\alpha_b^{-1}(t))}^{\phantom{*}}v_{c(b\alpha_b^{-1}(t))}^*.
\end{array}\]
Therefore we conclude that $\psi_{\beta,\phi_{\Irr(N(S))} \circ \varphi} = \phi$ by appealing to Lemma~\ref{lem:rec formula-gen}~(iv).
\end{proof}

Corollary~\ref{cor:KMS-states parametrization surjective} shows that every $\kms_\beta$-state on $C^*(S)$ for $\beta \in (\beta_c,\infty)$ arises from a trace on $C^*(S_c)$. In other words, we have obtained a surjective parametrisation. We now show that this mapping is also injective.

\begin{proposition}\label{prop:KMS-states parametrization injective}
 Let $S$ be an admissible right LCM semigroup. For every $\beta \in (\beta_c,\infty)$ and every trace $\tau$ on $C^*(S_c)$, let $\psi_{\beta, \tau}$ be the $\kms_\beta$-state obtained in Proposition~\ref{prop:construction of KMS-states}. Then $(\psi_{\beta,\tau})_{\text{Irr}(N(S))} \circ \varphi = \tau$.
\end{proposition}
\begin{proof}
Recall from Lemma~\ref{lem:d_p commutes with v_s-gen} that $\varphi(y)d_n = d_n\varphi(y)$ for all $y \in C^*(S_c)$ and $n \in N(S)$. Using Lemma~\ref{lem:rec formula-gen}, Proposition~\ref{prop:construction of KMS-states} and Lemma~\ref{lem:construction of KMS-states}, we have
\[\begin{array}{ll}
(\psi_{\beta,\tau})_{\text{Irr}(N(S))} \circ \varphi (y) \vspace*{2mm}\\
&\hspace*{-20mm}= \zeta_S(\beta) \tilde{\psi}_{\beta,\tau}(Q_{\text{Irr}(N(S))}\pi_{\psi_{\beta,\tau}}(\varphi(y)) Q_{\text{Irr}(N(S))}) \vspace*{2mm}\\
&\hspace*{-20mm}= \zeta_S(\beta) \tilde{\psi}_{\beta,\tau}(\pi_{\psi_{\beta,\tau}}(\varphi(y)) \prod\limits_{n \in \text{Irr}(N(S))} \pi_{\psi_{\beta,\tau}}(d_n)) \vspace*{2mm}\\
&\hspace*{-20mm}= \zeta_S(\beta) \lim\limits_{\substack{I_k \nearrow \text{Irr}(N(S))\\ \lvert I_k \rvert < \infty}} \psi_{\beta,\tau}(\varphi(y) \prod\limits_{n \in I_k} d_n) \vspace*{2mm}\\
&\hspace*{-20mm}= \lim\limits_{\substack{I_k \nearrow \text{Irr}(N(S))\\ \lvert I_k \rvert < \infty}} \sum\limits_{m \in N(S)}m^{-\beta}\sum\limits_{t \in \CT_m} \langle \pi_\tau(\varphi(y)\prod\limits_{n \in I_k} d_n)(\xi_t \otimes \xi_\tau),\xi_t \otimes \xi_\tau \rangle.
\end{array}\]
Noting that $\CT_1=\{1\}$ and $\pi_\tau(\prod_{n \in I_k} d_n)(\xi_1 \otimes \xi_\tau) = \xi_1 \otimes \xi_\tau$ as $1 \notin tS$ for all $t \in \CT_n$, where $n \in I_k \subset \text{Irr}(N(S))$, the term at $m=1$ is given by
\[\begin{array}{lclcl}
\langle \pi_\tau(\varphi(y)\prod\limits_{n \in I_k} d_n)(\xi_1 \otimes \xi_\tau),\xi_1 \otimes \xi_\tau \rangle
&=& \langle \pi_\tau(\varphi(y))(\xi_t \otimes \xi_\tau),\xi_t \otimes \xi_\tau \rangle 
 &=& \tau(y).
\end{array}\]
For $m \in N(S)$ with $m>1$, there is $k$ so that there is $n \in I_\ell$ with $m \in nN(S)$ for all $\ell \geq k$. We claim that $\pi_\tau(d_n)(\xi_t \otimes \xi_\tau) = 0$ for all $t \in \CT_m$. To prove the claim, we recall that $d_n = 1 - \sum_{f \in \CT_n} e_{fS}$ and that $\pi_\tau(e_{fS})(\xi_t \otimes \xi_\tau) = \chi_{i(fS)}(t) \ \xi_t \otimes \xi_\tau$. Let $n' \in N(S)$ such that $m=nn'$. Let $\CT_{n'}$ be any transversal for $N^{-1}(n')/_\sim$ contained in $S_{ci}^1$. According to Corollary~\ref{cor:dichotomy for S_ci under gen scale}, we either have $ff' \perp t$ or $ff' \in tS^*$ for every pair $(f,f') \in \CT_n \times \CT_{n'}$ as $\CT_n\CT_{n'} \subset N^{-1}(m)\cap S_{ci}$. However, since $\CT_n\CT_{n'}$ is an accurate foundation set, there is exactly one such pair $(f,f')$ satisfying $ff' \in tS^*$. This forces $\pi_\tau(d_n)(\xi_t \otimes \xi_\tau) = 0$ and thus $\langle \pi_\tau(\varphi(y)\prod_{n \in I_\ell} d_n)(\xi_t \otimes \xi_\tau),\xi_t \otimes \xi_\tau \rangle = 0$ for all $\ell \geq k$. Thus we get $(\psi_{\beta,\tau})_{\text{Irr}(N(S))} \circ \varphi (y) = \tau(y)$.
\end{proof}

\section{Uniqueness within the critical interval}\label{sec:uniqueness}
This section deals with the proof of the uniqueness assertions in Theorem~\ref{thm:KMS results-gen}~(2). For the purpose of this section, we assume throughout that $S$ is an admissible right LCM semigroup, and write $\tau$ for the canonical trace $\tau(w_s^{\phantom{*}}w_t^*) = \delta_{s, t}$ on $C^*(S_c)$. We begin with the case that $\alpha$ is an almost free action, where uniqueness will be an application of the reconstruction formula in Lemma~\ref{lem:rec formula-gen}~(iv):

\begin{proposition}\label{prop:strong min gives unique kms in the crit int}
Let $S$ be an admissible right LCM semigroup, and $\beta \in [1,\beta_c]$. If $\alpha\colon S_c \curvearrowright S/_\sim$ is almost free, then the only $\kms_\beta$-state on $C^*(S)$ is $\psi_\beta := \psi_{\beta,\tau}$ determined by $\psi_\beta \circ \varphi = \tau$.
\end{proposition}
\begin{proof}
By Proposition~\ref{prop:construction of KMS-states}, $\psi_\beta$ is a $\kms_\beta$-state. Let $\phi$ be an arbitrary $\kms_\beta$-state. Due to Proposition~\ref{prop:algebraic characterization of KMS states-gen}, it suffices to show that $\psi_\beta \circ \varphi = \phi \circ \varphi$ on $C^*(S_c)$. Since $\phi \circ \varphi$ is a trace, we have $\psi_\beta \circ \varphi (w_a^{\phantom{*}}w_a^*)=1$ for all $a \in S_c$. Let $a,b \in S_c$ with $a \neq b$. Lemma~\ref{lem:rec formula-gen}~(iv) gives
\[\begin{array}{lcl}
\phi (v_a^{\phantom{*}}v_b^*) &=& \zeta_I(\beta)^{-1} \sum\limits_{t \in \CT_I} N_t^{-\beta} \ \phi_I(v_t^*v_a^{\phantom{*}}v_b^*v_t^{\phantom{*}})
\end{array}\]
for every finite subset $I$ of $\text{Irr}(N(S))$. Note that $|\phi_I(v_t^*v_a^{\phantom{*}}v_b^*v_t^{\phantom{*}})| \leq 1$. We claim that a summand can only be nonzero if $\alpha_a^{\phantom{1}}\alpha_b^{-1}(t)=t$. Indeed, $v_t^*v_a^{\phantom{*}} = v_{a_1}^{\phantom{*}}v_{a_2}^*v_{\alpha_a^{-1}(t)}^*$ for some $a_1,a_2 \in S_c$, and similarly for $v_b^*v_t^{\phantom{*}}$. Thus a summand is nonzero if $v_{\alpha_a^{-1}(t)}^*v_{\alpha_b^{-1}(t)}^{\phantom{*}} \neq 0$, which is equivalent to $\alpha_a^{-1}(t) = \alpha_b^{-1}(t)$, that is $t \in \text{Fix}(\alpha_a^{\phantom{1}}\alpha_b^{-1})$.

As $\alpha$ is almost free, the set $\text{Fix}(\alpha_a^{\phantom{1}}\alpha_b^{-1}) \subset \CT$ is finite, so that
\[\begin{array}{c}
\bigl\lvert\sum\limits_{t \in F_I} N_t^{-\beta} \ \phi_I(v_t^*v_a^{\phantom{*}}v_b^*v_t^{\phantom{*}})\bigr\rvert \leq \lvert \text{Fix}(\alpha_a^{\phantom{1}}\alpha_b^{-1}) \rvert
\end{array}\]
while $\beta \in [1,\beta_c]$ forces $\zeta_I(\beta) \to \infty$ as $I \nearrow \text{Irr}(N(S))$. Thus we conclude that $\phi (v_a^{\phantom{*}}v_b^*) = 0$ for all $a,b \in S_c$ with $a \neq b$. In particular, there is only one such state $\psi_\beta$.
\end{proof}

The case where $\beta_c=1$, $\alpha$ is faithful, and $S$ has finite propagation requires more work. The underlying strategy is based on \cite{LRRW}*{Section~7}. Given $n \in N(S)$, let $\CT_n$ be a transversal for $N^{-1}(n)/_\sim$ contained in $S_{ci}^1$. Note that we may take $\CT_n$ of the form $\CT\cap N^{-1}(n)$, where $\CT$ is a transversal for $S/_\sim$ contained in $S_{ci}^1$. For $a,b \in S_c$, let
\[
\CT_n^{a,b} := \{ f \in \CT_n \mid af=bf \} \quad \text{and}\quad \kappa_{a,b,n} := \lvert \CT_n^{a,b}\rvert/n.
\]
Note that by Corollary~\ref{cor:dichotomy for S_ci under gen scale}, any other choice of transversal for $N^{-1}(n)/_\sim$ contained in $S_{ci}^1$ will give the same $\kappa_{a,b,n}$. If moreover $m\in N(S)$, then
since $i(\CT_m\CT_n) = \CT_{mn}$ and since $f \sim f'$ for core irreducible elements $f,f'$ implies $f \in f'S^*$, see Lemma~\ref{lem:equiv forms of (VII)}, we have $\CT_{mn}^{a,b} \supset i(\CT_m^{a,b}\CT_n)$ so that $(\kappa_{a,b,n})_{n \in N(S)}$ is increasing for the natural partial order $n' \geq n :\Leftrightarrow n' \in nN(S)$. As $N(S)$ is directed, we get a limit $\kappa_{a,b} = \lim_{n \in N(S)} \kappa_{a,b,n} \in [0,1]$ for all $a,b \in S_c$.

\begin{proposition}\label{prop:min gives unique kms in the crit int}
Let $S$ be an admissible right LCM semigroup and $\beta_c=1$. Then there exists a $\kms_1$-state $\psi_1$ characterized by $\psi_1 \circ \varphi = \rho$, where $\rho(w_a^{\phantom{*}}w_b^*) = \kappa_{a,b}$ for $a,b \in S_c$. In particular, $\rho$ is a trace on $C^*(S_c)$. If $\alpha$ is faithful and $S$ has finite propagation, then $\psi_1$ is the only $\kms_1$-state on $C^*(S)$.
\end{proposition}

The idea of the proof is as follows: Take $(\beta_k)_{k \geq 1}$ with $\beta_k > 1$ and $\beta_k \searrow 1$. Then a subsequence of the sequence of $\kms_{\beta_k}$-states $(\psi_{\beta_k,\tau})_{k \geq 1}$ obtained from Proposition~\ref{prop:construction of KMS-states} weak*-converges to a state $\psi_1$ on $C^*(S)$. For simplicity, we relabel and then assume that the original sequence converges. It is well-known that $\psi_1$ is a $\kms_1$-state in this case, see for example \cite{BRII}*{Proposition~5.3.23}. We then use Proposition~\ref{prop:KMS-state factoring through Q_p must have beta=1} to deduce that every $\kms_1$-state $\phi$ satisfies $\phi \circ \varphi = \rho$.

We start with two preparatory lemmas that will streamline the proof of Proposition~\ref{prop:min gives unique kms in the crit int}.

\begin{lemma}\label{lem:convergence under series perturbation}
Suppose $S$ is an admissible right LCM semigroup, $\beta_c=1$, and $\beta_k \searrow 1$. Then $\zeta_S(\beta_k)^{-1} \sum_{n \in N(S)} n^{-\beta_k+1} \kappa_{a,b,n}$ converges to $\kappa_{a,b}$ as $k \to \infty$ for all $a,b \in S_c$.
\end{lemma}
\begin{proof}
Let $a,b \in S_c$ and $\varepsilon > 0$. As $(\kappa_{a,b,n})_{n \in N(S)}$ is increasing and converges to $\kappa_{a,b}$, and $N(S) \cong \bigoplus_{n \in \text{Irr}(N(S))} \N$, there are $C \in \N, n_1,\ldots,n_C \in \text{Irr}(N(S)),m_1,\ldots,m_C \geq 1$ such that $n \in \tilde{n}N(S)$ implies $0 \leq \kappa_{a,b} -\kappa_{a,b,n} < \varepsilon/2$, where $\tilde{n} := \prod_{1 \leq i \leq C}n_i^{m_i}$. Next, note that
\begin{enumerate}[a)]
\item $N(S)\setminus \tilde{n}N(S) \subset \bigcup\limits_{1 \leq i \leq C}\bigcup\limits_{0 \leq m \leq m_i-1} n_i^m\{n \in N(S) \mid n \notin n_iN(S)\}$, and\vspace*{2mm}
\item $\zeta_S(\beta_k) = \prod\limits_{n \in \text{Irr}(N(S))} (1-n^{-\beta_k+1})^{-1}$.
\end{enumerate}
In particular, b) implies that
\[\begin{array}{lcl}
\zeta_S(\beta_k)^{-1}\sum\limits_{\substack{n \in N(S) \\ n \notin n_iN(S)}} n^{-\beta_k+1} &=& (1-n_i^{-\beta_k+1})^{-1}\zeta_{\text{Irr}(N(S))\setminus \{n_i\}}(\beta_k)^{-1}\sum\limits_{n \in \langle \text{Irr}(N(S))\setminus \{n_i\}\rangle} n^{-\beta_k+1} \\
&=& (1-n_i^{-\beta_k+1})^{-1} \stackrel{k \to \infty}{\longrightarrow} 0
\end{array}\]
for each $1 \leq i \leq C$ because $\beta_k \searrow 1$. Now choose $\ell$ large enough so that
\begin{enumerate}
\item[c)] $\zeta_S(\beta_k)^{-1}\sum_{n \in N(S)\setminus n_iN(S)} n^{-\beta_k+1} \leq \varepsilon/2Cm_i$ for all $1 \leq i \leq C$ and $k \geq \ell$.
\end{enumerate}
By definition, $\zeta_S(\beta_k) = \sum_{n \in N(S)} n^{-\beta_k+1}$, so we get
\[\begin{array}{lcl}
\bigl|\kappa_{a,b} - \zeta_S(\beta_k)^{-1} \sum\limits_{n \in N(S)} n^{-\beta_k+1} \kappa_{a,b,n}\bigr| \vspace*{1mm}\\
&\hspace*{-60mm}\stackrel{\kappa_{a,b} \geq \kappa_{a,b,n}}{=}& \zeta_S(\beta_k)^{-1} \sum\limits_{n \in N(S)} n^{-\beta_k+1} (\kappa_{a,b} - \kappa_{a,b,n}) \vspace*{2mm}\\
&\hspace*{-60mm}\stackrel{a)}{\leq}& \zeta_S(\beta_k)^{-1} \sum\limits_{n \in \tilde{n}N(S)} n^{-\beta_k+1} \underbrace{(\kappa_{a,b} - \kappa_{a,b,n})}_{< \varepsilon/2} \vspace*{2mm}\\
&\hspace*{-60mm}&+ \sum\limits_{i=1}^{C}\sum\limits_{m=0}^{m_i-1} \zeta_S(\beta_k)^{-1} \sum\limits_{\substack{n \in N(S)\\ n \notin n_iN(S)}} n^{-\beta_k+1}\underbrace{n_i^{m(-\beta_k+1)} (\kappa_{a,b} - \kappa_{a,b,nn_i^m})}_{\in [0,1]} \vspace*{2mm}\\
&\hspace*{-60mm}\stackrel{c)}{<}& \frac{\varepsilon}{2} + \sum\limits_{i=1}^{C}\sum\limits_{m=0}^{m_i-1} \frac{\varepsilon}{2Cm_i} = \varepsilon
\end{array}\]
for all $k \geq \ell$.
\end{proof}

\begin{remark}\label{rem:uniqueness absorption of core difference modulo core}
In analogy to $\CT_n^{a,b} = \{ f \in \CT_n \mid af=bf\} = \{ f \in \CT_n \mid i(af)=i(bf) \text{ and } c(af)=c(bf)\}$, let $G_n^{a,b} := \{ f \in \CT_n \mid i(af)=i(bf)\}$ for $n \in N(S)$ and $a,b \in S_c$. Note that $i(af)=i(bf)$ is equivalent to $\alpha_a([f]) = \alpha_b([f])$. By Corollary~\ref{cor:dichotomy for S_ci under gen scale}, the number $|G_n^{a,b}|$ does not depend on $\CT_n$. Also, for $m,n \in N(S)$, we have $\CT_{mn} = i(\CT_m\CT_n)$ and $aff' = i(af)i(c(af)f')c(c(af)f')$. By using Corollary~\ref{cor:dichotomy for S_ci under gen scale}, this allows us to deduce
\begin{equation}\label{eq:product rule for G and G without F}
G_{mn}^{a,b}\setminus \CT_{mn}^{a,b} = \{ i(ff') \mid f \in G_m^{a,b}\setminus \CT_m^{a,b}, f' \in G_n^{c(af),c(bf)}\setminus \CT_n^{c(af),c(bf)} \}.
\end{equation}
\end{remark}

\begin{lemma}\label{lem:minimality keeps G without F small}
 Let $S$ be an admissible right LCM semigroup. If $\alpha$ is faithful and $S$ has finite propagation, then $(\lvert G_n^{a,b} \setminus \CT_n^{a,b} \rvert/n)_{n \in N(S)}$ converges to $0$ for all $a,b \in S_c$.
\end{lemma}
\begin{proof}
Let $a,b \in S_c$. If $a=b$, then $G_n^{a,b} = \CT_n^{a,b} = \CT_n$ for all $n \in N(S)$, so there is nothing to show. Thus, we may suppose $a \neq b$. Since $S$ has finite propagation, there exists a transversal $\CT$ for $S/_\sim$ that witnesses finite propagation for $(a,b)$, i.e.~the sets $C_a = \{ c(at) \mid t \in \CT\}$ and $C_b$ are finite. In particular, the set $C_{a,b} := \{ (d,e) \in C_a \times C_b \mid d \neq e\} \subset S_c \times S_c$ is finite. We claim that for each $(d,e) \in C_{a,b}$, faithfulness of $\alpha$ gives $n(d,e) \in N(S)$ such that $\lvert G_{n(d,e)}^{d,e} \rvert\leq n(d,e) -1$. Indeed, we get $r \in S$ with $dr \not\sim er$. This forces $n(d,e) := N_r>1$ and $di(r) \perp ei(r)$ by Proposition~\ref{prop:NisaHMofRLCMs}~(ii). In particular, we have $i(di(r)) \neq i(ei(r))$.

With $n := \prod_{(d,e) \in C_{a,b}} n(d,e)$ and $\gamma := \max_{(d,e) \in C_{a,b}} (n(d,e)-1)/n(d,e) \in (0,1)$, we claim that $n^{-k} \lvert G_{n^k}^{a,b} \setminus \CT_{n^k}^{a,b} \rvert \leq \gamma^k$ holds for all $k \geq 1$. We proceed by induction. For $k=1$, we use $(a,b) \in C_{a,b}$ and \eqref{eq:product rule for G and G without F} for $n(a,b)$ to get
\[\begin{array}{c}
n^{-1} \lvert G_{n}^{a,b} \setminus \CT_{n}^{a,b} \rvert \leq \frac{(n(a,b)-1)}{n(a,b)} \leq \gamma.
\end{array}\]
For $k \mapsto k+1$, let us consider \eqref{eq:product rule for G and G without F} for $m:=n^k$ and $m'=n$. The induction hypothesis gives $\lvert G_{n^k}^{a,b} \setminus \CT_{n^k}^{a,b} \rvert \leq (\gamma n)^k$, so it suffices to have $\lvert G_n^{c(af),c(bf)}\setminus \CT_n^{c(af),c(bf)} \rvert \leq \gamma n$ for every $f \in G_{n^k}^{a,b}\setminus \CT_{n^k}^{a,b}$. As before, if $c(af)=c(bf)$, then the set in question is empty. But if $c(af)\neq c(bf)$, then the choice of $n$ guarantees that $\lvert G_n^{c(af),c(bf)}\setminus \CT_n^{c(af),c(bf)} \rvert \leq \gamma n$ by applying \eqref{eq:product rule for G and G without F} to $m=n(c(af),c(bf))$ and $m' = n(c(af),c(bf))^{-1}n$). This establishes the claim. In turn, $n^{-k} \lvert G_{n^k}^{a,b} \setminus \CT_{n^k}^{a,b} \rvert \leq \gamma^k$ for all $k \geq 1$ forces $\lvert G_m^{a,b} \setminus \CT_m^{a,b} \rvert/m \to 0$ as $m \to \infty$ in $N(S)$ with respect to the natural partial order.
\end{proof}

\begin{proof}[Proof of Proposition~\ref{prop:min gives unique kms in the crit int}]
We know that $\psi_1$ is a $\kms_1$-state, so we start by showing that $\psi_1 \circ \varphi (w_a^{\phantom{*}}w_b^*)= \rho(w_a^{\phantom{*}}w_b^*) = \kappa_{a,b}$ for $a,b \in S_c$. Note that $\psi_{\beta_k,\tau}(v_a^{\phantom{*}}v_b^*)=\psi_{\beta_k,\tau}(v_b^*v_a^{\phantom{*}})$ since $\psi_{\beta_k,\tau} \circ \varphi$ is tracial. The reconstruction formula from Lemma~\ref{lem:rec formula-gen}~(iv) and Proposition~\ref{prop:KMS-states parametrization injective} imply that
\[\begin{array}{lcl}
\psi_{\beta_k,\tau}(v_b^*v_a^{\phantom{*}}) &=& \zeta_S(\beta_k)^{-1} \sum\limits_{n \in N(S)} n^{-\beta_k}\sum\limits_{f \in \CT_n} (\psi_{\beta_k,\tau})_{\text{Irr}(N(S))}(v_{bf}^*v_{af}^{\phantom{*}}) \vspace*{2mm}\\
&=& \zeta_S(\beta_k)^{-1} \sum\limits_{n \in N(S)} n^{-\beta_k}\sum\limits_{f \in \CT_n} \delta_{i(af), i(bf)}(\psi_{\beta_k,\tau})_{\text{Irr}(N(S))}(v_{c(bf)}^*v_{c(af)}^{\phantom{*}}) \vspace*{2mm}\\
&\stackrel{\ref{prop:KMS-states parametrization injective}}{=}& \zeta_S(\beta_k)^{-1} \sum\limits_{n \in N(S)} n^{-\beta_k}\sum\limits_{f \in \CT_n} \delta_{i(af), i(bf)}\tau(w_{c(bf)}^*w_{c(af)}^{\phantom{*}}) \vspace*{2mm}\\
&=& \zeta_S(\beta_k)^{-1} \sum\limits_{n \in N(S)} n^{-\beta_k} \lvert \CT_n^{a,b}\rvert.
\end{array}\]
Since $n^{-\beta_k} \lvert \CT_n^{a,b}\rvert = n^{-\beta_k+1} \kappa_{a,b,n}$, Lemma~\ref{lem:convergence under series perturbation} and $\psi_{\beta_k,\tau} \to \psi_1$ jointly imply $\psi_1 \circ \varphi(w_a^{\phantom{*}}w_b^*) = \kappa_{a,b} = \rho(w_a^{\phantom{*}}w_b^*)$. According to Proposition~\ref{prop:algebraic characterization of KMS states-gen}, we conclude that $\rho$ is a trace on $C^*(S_c)$.

Now suppose $\phi$ is any $\kms_1$-state on $C^*(S)$. Again appealing to Proposition~\ref{prop:algebraic characterization of KMS states-gen}, it suffices to show that $\phi \circ \varphi = \rho$ on $C^*(S_c)$. Since $\phi$ is a $\kms_1$-state, it factors through $\CQ_p(S)$ (or $\CQ_(S)$), see Proposition~\ref{prop:KMS-state factoring through Q_p must have beta=1}. As $\pi_p(\sum_{f \in \CT_n} v_f^{\phantom{*}}v_f^*) = 1$ for all $n \in N(S)$, we get
\[\begin{array}{lclclcl}
\phi(v_a^{\phantom{*}}v_b^*) &=& \sum\limits_{f \in \CT_n} \phi(v_{af}^{\phantom{*}}v_{bf}^*) 
&=& \sum\limits_{f \in \CT_n} n^{-1} \delta_{i(af) i(bf)} \phi(v_{c(af)}^{\phantom{*}}v_{c(bf)}^*) \vspace*{2mm}\\
&&&=& \kappa_{a,b,n} + n^{-1}\sum\limits_{f \in G_n^{a,b}\setminus \CT_n^{a,b}}\phi(v_{c(af)}^{\phantom{*}}v_{c(bf)}^*).
\end{array}\]
Since $|\phi(v_{c(af)}^{\phantom{*}}v_{c(bf)}^*)| \leq 1$, Lemma~\ref{lem:minimality keeps G without F small} shows that $\phi(v_a^{\phantom{*}}v_b^*) = \lim_{n \in N(S)} \kappa_{a,b,n} = \rho(w_a^{\phantom{*}}w_b^*)$.
\end{proof}

While Proposition~\ref{prop:strong min gives unique kms in the crit int} and Proposition~\ref{prop:min gives unique kms in the crit int} provide sufficient criteria for uniqueness of $\kms_\beta$-states for $\beta$ in the critical interval, we will now prove that faithfulness of $\alpha$ is necessary, at least under the assumption that $S_c$ is abelian. If $S_c$ is abelian, it is also right cancellative, and hence a right Ore semigroup. It therefore embeds into a group $G_c = S_c^{\phantom{1}}S_c^{-1}$, and the action $\alpha\colon S_c \curvearrowright S/_\sim$ induces an action $\alpha'\colon G_c \curvearrowright S/_\sim$ by $\alpha'_{ab^{-1}} = \alpha_a^{\phantom{1}}\alpha_b^{-1}$ for all $a,b \in S_c$. We let $G_c^\alpha := \{ ab^{-1} \mid \alpha_a=\alpha_b\}$ be the subgroup of $G_c$ for which $\alpha'$ is trivial, and we write $u_g$ for the generating unitaries in $C^*(G_c^\alpha)$.

\begin{proposition}\label{prop:uniqueness: alpha faithful necessary}
Let $S$ be an admissible right LCM semigroup with $S_c$ abelian. Suppose that there exist $\beta \geq 1$ and distinct states $\phi,\phi'$ on $C^*(G_c^\alpha)$ such that there exist $ab^{-1} \in G_c^\alpha$ and a sequence $(I_k)_{k \in \N} \subset \text{Irr}(N(S))$ with $I_k$ finite and $I_k \nearrow \text{Irr}(N(S))$ satisfying
\begin{equation}\label{eq:separation for non-uniqueness}
\begin{array}{c}
\lim\limits_{k\to\infty} \zeta_{I_k}(\beta)^{-1} \sum\limits_{t \in \CT_{I_k}} N_t^{-\beta} \phi(u_{c(at)c(bt)^{-1}}) \neq \lim\limits_{k\to\infty} \zeta_{I_k}(\beta)^{-1} \sum\limits_{t \in \CT_{I_k}} N_t^{-\beta} \phi'(u_{c(at)c(bt)^{-1}}).
\end{array}
\end{equation}
Then $C^*(S)$ has at least two distinct $\kms_\beta$-states. If, for any given pair of distinct states on $C^*(G_c^\alpha)$, there exists such a sequence $(I_k)_{k \in \N}$, then there exists an affine embedding of the state space of $C^*(G_c^\alpha)$ into the $\kms_\beta$-states on $C^*(S)$. In particular, $C^*(S)$ does not have a unique $\kms_\beta$-state unless $\alpha$ is faithful under these assumptions.
\end{proposition}
\begin{proof}
Since $S_c$ is abelian, so are the groups $G_c$ and $G_c^\alpha$. In particular, $G_c$ and $G_c^\alpha$ are amenable, so that $C^*(G_c^\alpha) \subset C^*(G_c)$ as a unital subalgebra, and states are traces on $C^*(G_c)$. So if $\phi$ is a state on $C^*(G_c^\alpha)$, it extends to a state on $C^*(G_c)$, which then corresponds to a trace $\tau$ on $C^*(S_c)$. Note that we have $\tau(w_a^{\phantom{*}}w_b^*) = \phi(u_{ab^{-1}})$ for all $ab^{-1} \in G_c^\alpha$ by construction.

Now let $\beta \geq 1$, and $\phi,\phi'$ distinct states on $C^*(G_c^\alpha)$ such that there exists a sequence $(I_k)_{k \in \N} \subset \text{Irr}(N(S))$ with $I_k$ finite, $I_k \nearrow \text{Irr}(N(S))$, and \eqref{eq:separation for non-uniqueness}. Denote by $\tilde{\tau}$ and $\tau'$ the corresponding traces on $C^*(S_c)$. Following the proof of Proposition~\ref{prop:construction of KMS-states}, we obtain $\kms_\beta$-states $\psi_{\beta,\phi} := \psi_{\beta,\tilde{\tau}}$ and $\psi_{\beta,\phi'} :=\psi_{\beta,\tau'}$ by passing to subsequences $(I_{i(k)})_{k \in \N}$ and $(I_{j(k)})_{k \in \N}$ of $(I_k)_{k \in \N}$, if necessary. The state $\psi_{\beta,\phi}$ is given by
\[\begin{array}{lcl}
\psi_{\beta,\phi}(v_a^{\phantom{*}}v_b^*)
&=&\lim\limits_{k\to\infty} \zeta_{I_{i(k)}}(\beta)^{-1} \sum\limits_{t \in \CT_{I_{i(k)}}} N_t^{-\beta} \langle V_a^{\phantom{*}}V_b^*(\xi_t \otimes \xi_{\tilde{\tau}}),\xi_t \otimes \xi_{\tilde{\tau}} \rangle \vspace*{2mm}\\
&\stackrel{\ref{lem:eval gen}~(b)}{=}& \lim\limits_{k\to\infty} \zeta_{I_{i(k)}}(\beta)^{-1} \sum\limits_{t \in \CT_{I_{i(k)}}} \chi_{\text{Fix}(\alpha_a^{\phantom{1}}\alpha_b^{-1})}(t) N_t^{-\beta} \tilde{\tau}(w_{c(a\alpha_b^{-1}(t))}^{\phantom{*}}w_{c(b\alpha_b^{-1}(t))}^*) \vspace*{2mm}\\
\end{array}\]
for $a,b \in S_c$. But if $ab^{-1} \in G_c^\alpha$, then this reduces to
\[\begin{array}{lcl}
\psi_{\beta,\phi}(v_a^{\phantom{*}}v_b^*)
&=& \lim\limits_{k\to\infty} \zeta_{I_{i(k)}}(\beta)^{-1} \sum\limits_{t \in \CT_{I_{i(k)}}} N_t^{-\beta} \phi(u_{c(at)c(bt)^{-1}}) \vspace*{2mm}\\
\end{array}\]
The analogous computation applies to $\psi_{\beta,\phi'}$. Since \eqref{eq:separation for non-uniqueness} implicitly assumes existence of the displayed limits, passing to subsequences leaves the values invariant, so that $\psi_{\beta,\phi}(v_a^{\phantom{*}}v_b^*) \neq \psi_{\beta,\phi'}(v_a^{\phantom{*}}v_b^*)$ for some $ab^{-1} \in G_c^\alpha$. This shows that $\phi$ and $\phi'$ yield distinct $\kms_\beta$-states $\psi_{\beta,\phi}$ and $\psi_{\beta,\phi'}$. The map $\phi \mapsto \psi_{\beta,\phi}$ is affine, continuous, and injective by the assumption. Hence it is an embedding. Finally, the state space of $C^*(G_c^\alpha)$ is a singleton if and only if $G_c^\alpha$ is trivial, which corresponds to faithfulness of $\alpha$.
\end{proof}

\section{Open questions}\label{sec:questions}
The present work naturally leads to some questions for further research, and we would like to take the opportunity to indicate a few directions we find intriguing.

\begin{questions}\label{que:admissibility}
Can the prerequisites for Theorem~\ref{thm:KMS results-gen} be weakened? This leads to:
\begin{enumerate}[(a)]
\item Are there right LCM semigroups that are not core factorable?
\item Are there right LCM semigroups $S$ for which $S_{ci}\subset S$ is not $\cap$-closed?
\item How does a violation of \ref{cond:A4} affect the $\kms$-state structure?
\end{enumerate}
\end{questions}

\begin{question}\label{que:uniqueness vs alpha faithful}
Suppose $S$ is an admissible right LCM semigroup. Faithfulness of $\alpha\colon S_c\curvearrowright S/_\sim$ is assumed in both parts of Theorem~\ref{thm:KMS results-gen}~(2), see also Remark~\ref{rem:almost free actions}, and it is necessary under additional assumptions, see Proposition~\ref{prop:uniqueness: alpha faithful necessary}. The extra assumptions in both directions might be artifacts of the proofs, so it is natural to ask: Is faithfulness of $\alpha$ necessary or sufficient for uniqueness of $\kms_\beta$-states for $\beta$ in the critical interval?
\end{question}

If the answer to Question~\ref{que:uniqueness vs alpha faithful} is affirmative for both parts, then, for right cancellative $S_c$, the group $G_c^\alpha$ would correspond to the periodicity group $\text{Per} \Lambda$ for strongly connected finite $k$-graphs $\Lambda$, see \cite{aHLRS}.

Using self-similar actions $(G,X)$, see Proposition~\ref{prop:self-similar action}, one can give examples $S= X^*\bowtie G$ of admissible right LCM semigroups with $\beta_c=1$, $\alpha$ faithful, but not almost free, and without finite propagation. The following problem seems to be more difficult:

\begin{question}\label{que:non finite state with finite propagation}
Is there an example of a self-similar action $(G,X)$ that is not finite state such that $S= X^*\bowtie G$ has finite propagation?
\end{question}

In \cite{aHLRS}*{Section~12}, the authors showed that their findings can be obtained via \cite{Nes}*{Theorem~1.3} since their object of study admits a canonical groupoid picture.

\begin{questions}\label{que:non-A1}
As indicated in \cite{Sta3}, the boundary quotient diagram ought to admit a reasonable description in terms of $C^*$-algebras associated to groupoids. Assuming that this gap can be bridged, we arrive at the following questions:
\begin{enumerate}[(a)]
\item Does \cite{Nes}*{Theorem~1.3} apply to the groupoids related to the boundary quotient diagram for (admissible) right LCM semigroups?
\item Which groupoid properties do \ref{cond:A1}--\ref{cond:A4} correspond to?
\item What is the meaning of $\alpha\colon S_c \curvearrowright S/_\sim$ for the associated groupoids? Which groupoid properties do faithfulness and almost freeness of $\alpha$ correspond to?
\item Which groupoid property does finite propagation correspond to?
\end{enumerate}
\end{questions}

\end{document}